\documentclass[11pt, reqno]{article}
\usepackage{textcmds} 
\usepackage{amsmath, amssymb, amsfonts, amstext, verbatim, amsthm, mathrsfs}
\usepackage{microtype}
\usepackage[all]{xy}
\usepackage[modulo]{lineno}
\usepackage[dvipsnames]{xcolor}
\usepackage{aliascnt}
\usepackage{enumitem}
\usepackage{xspace}
\usepackage{amsfonts}
\usepackage{amssymb}
\usepackage{amsthm}
\usepackage{rotating}
\usepackage[margin=3.5cm]{geometry}
\usepackage{dsfont}
\usepackage{bm}
\usepackage{subfigure}
\usepackage{amsmath}
\usepackage{array}
\usepackage[all]{xy}
\usepackage{euscript}
\usepackage[T1]{fontenc}
\usepackage{mathbbol}
\usepackage{nccmath}
\usepackage{marginnote}
\usepackage{stmaryrd}
\usepackage{youngtab}
\usepackage{tikz}
\usepackage{blkarray}
\usepackage{chngcntr}
\usepackage{multicol}
\usepackage{cancel}
\usepackage{mathabx}
\usepackage{titlesec}
\usepackage{caption}
\usepackage{tikz}
\usepackage[utf8]{inputenc}  
\usetikzlibrary{calc}
\usepackage{graphics,graphicx}  
\usepackage[backend=biber,style=alphabetic, citestyle=alphabetic, isbn=false, doi=false, url=false, maxbibnames=10,sorting=nyt]{biblatex}
\usepackage[colorlinks=true,linkcolor=black,citecolor=magenta,urlcolor=blue,citebordercolor={0 0 1},urlbordercolor={0 0 1},linkbordercolor={0 0 1}]{hyperref} 

\tikzset{node distance=3cm, auto}

\makeatother

\newtheorem{thm}{Theorem}[subsection]
\newtheorem{thmlet}{Theorem}

\newtheorem{cor}[thm]{Corollary}
\newtheorem{corlet}[thmlet]{Corollary}

\newtheorem{conjecture}[thm]{Conjecture}
\newtheorem{conjlet}[thmlet]{Conjecture}
\newtheorem{lemma}[thm]{Lemma}
\newtheorem{prop}[thm]{Proposition}
\newtheorem{question}[thm]{Question}

\newtheorem{definition}[thm]{Definition}
\newtheorem{construction}[thm]{Construction}

\theoremstyle{remark}
\numberwithin{equation}{subsection} 
\numberwithin{table}{subsection}
\numberwithin{figure}{subsection}

\newtheoremstyle{customremark}
{3pt}
{3pt}
{}
{}
{\bfseries}
{.}
{.5em}
{}
\theoremstyle{customremark}
\newtheorem{rmk_no_diamond}[thm]{Remark}
\newenvironment{rmk}{\begin{rmk_no_diamond} } {\hfill$\er$ \end{rmk_no_diamond}}
\newtheorem{example_no_diamond}[thm]{Example}
\newenvironment{example}{\begin{example_no_diamond} } {\hfill$\er$ \end{example_no_diamond}}
\newtheorem{extension_no_diamond}[thm]{Extension}

\newcommand{\calX}{\mathcal{X}}

\newcommand{\acc}{\op{acc}}
\newcommand{\shear}{\mathbb{S}}

\newcommand{\rk}{{\rm rk}}

\newcommand{\NI}{{\noindent}}

\newcommand{\Nn}{{\mathcal N}}

\newcommand{\al}{{\alpha}}

\newcommand{\Om}{{\Omega}}
\newcommand{\om}{{\omega}}
\newcommand{\de}{{\delta}}

\newcommand{\ga}{{\gamma}}

\newcommand{\ka}{{\kappa}}
\newcommand{\la}{{\lambda}}
\newcommand{\La}{{\Lambda}}
\newcommand{\si}{{\sigma}}
\newcommand{\Si}{{\Sigma}}

\newcommand{\MS}{{\medskip}}

\newcommand{\er}{{\Diamond}}
\newcommand{\frakp}{\mathfrak{p}}

\newcommand{\Z}{\mathbb{Z}}

\newcommand{\F}{\mathbb{F}}
\newcommand{\R}{\mathbb{R}}
\newcommand{\Q}{\mathbb{Q}}
\newcommand{\C}{\mathbb{C}}
\newcommand{\CP}{\mathbb{CP}}
\newcommand{\D}{\mathbb{D}}

\newcommand{\eps}{\varepsilon}

\newcommand{\calL}{\mathcal{L}}
\newcommand{\calW}{\mathcal{W}}
\newcommand{\calD}{\mathcal{D}}

\newcommand{\Bb}{\mathcal{B}}
\newcommand{\bdy}{\partial}

\newcommand{\calM}{\mathcal{M}}

\newcommand{\wt}{\widetilde}

\newcommand{\ovl}{\overline}
\newcommand{\ovll}[1]{\overline{\overline{#1}}}
\newcommand{\op}[1]{{\operatorname{#1}}}

\newcommand{\std}{{\op{std}}}

\newcommand{\ind}{\op{ind}}

\newcommand{\Op}{\mathcal{O}p}
\newcommand{\nil}{\varnothing}
\newcommand{\sss}{\vspace{2.5 mm}}

\newcommand{\bb}{\frak{b}}

\newcommand{\sm}{\op{sm}}

\renewcommand{\lll}{\Langle}
\newcommand{\rrr}{\Rangle}

\newcommand{\T}{\mathcal{T}}
\newcommand{\E}{\mathbb{E}}
\newcommand{\Fib}{\op{Fib}}

\newcommand{\bl}{\op{Bl}}
\newcommand{\pd}{\op{PD}}
\newcommand{\pt}{pt}

\newcommand{\hooksymp}{\overset{s}\hookrightarrow}

\newcommand{\reg}{\op{reg}}

\newcommand{\Int}{\op{Int}\,}

\newcommand{\calB}{{\mathcal{B}}}

\renewcommand{\hom}{\op{Hom}}

\newcommand{\TT}{\mathbb{T}}

\newcommand{\calN}{\mathcal{N}}

\newcommand{\ra}{\rightarrow}
\newcommand{\Ddiv}{\mathbf{D}}

\newcommand{\calA}{\mathcal{A}}

\newcommand{\db}{\frak{p}}

\newcommand{\PD}{\op{PD}}

\newcommand{\intE}{{\mathring{E}}}
\newcommand{\intX}{{\mathring{X}}}
\newcommand{\CC}{\mathcal{C}}

\newcommand{\exc}{\mathbb{E}}
\newcommand{\uvl}{\underline}
\newcommand{\nn}{\frak{n}}
\newcommand{\vol}{\op{vol}}

\newcommand{\notdiv}{\nmid}
\newcommand{\G}{\mathbb{G}}
\newcommand{\comp}{\op{comp}}
\newcommand{\symp}{\op{symp}}

\newcommand{\roots}{{\mu}}
\newcommand{\spec}{\op{Spec}}
\definecolor{darkmagenta}{rgb}{0.55, 0.0, 0.55}
\newcommand{\hl}[1] {{\boldmath\textbf{{\color{darkmagenta}#1}}}}
\newcommand{\lan}{\langle}
\newcommand{\ran}{\rangle}
\newcommand{\fan}{\Sigma}

\newcommand{\pp}{\frakp}
\newcommand{\vv}{\frak{v}}
\newcommand{\mut}{\op{Mut}}
\newcommand{\atf}{\mathbb{A}}
\newcommand{\afflen}{{\op{Len}_\op{aff}}}
\newcommand{\prim}{\op{prim}}
\newcommand{\tordeg}{{\boldsymbol{\delta}}}
\newcommand{\vecv}{{\vec{v}}}
\newcommand{\vk}{{\varkappa}}
\newcommand{\vecfrako}{\vec{\mathfrak{o}}}
\newcommand{\frako}{\mathfrak{o}}
\newcommand{\vecb}{{\vec{b}}}

\newcommand{\slant}{\op{slant}}

\newcommand{\MM}{{\mathbb{M}}}
\newcommand{\NN}{\mathbb{N}}
\newcommand{\ww}{\mathfrak{w}}

\renewcommand{\qq}{\mathfrak{q}}

\newcommand{\mon}{\mathit{mon}}

\newcommand{\gl}{\op{GL}}
\newcommand{\eig}{\mathbb{L}}
\newcommand{\aut}{\op{Aut}}
\newcommand{\nodal}{\op{nodal}}

\newcommand{\vecn}{{\vec{n}}}
\renewcommand{\vert}{\op{vert}}
\newcommand{\LL}{\mathbb{L}}
\newcommand{\aur}{\op{Aur}}
\newcommand{\height}{\mathbb{h}}
\newcommand{\Jint}{J_{\op{int}}}
\newcommand{\full}{\op{full}}
\newcommand{\uu}{{\mathfrak{u}}}

\newcommand{\penc}{\mathbb{p}}

\newcommand{\intC}{\mathring{C}}
\newcommand{\tri}{\op{tri}}
\newcommand{\ttil}{\widecheck}
\newcommand{\perf}{\op{Perf}}
\newcommand{\calH}{\mathcal{H}}
\newcommand{\perfbar}{\uvl{\perf}}
\newcommand{\empha}[1]{{\em #1}}
\renewcommand{\setminus}{\smallsetminus}
\newcommand{\calS}{\mathcal{S}}

\newcommand{\vs}{\varsigma}

\newcommand{\calU}{\mathcal{U}}

\makeatletter
\newcommand{\dashover}[2][\mathop]{#1{\mathpalette\df@over{{\dashfill}{#2}}}}
\newcommand{\fillover}[2][\mathop]{#1{\mathpalette\df@over{{\solidfill}{#2}}}}
\newcommand{\df@over}[2]{\df@@over#1#2}
\newcommand\df@@over[3]{%
  \vbox{
    \offinterlineskip
    \ialign{##\cr
      #2{#1}\cr
      \noalign{\kern1pt}
      $\m@th#1#3$\cr
    }
  }%
}
\newcommand{\dashfill}[1]{%
  \kern-.5pt
  \xleaders\hbox{\kern.5pt\vrule height.4pt width \dash@width{#1}\kern.5pt}\hfill
  \kern-.5pt
}
\newcommand{\dash@width}[1]{%
  \ifx#1\displaystyle
    2pt
  \else
    \ifx#1\textstyle
      1.5pt
    \else
      \ifx#1\scriptstyle
        1.25pt
      \else
        \ifx#1\scriptscriptstyle
          1pt
        \fi
      \fi
    \fi
  \fi
}
\newcommand{\solidfill}[1]{\leaders\hrule\hfill}

\newcommand{\oset}[3][0ex]{%
  \mathrel{\mathop{#3}\limits^{
    \vbox to#1{\kern-2\ex@
    \hbox{$\scriptstyle#2$}\vss}}}}

\newcounter{countitems}
\newcounter{nextitemizecount}
\newcommand{\setupcountitems}{%
  \stepcounter{nextitemizecount}%
  \setcounter{countitems}{0}%
  \preto\item{\stepcounter{countitems}}%
}
\makeatletter
\newcommand{\computecountitems}{%
  \edef\@currentlabel{\number\c@countitems}%
  \label{countitems@\number\numexpr\value{nextitemizecount}-1\relax}%
}
\newcommand{\nextitemizecount}{%
  \getrefnumber{countitems@\number\c@nextitemizecount}%
}
\newcommand{\previtemizecount}{%
  \getrefnumber{countitems@\number\numexpr\value{nextitemizecount}-1\relax}%
}
\makeatother    
\newenvironment{AutoMultiColItemize}{%
\ifnumcomp{\nextitemizecount}{>}{3}{\begin{multicols}{2}}{}%
\setupcountitems\begin{itemize}}%
{\end{itemize}%
\unskip\computecountitems\ifnumcomp{\previtemizecount}{>}{3}{\end{multicols}}{}}

\makeatother

\title{Singular algebraic curves and infinite symplectic staircases}

\author{Dusa McDuff and Kyler Siegel\thanks{K.S. is partially supported by NSF grant DMS-2105578}}

\date{\today}

\begin{document}

\maketitle

\begin{abstract}
We show that the infinite staircases which arise in the ellipsoid embedding functions of rigid del Pezzo surfaces (with their monotone symplectic forms) can be entirely explained in terms of rational sesquicuspidal symplectic curves. Moreover, we show that these curves can all be realized algebraically, giving various new families of algebraic curves with one cusp singularity. Our main techniques are (i) a generalized Orevkov twist, and (ii) the interplay between algebraic $\Q$-Gorenstein smoothings and symplectic almost toric fibrations. Along the way we develop various methods for constructing singular algebraic (and 
hence symplectic) curves which may be of independent interest.
\end{abstract}

\tableofcontents

\section{Introduction}

\subsection{Brief summary}

 One starting point for this paper is the observation that the numerics of the following two mathematical objects coincide:
 \begin{enumerate}[label=(\alph*)]
   \item the family of unicuspidal rational plane curves constructed by Orevkov in \cite{orevkov2002rational} (see also \cite{kashiwara_hiroko,fernandez2006classification})
   \item the outer corners of the steps of the Fibonacci staircase for the symplectic ellipsoid embedding function $c_{\CP^2}(x)$ of the complex projective plane (see e.g. \cite{McDuff-Schlenk_embedding}).
 \end{enumerate} 
A priori these belong to rather distinct subfields: the former pertains to the classical problem of characterizing 
algebraic plane curves of given degree and genus with prescribed singularities (see e.g. \cite{greuel2018singular}), while the latter belongs to the burgeoning area of quantitative symplectic embeddings (see e.g. \cite{Schlenk_old_and_new}).
In \cite{cusps_and_ellipsoids} we showed that symplectic unicuspidal curves give (stable) symplectic embedding obstructions, and in particular that the family (a) recovers the Fibonacci staircase outer corners (b).
In this paper:
\begin{itemize}
   \item We show that the infinite staircases for rigid del Pezzo surfaces\footnote{That is, either the blowup of $\CP^2$ up to four times or the product $\CP^1\times \CP^1$, with the monotone symplectic form; see Remark~\ref{rmk:rigid_dP}.}  found in \cite{cristofaro2020infinite} can be \empha{entirely understood} in terms of genus zero sesquicuspidal symplectic curves. Here the obstructions at outer corners come from index zero curves, while the embeddings at inner corners come from higher index curves (via a version of symplectic inflation in \S\ref{sec:inflate}). As a byproduct, all of these staircases stabilize.
   \item We show that all of these curves can be realized \empha{algebraically}. In particular, this gives new families of unicuspidal algebraic curves whose existence is suggested by (and has applications to) quantitative symplectic geometry. As an application, we give a new classification theorem for algebraic unicuspidal rational curves in the first Hirzebruch surface.
 \end{itemize}

The core of this paper develops new techniques for constructing algebraic (and
hence symplectic) unicuspidal curves. First, in \S\ref{sec:twist} we give a generalization of Orevkov's twist from \cite{orevkov2002rational} which holds in any rigid del Pezzo surface. We apply this to construct algebraic curves for each of the relevant outer corners, and later in \S\ref{subsec:ghost_stairs} to produce a new sequence of algebraic plane curves responsible for the stabilized ghost stairs from \cite{Ghost}.

Then, in \S\ref{sec:ATF1} we give a perspective on $\Q$-Gorenstein smoothings of singular toric surfaces which closely parallels the theory of symplectic almost toric fibrations. Using this we establish general constructions of algebraic unicuspidal rational curves in \S\ref{sec:singI} and \S\ref{sec:singII}. These take as input tropical curves in a base polygon $Q$ and reflect a rich combinatorial theory of polygon mutations.
This approach naturally produces algebraic curves for both the inner and outer corners 
of the rigid del Pezzo infinite staircases, as well as more general curve families.

The remainder of this extended introduction is structured as follows. 
In \S\ref{subsec:context_motiv} we first provide some context and motivation for the study of unicuspidal algebraic curves, as well as symplectic ellipsoid embeddings and infinite staircases.
Then in \S\ref{subsec:main_results} we give precise formulations of our main results.

\subsection{Context and motivation}\label{subsec:context_motiv}

\subsubsection{Singular plane curves}\label{subsubsec:sing_curves}

To set the stage, let us first recall a few basics about singular algebraic curves. In this paper all algebraic curves will be defined over the complex numbers. 
By ``plane curve'' we mean a complex algebraic curve in $\CP^2$, which concretely is of the form $V(F) := \{F(x,y,z) = 0\}$ for some homogeneous polynomial $F(x,y,z)$.
A point $p_0 = [x_0:y_0:z_0] \in V(F)$ is singular if and only if we have $\bdy_x F(p_0) = \bdy_y F(p_0) = \bdy_z F(p_0) = 0$. 
The following local topological models, written in affine coordinates with the singular point at the origin, will be relevant for us:
\begin{itemize}
  \item $\{x^2 = y^2\}$ is the \hl{ordinary double point} (a.k.a. the $A_1$ singularity) 
  \item $\{x^3 = y^2\}$ is the \hl{ordinary cusp}
 \item more generally, $\{x^p = y^q\}$ for $p,q \in \Z_{\geq 1}$ coprime is the \hl{$(p,q)$ cusp}. 
  \end{itemize} 
Note that topologically the $(p,q)$ cusp is the cone over the $(p,q)$ torus knot, and if $p=1$ or $q=1$ this is just a smooth point.\footnote{Unless stated otherwise, by default we consider curve singularities up to topological (as opposed to analytic) equivalence, i.e. up to local homeomorphism of pairs. In particular, any $(p,q)$-cusp with $p>q$ has an analytic parametrization of the form $t \mapsto (t^q,t^p + {\it h.o.t.})$ for some choice of local coordinates near the singular point. We will say that the cusp is \hl{analytically standard} if the higher order terms can be removed, i.e. if it takes the form $\{z_1^p = z_2^q\}$ for some choice of local coordinates $z_1,z_2$.}

\begin{example}\label{ex:two_cusps}
  The plane curve $C = \{X^p = Y^qZ^{p-q}\} \subset \CP^2$ has two singularities: a $(p,q)$ cusp at $[0:0:1]$ and a $(p,p-q)$ cusp at $[0:1:0]$.
  Moreover, it is rational since it admits a parametrization $\CP^1 \ra C$, $[s:t] \mapsto [s^qt^{p-q}:s^p:t^p]$.
\end{example}

A (reduced and irreducible) algebraic curve\footnote
{The curves considered in this paper will be rational, that is parametrizable by $\CP^1$ (and in particular irreducible), unless explicit mention is made to the contrary. } 
is called 
\hl{$(p,q)$-unicuspidal}
if it has a single $(p,q)$ cusp  (with $\gcd(p,q)=1$) and no other singularities. More generally, it is called 
\hl{$(p,q)$-sesquicuspidal} 
if in addition it has some ordinary double points.

To anchor the discussion, let us recall the following classification result.
We denote the Fibonacci numbers by $\Fib_1 = 1, \Fib_2 = 1, \Fib_3 = 2$ and so on.
\begin{thm}[\cite{fernandez2006classification}]\label{thm:bob_et_al}
There exists a $(p,q)$-unicuspidal
rational plane curve of degree $d$ with $p > q \ge 2$ if and only if $(d,p,q)$ is one of the following: 
\begin{enumerate}[label=(\alph*)]
  \item $(p,q) = (d,d-1)$ for $d \geq 3$
  \item $(p,q) = (2d-1,d/2)$ for $d \geq 4$ even
  \item $(p,q) = (\Fib_{k+2}^2,\Fib_{k}^2)$ and $d = \Fib_{k+2}\Fib_k$ for $k \geq 3$ odd
  \item $(p,q) = (\Fib_{k+4},\Fib_{k})$ for $d = \Fib_{k+2}$ for $k \geq 3$ odd
  \item $(p,q) = (22,3)$ and $d = 8$
  \item $(p,q) = (43,6)$ for $d = 16$.
\end{enumerate}
\end{thm} 
Family (a) is the specialization of Example~\ref{ex:two_cusps} with $(p,q) = (d,d-1)$, while family (b) is given by $\{(zy-x^2)^{d/2} = xy^{d-1}\} \subset \CP^2$.
See Example~\ref{ex:sporadic} below for constructions of (e) and (f).
The curves in family (c) are more complicated but are described by explicit equations in \cite[\S5]{fernandez2006classification}, following \cite{kashiwara_hiroko}.
Family (d) corresponds to the aforementioned Orevkov curves \cite{orevkov2002rational}.

To make further sense of Theorem~\ref{thm:bob_et_al}, it will be helpful to introduce the following:
\begin{definition}
The (real) \hl{ index} of a $(p,q)$-sesquicuspidal
rational curve $C$ in a complex surface or symplectic four-manifold 
$M$ is 
\begin{align}
\ind_\R(C) := 2c_1(A) - 2p - 2q,
\end{align}
where $A \in H_2(M)$ denotes the homology class of $C$ and $c_1(A)$ is its first Chern number.
\end{definition}

\NI As explained in detail in \cite{cusps_and_ellipsoids}, the index corresponds to the expected (real) dimension of the space of rational curves in homology class $A$ with a $(p,q)$ cusp and satisfying a maximal order tangency constraint at the cusp. For instance, for an ordinary $(3,2)$ cusp there is a well-defined complex tangent line at the singular point, and 
the constraint corresponds to specifying both the location of the singularity and its tangent line at that point. For a general $(p,q)$ cusp the constraint also involves higher jet constraints. Equivalently, the index is the (real) Fredholm index of the normal crossing resolution (c.f. \cite[\S4.1]{cusps_and_ellipsoids} or \S\ref{subsec:toric_p_q} below).

In particular, for a $(p,q)$-unicuspidal rational plane curve $C$ of degree $d$ we have $\ind_\R(C) = 6d -2p-2q$, and the indices for the curves in Theorem~\ref{thm:bob_et_al} are as follows:
\begin{equation}
 \begin{array}{|c|c|c|c|c|c|c|}
 \hline
& (a) & (b) & (c) & (d) & (e) & (f) \\ \hline
\text{index} & 2d+2 & d+2 & 2 & 0 & -2 & -2 \\\hline
 \end{array}.
\end{equation}
For index zero curves such as those in family (d) one expects to get a finite count, and indeed these are encoded by the Gromov--Witten-type invariants $N_{\CP^2,d[L]}\lll \CC^{(p,q)}\pt \rrr$ defined in \cite[\S3]{cusps_and_ellipsoids}.
The curves in family (c) cannot quite be counted (they occur in complex $1$-parameter families), but they naturally degenerate to those in family (d) (c.f. \cite[\S5]{fernandez2006classification}, based on \cite{kashiwara_hiroko,Miyanishi-Sugie}).
Meanwhile, the sporadic cases (e) and (f) have negative index, so they should disappear for a generic almost complex structure.
In this article, families (c) and (d) (and their generalizations) will be the most significant, as they precisely correspond to the inner and outer corners respectively of the Fibonacci staircase.
Incidentally, in \S\ref{subsec:unicusp_appl} we exploit this connection with symplectic geometry to give an alternative proof that the list in part (d) of Theorem\ref{thm:bob_et_al} above is complete, and we extend 
this classification to the first Hirzebruch surface (for which the corresponding list is much more complicated).

\subsubsection{Symplectic embeddings and infinite staircases}\label{subsubsec:symp_emb}

Let us now briefly recall some notions surrounding symplectic ellipsoid embeddings and infinite staircases.
Given a four-dimensional symplectic manifold $X$, its ellipsoid embedding function is defined by 
\begin{align}
c_X(x) := \inf \{\la \in \R_{> 0}\;|\; E(\tfrac{1}{\la},\tfrac{x}{\la}) \hooksymp X\}.
\end{align}
Here the infimum is over all $\la \in \R_{>0}$ for which there exists a symplectic embedding of the ellipsoid $$
E(\tfrac{1}{\la},\tfrac{x}{\la}) := \{(z_1,z_2) \;|\; \pi |z_1|^2 \la + \pi |z_2|^2 \la/x \leq 1\} \subset \C^2
$$ (endowed with the restriction of the standard symplectic form) into $X$.
In \cite{McDuff-Schlenk_embedding}, the ellipsoid embedding function for the four-ball $B^4(1) = E(1,1)$ was explicitly worked out.
In particular, the portion for $1 \leq x \leq \tau^4 := \tfrac{3\sqrt{5}+7}{2}$ is a piecewise linear function whose graph is a zigzag, that alternately slopes up and is horizontal, with infinitely many nonsmooth points that
  accumulate at $\tau^4$ and have coordinates given by ratios of odd index Fibonacci numbers -- see \cite[Fig 1.1]{McDuff-Schlenk_embedding}. 
Subsequently, similar infinite staircases were discovered for other target spaces such as $B^2(1) \times B^2(1)$ \cite{Frenkel-Muller}, $E(1,3/2)$ \cite{cristofaro2020ehrhart}, and more (see e.g. \cite{usher2019infinite}).
More recently, the authors of \cite{cristofaro2020infinite} gave a unified description of infinite staircases for the six rigid del Pezzo surfaces with their monotone symplectic forms, namely $\CP^2(3) \#^{\times k} \ovl{\CP}^2(1)$ for $k=0,1,2,3,4$ and $\CP^1(2) \times \CP^1(2)$.\footnote{Here $\CP^2(a)$ is endowed with the Fubini--Study form normalized so that a line has area $a$. The qualifier ``rigid'' is a slight misnomer since it refers to the complex rather than symplectic structure -- see Remark~\ref{rmk:rigid_dP} below.}

\begin{thm}[\cite{cristofaro2020infinite}]\label{thm:CG_et_all_staircases} 
For each rigid del Pezzo surface $M$, the ellipsoid embedding function $c_M(x)$ has an infinite staircase with explicitly described accumulation point and step coordinates.
\end{thm}
\NI 
By elementary scaling and monotonicity considerations, establishing these infinite staircases boils down to (a) obstructing symplectic embeddings at the outer corners and (b) constructing symplectic embeddings at the inner corners.
Embeddings corresponding to the inner corners were constructed in \cite{cristofaro2020infinite,casals2022full}
using almost toric fibrations and their mutations (see e.g. \cite{symington71four,evans2023lectures} or \S\ref{subsec:atfs_and_polys} below), and hence are also related to generalized Markov equations and exotic Lagrangian tori as in \cite{vianna2017infinitely}.
Meanwhile, obstructions corresponding to outer corners were established in \cite{cristofaro2020infinite} using embedded contact homology (ECH) capacities (see e.g. \cite{Hlect}).
In this paper our approach to the inner corners and one of our approaches to the outer corners are also based on almost toric fibrations, but used in a quite distinctive way through the lens of sesquicuspidal curves.

\begin{rmk}\label{rmk:rigid_dP}
By definition a del Pezzo surface is a smooth complex projective surface with ample anticanonical bundle. 
Up to diffeomorphism these are $\CP^2 \#^{\times k}\ovl{\CP}^2$ for $k = 0,\dots,8$ and $\CP^1 \times \CP^1$. Up to biholomorphism there is a unique del Pezzo surface having smooth type $\CP^2 \#^{\times k} \ovl{\CP}^2$ for $k=0,\dots,4$ or $\CP^1 \times \CP^1$ (these are the rigid ones), while the remaining cases appear in nontrivial moduli spaces;
see \cite[\S8]{dolgachev2012classical}

Each del Pezzo surface admits a unique monotone symplectic form up to symplectomorphism and scaling (see e.g. \cite{salamon2013uniqueness}), and  
unless explicit mention is made to the contrary we work with the monotone symplectic structure normalized to have monotonicity constant $1$, i.e. $c_1(M) = [\omega_M] \in H^2(M;\R)$ (e.g. $\CP^2(3)$). 
One should keep in mind that the moduli spaces of complex and symplectic structures on these smooth manifolds are quite distinct,\footnote{Roughly speaking, in the complex category the locations of blowup points matter, while in the symplectic category the sizes of blowups matters.} but it should be clear from the  context whether we view $M$ in the complex, symplectic, or smooth category. 
\end{rmk}

\begin{rmk}\label{rmk:dp_versus_ctd}
  The treatment in \cite{cristofaro2020infinite} emphasizes the $12$ convex toric domains $X_1,\dots,X_{12}$ pictured in Figure~\ref{fig:rational_staircase_targets} below, which includes $B^4(1),B^2(1) \times B^2(1)$, and $E(1,3/2)$ as special cases. It is shown in \cite{cristofaro2020infinite} that the ellipsoid embedding function for each $X_i$ is 
  (up to scaling) equal 
  to the ellipsoid embedding function for one of the monotone rigid del Pezzo surfaces, namely the one with the same negative weight expansion (for instance we have $c_{B^4(a)}(x) = c_{\CP^2(a)}(x)$).
Thus for simplicity of exposition we will mostly restrict our discussion to the closed target spaces (except when discussing the stable folding curve in \S\ref{subsec:folding_curve}).
\end{rmk}

\subsection{Main results}\label{subsec:main_results}

\subsubsection{Singular curves and symplectic embeddings}

We first explain how singular symplectic curves can be used both to  obstruct and to construct symplectic ellipsoid embeddings. 
For $p,q \in \Z_{\geq 1}$  coprime, a 
\hl{$(p,q)$-sesquicuspidal symplectic curve}
 in a symplectic four-manifold $M^4$ is a subset $C \subset M$ which has one point $x_0 \in C$ locally modeled on a $(p,q)$ cusp point of an algebraic curve in $\C^2$, and such that $C$ is otherwise an immersed symplectic submanifold with only positive double points (see \cite[Def. 3.5.1]{cusps_and_ellipsoids}).

The following explicit link between sesquicuspidal curves and symplectic embedding obstructions was established in \cite{cusps_and_ellipsoids}:

\begin{thm}[{Theorems A(b), D, and E in \cite{cusps_and_ellipsoids}}]\label{thm:stab_obs_from_curve}
Let $(M^4,\omega_M)$ be a four-dimensional closed symplectic manifold, and suppose there exists an index zero $(p,q)$-sesquicuspidal rational symplectic curve in $M$ in homology class $A \in H_2(M)$.
Then any symplectic embedding $E(cq,cp) \hooksymp M$ must satisfy $c \le \frac{[\omega_M] \cdot A}{pq}$. 
Moreover, the same is true for any symplectic embedding $E(cq,cp) \times \C^N \hooksymp M \times \C^N$ for $N \in \Z_{\geq 1}$, provided that $M \times \C^N$ is semipositive.\footnote{Here semipositivity is a technical condition which allows one to rule out sphere bubbling using only classical perturbations. Note that $M \times \C^N$ is automatically semipositive if $N=1$ or if $M$ is monotone.
Using e.g. \cite[Cor. 2.7.2]{cusps_and_ellipsoids}, we can also quantify the above stable symplectic embedding obstructions by replacing the domain $E(1,a) \times \C^N$ with $E(1,a,b_1,\dots,b_N)$ for suitable finite $b_1,\dots,b_N \in \R_{>0}$. A similar remark applies to all other stable obstructions which follow, although for simplicity we will formulate results without this quantification.} 
\end{thm}

\NI In other words, the existence of $C$ implies $c_M(p/q) \geq \tfrac{p}{[\omega_M] \cdot [C]}$.\footnote
{
The basic reason for the existence of this obstruction is that one can construct an SFT-type curve in the complement of (a slight perturbation of) the ellipsoid that must have positive symplectic area. Equivalently, the exceptional divisor given by the normal crossing resolution of the singular curve must have positive area.}
Since any unicuspidal algebraic curve in a complex projective surface is in particular a unicuspidal symplectic curve,  applying Theorem~\ref{thm:stab_obs_from_curve} to family (d) in Theorem~\ref{thm:bob_et_al} (with $N = 0$) immediately gives the obstructive part (i.e. outer corners) of the Fibonacci staircase in $c_{\CP^2}(x)$. 
Moreover, the case $N \geq 1$ shows that these obstructions stabilize, i.e. 
we have $c_{\CP^2 \times \C^N}(x) = c_{\CP^2}(x)$ for all $1 \leq x \leq \tau^4$ (this is the main result of \cite{CGH}, originally proved using embedded contact homology), 
where we put
\begin{align}
c_{X \times \C^N}(a) := \inf \{\la \in \R_{> 0}\;|\; E(\tfrac{1}{\la},\tfrac{a}{\la}) \times \C^N \hooksymp X \times \C^N\}
\end{align}
for any symplectic four-manifold $X^4$ and $N \in \Z_{\geq 1}$.

\sss

As for constructing symplectic embeddings, the following theorem is proved in \S\ref{sec:inflate} below via the method of symplectic inflation.
Recall that any local branch of a holomorphic curve near a singularity is homeomorphic to the cone over an iterated torus knot (see \cite{eisenbud1985three}).
The cabling parameters can be read off from the Puiseux pairs $(n_1,d_1),\dots,(n_g,d_g)$, which can in turn be read off from a Puiseux series parametrization $x(t) = t^m, y(t) = \sum\limits_{k=m}^\infty a_k t^k$ -- see \S\ref{subsec:res_mult_Puis} for more details.
In particular, a $(p,q)$ cusp corresponds to a single Puiseux pair $(n_1,d_1) = (p,q)$, and our other main examples will be cusps with two Puiseux pairs $(n_1,d_1) = (p,q)$, $(n_2,d_2) = (kp+1,k)$ for some $k \in \Z_{\geq 1}$.

\begin{thmlet}\label{thmlet:inflation_from_sescusp} 
\begin{enumerate}[label=(\roman*)] 
Let $(M^4,\omega_M)$ be a four-dimensional closed symplectic manifold.

\hfill
  \item Let $C$ be a $(p,q)$-sesquicuspidal symplectic curve in $M$ whose homology class satisfies $[C] = c \,\PD[\om_M] \in H^2(M;\R)$ for some $c \in \R_{>0}$ and $[C] \cdot [C] \geq pq$.
Then there exists a symplectic embedding
$E(\tfrac{q}{c'},\tfrac{p}{c'}) \hooksymp M$ for any $c' > c$. 
In particular, if $[C] \cdot [C] = pq$ then this is a full filling, i.e the domain and target have arbitrarily close volume.\footnote{Note that if $M$ is a symplectic blowup of $\CP^2$ (and more generally) this actually implies the existence of a symplectic embedding of the open ellipsoid $\intE(\tfrac{q}{c},\tfrac{p}{c}) \hooksymp M$ that fills the entire volume of $M$ (c.f. \cite[Proof of Prop. 1.5]{cristofaro2019symplectic}).
}

  \item More generally, let $C$ be a sesquicuspidal symplectic curve in $M$ with Puiseux pairs $(p,q),(p_2,q_2),\dots,(p_g,q_g)$, whose homology class satisfies $[C] = c \PD[\omega_M] \in H^2(M;\R)$ for some $c \in \R_{>0}$ and $[C] \cdot [C] \geq k^2 pq$ 
   with $k = q_2\cdots q_g$.
  Then there exists a symplectic embedding $E(\tfrac{kq}{c'},\tfrac{kp}{c'}) \hooksymp M$ for any $c' > c$.
\end{enumerate}
\end{thmlet}

\NI In particular the existence of $C$ in (i) implies $c_M(p/q) \leq \tfrac{c}{q}$ (assuming $p > q$).
Note that the last sentence of (i) follows since we have
\begin{align*}
\vol(M,\omega_M) = \tfrac{1}{2}\int_M \omega_M \wedge \omega_M = \tfrac{1}{2} \pd [\omega_M] \cdot \pd [\omega_M] = \tfrac{1}{2c^2}[C] \cdot [C].
\end{align*}
The conditions on $[C]\cdot[C]$ imply that the index of the (partial) resolution of $C$ along which we inflate is positive. Indeed, the expression $[C] \cdot [C] - pq$ in (i) corresponds to the self-intersection number of the normal crossing resolution of $C$, while, when $g=2$, the expression $[C] \cdot [C] - k^2 pq$ in (ii) corresponds to the self-intersection number of the minimal resolution of $C$ (see \S\ref{sec:inflate}).

Applying Theorem~\ref{thmlet:inflation_from_sescusp} to family (c) from Theorem~\ref{thm:bob_et_al} recovers the constructive part (i.e. inner corners) for $c_{\CP^2}(x)$. 
We will see below that a similar picture holds for all of the monotone rigid del Pezzo surfaces.

\subsubsection{Outer and inner corner curves} 

We first construct singular algebraic curves responsible for the obstructions at outer corners, generalizing family (d) from Theorem~\ref{thm:bob_et_al}.

\begin{thmlet}\label{thmlet:outer_corner_curves}
In each rigid del Pezzo surface $M$ there is a countable family of rational index zero unicuspidal algebraic\footnote{\label{footnote:Chow}Note that by Chow's theorem we may speak interchangeably about ``algebraic'' and ``holomorphic'' curves, although in the body of the paper we work mostly in the holomorphic category.} curves which correspond precisely to the outer corners of the steps of the infinite staircase in $c_{M}(x)$. 
More specifically, if $(x,y)$ is an outer corner point on the graph of $c_M$, then the corresponding $(p,q)$-unicuspidal curve $C$ in $M$ satisfies $p/q = x$ and $\tfrac{p}{[\omega_M]\cdot [C]} = y$.
\end{thmlet}

\begin{corlet}
Each of the rigid del Pezzo infinite staircases stabilizes, i.e. for each monotone rigid del Pezzo surface $M$ we have $c_{M \times \C^N}(x) = c_{M}(x)$ for all 
$1 \leq x \leq a_\acc(M)$, where $a_\acc(M)$ denotes the accumulation point of the infinite staircase in $c_{M}$. 
\end{corlet}

Our first proof of Theorem B in \S\ref{sec:twist}
is based on a generalization of Orevkov's birational transformation $\CP^2 \dashrightarrow \CP^2$ to a birational transformation $\Phi_{M}: M \dashrightarrow M$ for each rigid del Pezzo surface $M$. In brief, we start with two or three ``seed curves'' in $M$, and then we iteratively apply $\Phi_{M}$ to produce the rest of the family. The key point is that for seed curves which are ``well-placed'' (see Definition~\ref{def:well-placed}), successive applications of $\Phi_M$ lead to curves with increasingly singular cusps.

\begin{rmk}
 The number of ``strands'' of the infinite staircase is determined by the number of initial seed curves, which is three for $\CP^2(3) \# \ovl{\CP}^2(1)$ and $\CP^2(3) \#^{\times 2} \ovl{\CP}^2(1)$ and two in the remaining cases (this corresponds to $J$ in Table~\ref{table:CG_et_al}). 
 This number can also be seen in terms of the almost toric structures supported by the symplectic manifold $M$ (i.e. triangles or quadrilaterals, see Figure~\ref{fig:smooth_Fano_polygons}).
\end{rmk}

\sss

We also give a different construction of these outer corner curves in \S\ref{sec:singII} based on almost toric fibrations and $\Q$-Gorenstein deformations. 

\begin{thmlet}\label{thmlet:nodal} 
Let $\pi: \atf \ra Q$ be an almost toric fibration, where $Q \subset \R^2$ is a polygon\footnote{All polygons in this paper are assumed to be convex.} and $\atf$ is a (not necessarily monotone) closed symplectic four-manifold which is diffeomorphic to a rigid del Pezzo surface $M$.
Suppose that $Q$ contains consecutive edges pointing in the directions $(-mr^2,mra-1),(0,-1),(1,0)$ for some $m,r,a \in \Z_{\geq 1}$ with $\gcd(r,a) = 1$.\footnote{Equivalently, $Q$ has a vertex $\vv$ with edge directions $(1,0),(0,1)$ and a vertex on the edge in direction $(0,1)$ with eigenray  in the direction $(r,-a)$ -- see Figure~\ref{fig:ATF_3_pics} or \S\ref{sec:ATF1} for more details.}
Then $M$ contains an index zero $(r,a)$-unicuspidal rational algebraic curve. 
\end{thmlet}
\NI 
The rough idea is  first to construct unicuspidal symplectic curves in $\atf$ which are ``visible'' in $Q$ (and lie over a line segment connecting a vertex to a base-node as in Figure~\ref{fig:ATF_3_pics} right), and then to compare $\atf$ with a $\Q$-Gorenstein smoothing of the corresponding singular toric surface $V_Q$ in order to upgrade these to algebraic curves.
Theorem~\ref{thmlet:nodal} is a corollary of Theorem~\ref{thm:uni_alg_curve}, which holds beyond rigid del Pezzo surfaces and which we state using the purely combinatorial language introduced in \S\ref{sec:ATF1}. 
A useful consequence of Theorem~\ref{thmlet:nodal} is that we can directly observe ``visible'' obstructions for ellipsoid embeddings into $\atf$ in terms of triangles in the polygon $Q$ (see the shading in Figure~\ref{fig:ATF_3_pics} right and see \S\ref{subsec:singII_intro} for more details). 
 Another noteworthy feature of Theorem~\ref{thmlet:nodal} is that $\atf$ need not be monotone (or equivalently the polygon $Q$  need not be dual Fano in the sense of \S\ref{subsubsec:dual_Fano}). This potentially allows us to construct a much larger class of unicuspidal curves than we could just by looking at monotone ATFs, as we illustrate with Theorem~\ref{thmlet:F_1_classif} below.

\begin{rmk}
One reason for giving two different proofs of Theorem~\ref{thmlet:outer_corner_curves} is that the underlying techniques naturally extend in different directions. For instance, the generalized Orevkov twist is used in \S\ref{subsec:ghost_stairs} to construct sesquicuspidal algebraic curves in $\CP^2$ (the ``ghost stairs'' curves) which we do not currently know how to see using almost toric fibrations.
\end{rmk}

\MS

We now discuss curves responsible for constructing symplectic embeddings, generalizing family (c) in Theorem~\ref{thm:bob_et_al}.
We consider the two-stranded and three-stranded cases separately, as they behave somewhat differently, with the latter requiring more complicated cusp singularities.

\begin{thmlet}\label{thmlet:inner_corner_curves} \hfill

\begin{enumerate}[label=(\alph*)]
  \item For $M$ each of the rigid del Pezzo surfaces $\CP^2, \CP^1 \times \CP^1,$ and $ \CP^2 \#^{\times j} \ovl{\CP}^2,$ $ j = 3,4$,
  there is a countable family of index two unicuspidal rational algebraic curves in $M$ which correspond precisely to the inner corners of the infinite staircase in $c_M(x)$. More specifically, if $(x,y)$ is an inner corner point on the graph of $c_M$, then the corresponding $(p,q)$-unicuspidal curve $C$ in $M$ satisfies $[C] \cdot [C] = pq$ and $[C] = c\, \pd([\omega_M])$ for $c = \tfrac{pq}{p+q+1}$, with $p/q = x$ and $c/q = y$. 
  \item For $M$ each of the rigid del Pezzo surfaces $\CP^2 \# \ovl{\CP}^2$ and $\CP^2 \#^{\times 2} \ovl{\CP}^2$, there is a countable family of weakly\footnote
  {
  The word ``weakly''  indicates that these curves might have other singularities in addition to the distinguished cusp and some ordinary double points. In the symplectic category we can always perturb such a curve to a genuine sesquicuspidal one, but this is not guaranteed in the algebraic category. It seems plausible that this qualifier here could be removed by a more careful analysis of the curves in the proof of Proposition~\ref{prop:curve_in_rat_surf}.} sesquicuspidal 
  rational algebraic curves in $M$ which correspond precisely to the inner corners of the infinite staircase in $c_M(x)$. More specifically, if $(x=p/q,y)$ is an inner corner point on the graph of $c_M$, then the corresponding curve $C$ in $M$ satisfies $[C] = kqy \,\pd([\omega_M])$ and $[C] \cdot [C] \geq k^2pq$ for some $k \in \Z_{\geq 1}$ and has a cusp with one Puiseux pair $(p,q)$ if $k = 1$ and two Puiseux pairs $(p,q), (kp+1,k)$ if $k \in \Z_{\geq 2}$.
\end{enumerate}
\end{thmlet}

Our proof of Theorem~\ref{thmlet:inner_corner_curves} is based on an analogue of Theorem~\ref{thmlet:nodal}, namely Theorem~\ref{thm:sesqui_main}, which says roughly that every ``visible ellipsoid embedding'' (in the sense of \S\ref{subsubsec:vis_ell}) corresponds to an algebraic (weakly) sesquicuspidal rational curve.
The proof of Theorem~\ref{thm:sesqui_main} is similarly based on almost toric fibrations and $\Q$-Gorenstein smoothings, except that the relevant curves no longer correspond to straight line segments in the polygon $Q$ but rather to tropical curves which intersect every edge.

\begin{rmk}
In the above we have observed that we can recover embeddings corresponding to inner corners of infinite staircases by inflating along suitable singular symplectic curves, and moreover these curves can be explicitly constructed with the aid of almost toric fibrations.
However, it should be emphasized that the very same almost toric fibrations can be used much more directly to construct symplectic ellipsoid embeddings (c.f. Proposition~\ref{prop:vis_ell_emb}), which is the approach taken in \cite{cristofaro2020infinite,casals2022full}. 
The main novelty of Theorem~\ref{thmlet:inner_corner_curves} is the connection with singular algebraic (and in particular symplectic) curves.
\end{rmk}

\subsubsection{Unicuspidal curves in the first Hirzebruch surface}

Let $F_1 = \CP^2 \# \ovl{\CP}^2$ denote the first Hirzebruch surface, and let $\ell,e \in H_2(F_1)$ denote the line and exceptional classes.
We showed in \cite[Thm. G]{cusps_and_ellipsoids} that there exists an index zero $(p,q)$-unicuspidal rational symplectic curve in $F_1$ in homology class $A = d\ell-me$ if and only if $A \in H_2(F_1)$ is a $(p,q)$-perfect exceptional class (see \cite[Def. 4.4.2]{cusps_and_ellipsoids}).\footnote{Note that this does not depend on the choice of symplectic form on $F_1$, by e.g.  
Theorem E and Theorem 3.3.2 in \cite{cusps_and_ellipsoids}, along with the fact that all symplectic forms on $F_1$ are deformation equivalent.}
The set 
\begin{align*}
\perf(F_1) := \{(p,q,d,m)\;|\; A = d\ell - me \text{ is a } \text{$(p,q)$-perfect exceptional class}\}
\end{align*} 
is quite complicated but was recently worked out explicitly in \cite{magillmcd2021,magill2022staircase} in the course of classifying infinite staircases for all (not necessarily monotone) symplectic forms on $F_1$ (see also \S\ref{subsec:unicusp_appl} below for a detailed overview). Using results in \cite{magill2022unobstructed}, we
show in  \S\ref{subsec:unicusp_appl} that a polygon as in Theorem~\ref{thmlet:nodal} may be associated to each $(p,q,d,m) \in \perf(F_1)$ with $p/q>6$, and thereby construct an {\em algebraic} $(p,q)$-unicuspidal curve.  Hence we have:

\begin{thmlet}\label{thmlet:F_1_classif}
For any coprime positive integers $p > q$ and homology class $A \in H_2(F_1)$, the following are equivalent:
\begin{itemize}
  \item there exists an index zero $(p,q)$-unicuspidal rational {\em symplectic} curve $C$ in $F_1$ with $[C] = A$
  \item there exists an index zero $(p,q)$-unicuspidal rational {\em algebraic} curve $C$ in $F_1$ with $[C] = A$.
\end{itemize}
\end{thmlet}

\begin{rmk}
It follows by Theorem~\ref{thmlet:F_1_classif} and the results in \cite{magillmcd2021,magill2022staircase} that every rational unicuspidal algebraic curve with one Puiseux pair in $F_1$ corresponds to the outer corner of a staircase in $c_{\calH_b}(x)$ for some $b \in [0,1)$, where $\calH_b := \CP^2(1) \# \ovl{\CP}^2(b)$ denotes $F_1$ with the symplectic form such that a line has area $1$ and the exceptional divisors has area $b$ (this is unique up to symplectomorphism). Note that $\calH_b$ is monotone only if $b=1/3$. 
\end{rmk}

We prove Theorem~\ref{thmlet:F_1_classif} by showing that  such a $(p,q)$-unicuspidal curve  is realized by an almost toric fibration which satisfies the hypotheses of Theorem~\ref{thmlet:nodal}. The analogous statement holds for $\CP^2$ (see Lemma~\ref{lem:d_is_complete}) and is expected for $\CP^1 \times \CP^1$ (c.f. \cite{usher2019infinite,farley2022four} and Remark~\ref{rmk:rel_to_CP1xC^1} below).
Thus it is natural to ask whether something analogous holds also for the remaining rigid del Pezzo surfaces:

\begin{conjlet}\label{conjlet:remaining_four}
Let $M$ be a rigid del Pezzo surface. The following are equivalent:
\begin{enumerate}[label=(\alph*)]
  \item  There exists an index zero $(p,q)$-unicuspidal rational symplectic curve in $M$
  \item There exists an almost toric fibration $\pi: \atf \ra Q$ as in Theorem~\ref{thmlet:nodal}, where $\atf$ is diffeomorphic to $M$ and $Q$ has consecutive edges pointing in the directions $(-mq^2,mpq-1),(0,-1),(1,0)$ for some $m \in \Z_{\geq 1}$. In particular, there exists an index zero $(p,q)$-unicuspidal rational algebraic curve in $M$.
\end{enumerate}
\end{conjlet} 

\NI Note that (a) is equivalent to the existence of a $(p,q)$-perfect exceptional homology class $A \in H_2(M)$ (again by \cite[Thm. G]{cusps_and_ellipsoids}), although the set $\perf(M)$ remains to be worked out (and presumably becomes more complicated with more blowups).

\begin{rmk}\label{rmk:nonmonotone_inner}
We do not attempt to construct inner corner curves for staircases in nonmonotone manifolds such as $\calH_b$ for $b \neq 1/3$. One main issue is that in all known cases such staircases exist only when 
the symplectic form has no rational multiple, so that one cannot 
construct optimal ellipsoid embeddings by simply inflating along a curve in a class Poincar\'e dual to a multiple of the symplectic form. 
\end{rmk}

\subsubsection{Sesquicuspidal curves and obstructions beyond staircases}

Another natural question in the spirit of Theorem~\ref{thm:bob_et_al} is to try to classify index zero $(p,q)$-sesquicuspidal rational curves in $\CP^2$. Note that for such a curve we have $d = (p+q)/3$, and by the adjunction formula the number of ordinary double points must be $\tfrac{1}{2}(d-1)(d-2) - \tfrac{1}{2}(p-1)(q-1)$. 
This question is known to be closely related to the study of the stabilized ellipsoid embedding function $c_{\CP^2 \times \C^N}(x)$ beyond the Fibonacci staircase, i.e. for $x > \tau^4$ (see e.g. \cite[\S1]{cusps_and_ellipsoids} and the references therein). In \S\ref{sec:stab} we apply the Orevkov twist $\Phi_{\CP^2}$ to more interesting seed curves to produce a new (to our knowledge) family of rational algebraic plane curves:

\begin{thmlet}\label{thmlet:deg_3_seed}
There is an infinite sequence of index zero rational sesquicuspidal algebraic curves $C_1,C_2,C_3,\dots$ in $\CP^2$ which correspond precisely to the ``ghost stairs'' obstructions from \cite{McDuff-Schlenk_embedding}. 
More specifically, for $k \in \Z_{\geq 1}$, $C_k$ has degree $d_k$ and a $(p_k,q_k)$ cusp and one ordinary double point, where $(p_k,q_k) = (\Fib_{4k+2},\Fib_{4k-2})$ and $d_k = \tfrac{1}{3}(p_k+q_k) = \Fib_{4k}$. 
\end{thmlet} 
\NI Combined with Theorem~\ref{thm:stab_obs_from_curve}, this recovers the main result from \cite{Ghost}, namely there is an infinite sequence $x_1 > x_2 > x_3 > \cdots$ of positive real numbers with $\lim\limits_{i \ra \infty}x_i = \tau^4$ and such that $c_{B^4(1) \times \C^{N}}(x_i) = \tfrac{3x_i}{x_i+1}$ for all $i,N \in \Z_{\geq 1}$.  
We recall the significance of the ``folding function'' $\tfrac{3x}{x+1}$ and discuss its (partly conjectural) analogue for other target spaces in \S\ref{subsec:folding_curve}.
We also note that the same techniques allow for a vast generalization of Theorem~\ref{thmlet:deg_3_seed} conditional on the existence of ``higher degree seed curves'', which we take up in the forthcoming work \cite{sesqui} (see Remark~\ref{rmk:higher_deg_seeds}).

\section*{Acknowledgements} We would like to thank Jonny Evans, Bob Friedman, and Rob Lazarsfeld for helpful discussions, and J\'anos Koll\'ar for many illuminating comments on the first version of this paper.

\section{Unicuspidal curves and the generalized Orevkov twist}\label{sec:twist} 

Our goal in this section is to introduce the generalized Orevkov twist and use it to prove Theorem~\ref{thmlet:outer_corner_curves}. 
We first formalize the twist in $\CP^2$ in \S\ref{subsec:orev_orig}, with a view towards allowing more general seed curves (e.g. those considered in \S\ref{subsec:folding_curve}), and then in \S\ref{subsec:gen_twist} we extend the ambient space to rigid del Pezzo surfaces. 
In \S\ref{subsec:two_twists} we slightly modify the construction in \S\ref{subsec:gen_twist} to produce another twist with quite different numerical properties.
After a brief interlude in \S\ref{subsec:staircase_numerics} to recall some staircase numerics from \cite{cristofaro2020infinite}, we complete the proof by constructing the relevant seed curves in \S\ref{subsec:staircase_numerics}.

\subsection{The Orevkov twist in $\CP^2$}\label{subsec:orev_orig}

In this subsection we recall the definition and basic properties of the birational transformation $\Phi_{\CP^2}: \CP^2 \dashrightarrow \CP^2$ from \cite[\S6]{orevkov2002rational}.\footnote{As pointed out to us by J. Koll\'ar, this transformation was 
described earlier by Yoshihara in \cite[Lem.4.22]{Yoshihara} in connection with his study
of the automorphism group of complements to projective plane curves. However, the usage in \cite{orevkov2002rational} to construct the Fibonacci unicuspidal curves in Theorem~\ref{thm:bob_et_al}(d) is a more direct precursor to our generalization in this paper.} 
Let $\calN \subset \CP^2$ be a fixed nodal cubic, which for concreteness we can take to be 
$\calN_0 := \{X^3 + Y^3 = XYZ\}$ 
(any other nodal cubic is projectively equivalent to this one). 
Let $\db \in \calN$ denote the double point, and let $\calB_+,\calB_-$ denote the two local branches near $\db$.

\begin{construction}\label{constr:Orev_twist}
 The birational transformation $\Phi_{\CP^2}: \CP^2 \dashrightarrow \CP^2$ is defined as follows. Let $\bl^1\CP^2$ denote the blowup\footnote{All blowups in this section are at points and occur in the complex category. In later sections we also consider symplectic blowups, which depend on a symplectic embedding of a ball.}
of $\CP^2$ at the point $\db$, with resulting exceptional divisor $\F_1^1$. Let $\calN^1 \subset \bl^1\CP^2$ denote the proper transform of $\calN$, and let $\calB^1_+ \subset \calN^1$ denote the proper transform of the local branch $\calB_+$. Put $\db^1 := \calB_+^1 \cap \F_1^1 \subset \calN^1$.

 Next, let $\bl^2\CP^2$ denote the blowup of $\bl^1\CP^2$ at the point $\db^1$, with resulting exceptional divisor $\F_2^2$, and with $\F_1^2 \subset \bl^2\CP^2$ the proper transform of $\F_1^1$. Let $\calN^2 \subset \bl^2\CP^2$ denote the proper transform of $\calN^1$, and let $\calB^2_+ \subset \calN^2$ denote the proper transform of the local branch $\calB_+^1$. Put $\db^2 := \calB_+^2 \cap \F_2^2 \subset \calN^2$.

Continuing in this manner, after a total of $7$ blowups we arrive at $\bl^7\CP^2$, which contains a chain of rational curves $\
\calN^7,\F_1^7,\dots,\F_7^7$ with intersection graph as in Figure~\ref{fig:Orev_int_graph}.
In terms of the natural identification $H_2(\bl^7\CP^2) \cong H_2(\CP^2) \oplus \Z\langle e_1,\dots,e_7\rangle$, we have
  \begin{AutoMultiColItemize}
  \item $[\calN^7] = 3\ell - 2e_1 - e_2 - \cdots - e_7$
  \item $[\F_1^7] = e_1-e_2$
  \item $[\F_2^7] = e_2-e_3$
  \item $[\F_3^7] = e_3-e_4$
  \item $[\F_4^7] = e_4-e_5$
  \item $[\F_5^7] = e_5-e_6$
  \item $[\F_6^7] = e_6-e_7$
  \item $[\F_7^7] = e_7$,
  \end{AutoMultiColItemize}where $\ell \in H_2(\CP^2)$ denotes the line class.

\begin{figure}[h]\caption{The chain of rational curves (decorated by their self-intersection numbers) which arise in the half part of the Orevkov twist.}
\label{fig:Orev_int_graph}
\centering
\includegraphics[width=.8\textwidth]{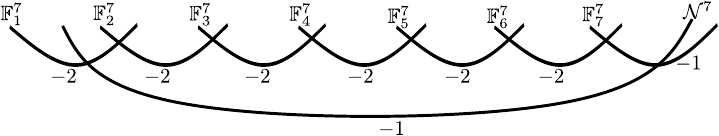}
\end{figure}

Since $[\calN^7] \cdot [\calN^7] = -1$, we can blow down $\bl^7\CP^2$ along $\calN^7$ to obtain $\bl^{7;1}\CP^2$. Let $\F_1^{7;1},\dots,\F_7^{7;1} \subset \bl^{7;1}\CP^2$ denote the proper transforms of $\F_1^7,\dots,\F_7^7$ respectively.
Then since $[\F_1^{7;1}] \cdot [\F_1^{7;1}] = -1$, we can blow down $\bl^{7;1}\CP^2$ along $\F_1^{7;1}$ to obtain $\bl^{7;2}\CP^2$. Let $\F_2^{7;2},\dots,\F_7^{7;2} \subset \bl^{7;2}\CP^2$ denote the proper transforms of $\F_2^{7;1},\dots,\F_7^{7;1}$ respectively.
Continuing in this manner, after a total of $7$ blowdowns (along $\calN^7, \F_1^{7;1},\dots,F_6^{7;6}$), we arrive at $\bl^{7;7}\CP^2$, which contains the nodal rational curve $\F_7^{7;7}$ with $[\F_7^{7;7}] \cdot [\F_7^{7;7}] = 9$. 

Finally, by composing this sequence of $7$ blowups and $7$ blowdowns with a biholomorphism $\bl^{7;7}\CP^2 \cong \CP^2$ sending $\F_7^{7;7}$ to $\calN$,
we arrive at the birational transformation $\Phi_{\CP^2}: \CP^2 \dashrightarrow \CP^2$. 
\end{construction}

\begin{rmk}
Note that the biholomorphism $\bl^{7;7}\CP^2 \cong \CP^2$ is not unique. In fact, the automorphism group of the pair $(\CP^2,\calN_0)$ is isomorphic to the symmetric group $S_3$, generated by $[X:Y:Z] \mapsto [Y:X:Z]$ and $[X:Y:Z] \mapsto [\mu X: \mu^2 Y:Z]$ for $\mu^3 = 1$, where the former swaps the branches $\calB_+,\calB_-$ -- see \cite{Kol}.
For our purposes of constructing curves with prescribed cusps this choice will not play any material role, and thus we will often suppress it from the discussion (and similarly in \S\ref{subsec:gen_twist} below -- see Remark~\ref{rmk:no_nice_biholom}).
\end{rmk}

Given a curve $C \subset \CP^2$, we denote its proper transform under the above birational transformation by $\Phi_{\CP^2}(C) \subset \CP^2$. 
In the following, we say that a curve $C$ satisfies the constraint $\lll \T^{(m)}_{\calB_-}\db\rrr$ if $C$ passes through $\db$ and has a branch with contact order at least $m$ (i.e. tangency order at least $m-1$) to $\calB_-$.
Note that a curve satisfying $\lll \T^{(m)}_{\calB_-}\db\rrr$ must have intersection multiplicity at least $m+1$ with $\calN$, since it also intersects the branch $\calB_+$ at $\db$.

\begin{thm}[\cite{orevkov2002rational}]\label{thm:orev_orig}
Let $L \subset \CP^2$ denote the unique line which satisfies $\lll \T^{(2)}_{\calB_-} \db\rrr$, and put $C_{2k+ 1} := \Phi_{\CP^2}^k(L)$ for $k \in \Z_{\geq 0}$.
Similarly, let $Q \subset \CP^2$ denote the unique (irreducible) conic\footnote{As an alternative to the conic $Q$ we could take the unique line in $\CP^2$ which satisfies $\lll \T_{\calB_+}^{(2)}\db\rrr$, as this transforms into $Q$ under one application of the Orevkov twist $\Phi_{\CP^2}$ (this is actually the approach taken in \cite{orevkov2002rational}).}
 which satisfies $\lll \T^{(5)}_{\calB_-} \db\rrr$, 
and put $C_{2k+2} := \Phi_{\CP^2}^k(Q)$ for $k \in \Z_{\geq 0}$.
Then for $k \in \Z_{\geq 1}$, $C_k$ is a $(p_k,q_k)$-unicuspidal rational plane curve of degree $d_k$, where $(p_k,q_k) = (\Fib_{2k+1},\Fib_{2k-3})$ and $d_k = \tfrac{1}{3}(p_k+q_k) = \Fib_{2k-1}$. 
\end{thm}

\NI Here $\Fib_k$ is the $k$th Fibonacci number, i.e. $\Fib_1 = \Fib_2 = 1$ and $\Fib_{k+2} = \Fib_k + \Fib_{k+1}$ (it will also be convenient to put $\Fib_{-1} := 1$). Recall that the $x$-coordinate of the $k$th outer corner point of the Fibonacci staircase in \cite{McDuff-Schlenk_embedding} is precisely the odd index Fibonacci ratio $\Fib_{2k+1}/\Fib_{2k-3}$.

\begin{figure}[h]\caption{The Orevkov twist in $\CP^2$. Notice that $C^{7;1}$ is tangent to $\F^{7;1}_1$ because its blowup $C^7$ intersects $\F^7_1$. However, because the curve $C^{7;1}$ intersects  $\F^{7;1}_1$ at its intersection with $\F^{7;1}_7$ rather than  $\F^{7;1}_2$,  when $\F^{7;1}_1$ is blown down  $C^{7;2}$ is tangent to  $\F^{7;2}_7$ rather than to $\F^{7;2}_2$, and its blowdowns remain tangent to $\F^{7;k}_7$ for $k >2$.}
\centering
\includegraphics[width=1\textwidth]{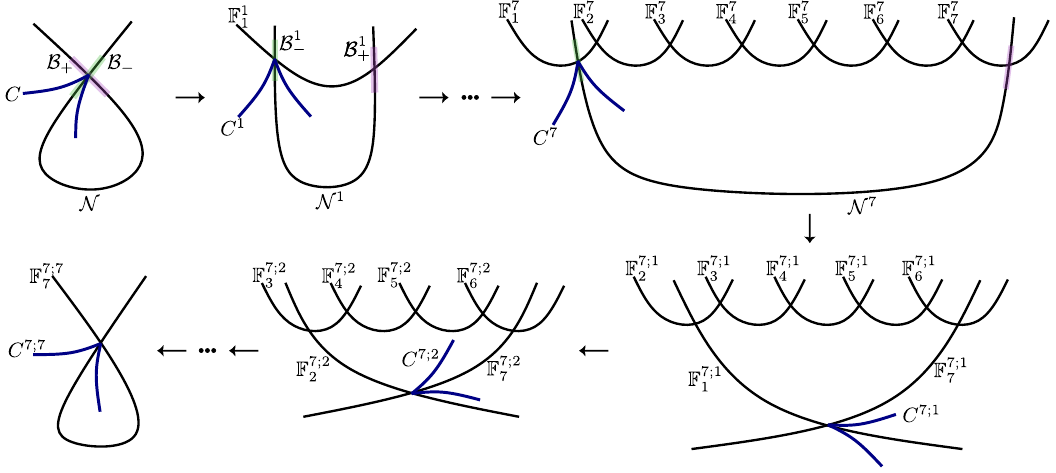}
\label{fig:Orev_curve_transf}
\end{figure}

Theorem~\ref{thm:orev_orig} follows from the identity $\Fib_{2k+5} = 7\Fib_{2k+1}-\Fib_{2k-3}$, together with the fact that for a curve $C \subset \CP^2$ with a well-placed $(p,q)$ cusp (see Definition~\ref{def:well-placed} below), its twist $\Phi_{\CP^2}(C)$ has a well-placed $(7p-q,p)$ cusp.
Indeed, let us analyze the construction of $\Phi_{\CP^2}(C)$ in more detail as follows (see Figure~\ref{fig:Orev_curve_transf}).
For simplicity assume $p > 2q$, and let $C \subset \CP^2$ be a curve which has a $(p,q)$ cusp maximally tangent to the branch $\calB_-$ of the nodal cubic $\calN$ at its double point $\db$. 
The proper transform $C^1 \subset \bl^1\CP^2$ of $C$ then has a $(p-q,q)$ cusp maximally tangent to the branch $\calB_-^1$ of $\calN^1$. 
After $6$ further blowups (which do not affect  $C^1$ since it is disjoint from the blowup points),
we arrive at the curve $C^7 \subset \bl^7\CP^2$, which has a $(p-q,q)$ cusp maximally tangent to the branch $\calB_-^7$ of $\calN^7$. 
We then blow down to obtain the curve $C^{7;1} \subset \bl^{7;1}\CP^2$, which has a $(p,p-q)$ cusp maximally tangent to $\F_1^{7;1}$. In the next blowdown we obtain $C^{7;2} \subset \bl^{7;2}\CP^2$, which has a $(2p-q,p)$ cusp maximally tangent to $\F_7^{7;2}$. Finally, after $5$ further blowdowns we arrive at $\Phi_{\CP^2}(C) = C^{7;7} \subset \bl^{7;7}\CP^2$, which has a $(7p-q,p)$ cusp maximally tangent to a branch of the nodal curve $\F_7^{7;7}$ at its double point.

\begin{example}\label{ex:sporadic}
Suppose that $C$ is a degree $d$ plane curve which intersects the nodal cubic $\calN$ at some point distinct from the double point $\db$ with contact order $3d$.
Then one can check that the twist $\Phi_{\CP^2}(C)$ has a $(21d+1,3d)$ cusp.
In particular, in the case $d=1$, $C$ is a flex line, and $\Phi_{\CP^2}(C)$ has degree $8$ by the adjunction formula, so $\Phi_{\CP^2}(L)$ is a sporadic $(22,3)$-unicuspidal curve as in Theorem~\ref{thm:bob_et_al}(e).
Similarly, when $d=2$, $\Phi_{\CP^2}(C)$ has a $(43,6)$ cusp and degree $16$, and hence represents Theorem~\ref{thm:bob_et_al}(f).
\end{example}

\subsection{The generalized Orevkov twist}\label{subsec:gen_twist}

We now generalize the Orevkov twist to any rigid complex del Pezzo surface $M$, i.e. either $\bl^k\CP^2$ for  $k = 0,1,2,3,4$ or $\CP^1 \times \CP^1$.
Note that $\pd(c_1(M)) \in H_2(M)$ is given by $3\ell - e_1 - \cdots - e_k$ if $M = \bl^k\CP^2$, or by $2\ell_1 + 2\ell_2$ in the case $M = \CP^1 \times \CP^1$ (here we put $\ell_1 := [\CP^1 \times \{\pt\}]$ and $\ell_2 := [\{\pt\} \times \CP^1]$).

In the following we consider $\calN \subset M$ to be a rational nodal curve which is anticanonical (i.e. $[\calN] = \pd(c_1(M))$) and has a unique double point (note that this is consistent with the adjunction formula).
Concretely, in the case $M = \bl^k \CP^2$ for $k = 0,\dots,4$ we can assume up to biholomorphism that $M$ is given by blowing up $\CP^2$ at $k$ points $\nn_1,\dots,\nn_k$ on the standard nodal cubic 
$\calN_0 := \{X^3+Y^3=XYZ\}$ that are away from the node,
 and then we 
 can take $\calN$ to be the proper transform of $\calN_0$ in $M$. (Note that any tuple of $4$ points in $\CP^2$, no $3$ of which are collinear, is projectively equivalent to any other such tuple.)
 In the case $M = \CP^1 \times \CP^1$, we 
 can take $\calN$ to be $\left(\CP^1 \times \{q_1,q_2\}\right) \cup \left(\{p_1,p_2\} \times \CP^1\right)$ for $q_1,q_2,p_1,p_2 \in \CP^1$ with $q_1 \neq q_2$ and $p_1 \neq p_2$, after smoothing the nodes at $(p_1,q_2),(p_2,q_1),(p_2,q_2)$.

The degree of the del Pezzo surface $M$ is by definition $[\calN] \cdot [\calN]$, and it will also be convenient to put $K := [\calN] \cdot [\calN] - 2$ (see Table~\ref{table:CG_et_al}).

\begin{construction}\label{constr:gen_Orev_twist}
For $M$ a rigid del Pezzo surface as above, the birational transformation $\Phi_{M}: M \dashrightarrow M$ is defined as follows. 
Let $\calN \subset M$ be a rational
nodal
anticanonical curve, with local branches $\calB_+,\calB_-$ near the unique double point $\db$.
Let $\bl^1M$ denote the blowup of $M$ at the point $\db$, with resulting exceptional divisor $\F_1^1$. Let $\calN^1 \subset \bl^1M$ denote the proper transform 
of $\calN$, and let $\calB^1_+ \subset \calN^1$ denote the proper transform of the local branch $\calB_+$. Put $\db^1 := \calB_+^1 \cap \F_1^1 \subset \calN^1$.

 Next, let $\bl^2M$ denote the blowup of $\bl^1M$ at the point $\db^1$, with resulting exceptional divisor $\F_2^2$, and with $\F_1^2 \subset \bl^2M$ the proper transform of $\F_1^1$. Let $\calN^2 \subset \bl^2M$ denote the proper transform of $\calN^1$, and let $\calB^2_+ \subset \calN^2$ denote the proper transform of the local branch $\calB_+^1$. Put $\db^2 := \calB_+^2 \cap \F_2^2 \subset \calN^2$.

Continuing in this manner, after a total of $K$ blowups we arrive at $\bl^K M$, which contains a chain of rational curves $\
\calN^K,\F_1^K,\dots,\F_K^K$.
Since $[\calN^K] \cdot [\calN^K] = -1$, we can blow down $\bl^K M$ along $\calN^K$ to obtain $\bl^{K;1}M$. Let $\F_1^{K;1},\dots,\F_K^{K;1} \subset \bl^{K;1}M$ denote the  
images of the  curves $\F_1^K,\dots,\F_K^K$ under this blowdown.
Then since $[\F_1^{K;1}] \cdot [\F_1^{K;1}] = -1$, we can blow down $\bl^{K;1}M$ along $\F_1^{K;1}$ to obtain $\bl^{K;2}M$. Let $\F_2^{K;2},\dots,\F_K^{K;2} \subset \bl^{K;2}M$ denote the corresponding blowdown images. 
Continuing in this manner, after a total of $K$ blowdowns (along $\calN^K, \F_1^{K;1},\dots,F_{K-1}^{K;K-1}$), we arrive at $\bl^{K;K}M$, which contains the rational nodal curve $\F_K^{K;K}$ with $[\F_K^{K;K}] \cdot [\F_K^{K;K}] = K+2$. 

Finally, by composing this sequence of $K$ blowups and $K$ blowdowns with a biholomorphism $\Xi: \bl^{K;K}M \cong M$ (which exists by Lemma~\ref{lem:gen_Orev_gives_back_M} below), we arrive at the birational transformation $\Phi_{M} : M \dashrightarrow M$, whose image  contains the rational nodal curve $\calN' := \Xi(\F_K^{K;K})$.
\end{construction}
\NI We will sometimes denote the generalized Orevkov twist $\Phi_M$ by $\Phi_{M;\calN}$ when we wish to emphasize the role of the anticanonical curve $\calN$.

We begin with some lemmas to justify Construction~\ref{constr:gen_Orev_twist}.
\begin{lemma}\label{lem:N'_antican}
In the setting of Construction~\ref{constr:gen_Orev_twist}, $\calN'$ is also an anticanonical curve in $M$, i.e. we have $[\calN'] = \pd(c_1(M)) \in H_2(M)$.
\end{lemma}
\begin{proof}
It suffices to check that $\F_K^{K;K}$ is an anticanonical curve in $\bl^{K;K}M$.
To see this, observe that if $C$ is a (reduced but not necessarily irreducible) anticanonical curve in a smooth complex surface $X$, and if $X'$ denotes the blowup of $X$ at an ordinary double point of $C$, then the total transform of $C$ in $X'$ is again anticanonical. Conversely, if $C$ is an anticanonical curve in a smooth complex surface $Y'$ with an exceptional component $C_0 \subset C$ such that $C_0$ intersects $\ovl{C \setminus C_0}$ transversely in two points, and if $Y$ denotes the blowdown of $Y'$ along $C_0$, then the image of $C$ under the blowdown map $Y' \ra Y$ is again anticanonical.
Since $\F_K^{K;K} \subset \bl^{K;K}M$ is obtained from the anticanonical curve $\calN \subset M$ by a sequence of blowups and blowdowns of these forms, it follows that $\F_K^{K;K}$ is again anticanonical.
\end{proof}
\begin{lemma}\label{lem:gen_Orev_gives_back_M}
In the setting of Construction~\ref{constr:gen_Orev_twist}, there is a biholomorphism $\bl^{K;K}M \cong M$.
\end{lemma}
\begin{proof}
We first claim that $\bl^{K;K}M$ is Fano.
By Lemma~\ref{lem:N'_antican}, $\F^{K;K}_K \subset \bl^{K;K}M$ is anticanonical, so it suffices to check that this is ample, and this follows easily by the Nakai--Moishezon criterion, since any curve in $\bl^{K;K}M$ disjoint from $\F^{K;K}_K$ would imply a curve in $M$ disjoint from $\calN$.

Since $\bl^{K;K}M$ has the same integral homology as $M$, by the classification of del Pezzo surfaces the only possible ambiguity lies in distinguishing $\CP^1 \times \CP^1$ from the first Hirzebruch surface $\bl^1 \CP^2$.
In the case $M = \bl^1 \CP^2$, let $\exc \subset M$ be the exceptional curve, and let $\exc^{6;6}$ be its proper transform in $\bl^{6;6}M$.
One can check (e.g. using an analysis similar to the proof of Lemma~\ref{lem:N'_antican}) that we have $c_1([\exc^{6;6}]) = 7$, and hence $\bl^{6;6}M \cong \bl^1\CP^2$, as all homology classes in $\CP^1 \times \CP^1$ have even Chern number.

A similar but slightly more complicated  argument shows  directly that if $M= 
\CP^1 \times \CP^1$ then we again must have $\bl^{6;6}M \cong M$. However, one can also argue using symmetry: the inverse of the Orevkov twist is given by exactly the same number of blowups and blowdowns but with the two  branches of $\Nn$ reversed. 
The above argument concerning $\bl^1 \CP^2$ shows that any  twist on $\bl^1 \CP^2$ gives $\bl^1 \CP^2$. Therefore,  if the  twist done to $\CP^1 \times \CP^1$ did give $\bl^1 \CP^2$, then the inverse operation done on the other branch would have also to give $\bl^1 \CP^2$, which is impossible.
\end{proof}

\sss

The following language will be useful for understanding how cusps transform under successive applications of the generalized Orevkov twist $\Phi_M$.

\begin{definition}\label{def:well-placed}
A curve $C$ in $M$ is \hl{$(p,q)$-well-placed} with respect to $\calN$ if we have $C \cap \calN = \{\db\}$, $C$ is locally irreducible near $\db$, and we have $(C \cdot \calB_-)_{\db} = p$ and $(C \cdot \calB_+)_{\db} = q$.
\end{definition}

\NI Here $\calN$ is any rational 
nodal 
anticanonical curve as in Construction~\ref{constr:gen_Orev_twist}, and we will sometimes simply say that $C$ is ``well-placed'' if $\calN$ is clear from the context.
Note that if  $\gcd(p,q)=1$ this implies that $C$ has a $(p,q)$ cusp at $\db$ which is maximally tangent to the branch $\calB_-$ (in the sense of  \cite[\S3.5]{cusps_and_ellipsoids}). 
We also allow the case $q = 1$, i.e. $C$ is $(p,1)$-well-placed if it has a single branch passing through $\db$ which is smooth and strictly satisfies the tangency condition $\lll \T_{\calB_-}^{(p)}\db\rrr$, and $C$ is otherwise disjoint from $\calN$.
Note that the line $L$ (resp. conic $Q$) in Theorem~\ref{thm:orev_orig} is $(2,1)$-well-placed (resp. $(5,1)$-well-placed) with respect to $\calN$.

\begin{rmk}\label{rmk:well-placed_index_zero}
Note that in Definition~\ref{def:well-placed} the total homological intersection number $[C] \cdot [\calN]$ is the sum of the local intersection numbers 
$(C \cdot \calB_-)_{\db}$ and $(C \cdot \calB_+)_{\db}$ corresponding to the two branches of the unique intersection point of $C$ with $\calN$.
This means that for any (rational) curve $C \subset M$ which is $(p,q$)-well-placed with respect to $\calN$ we have $c_1([C]) = [C] \cdot [\calN] = p+q$, i.e. $C$ must have index zero. 
Note also that the singularities of $C \setminus \db$ are in bijective correspondence with the singularities of $\Phi_M(C) \setminus \db$. In particular, if $C$ is unicuspidal (with $p,q > 1$) then so is $\Phi_M(C)$. 
\end{rmk}
\begin{rmk}
One can check that the curves in Example~\ref{ex:sporadic} with a $(21d+1,3d)$ cusp are $(21d,3d)$-well-placed, which does not contradict the above discussion since $\gcd(21d,3d) = 3d > 1$.
\end{rmk}

The singularity analysis given at the end of \S\ref{subsec:orev_orig} (and depicted in Figure~\ref{fig:Orev_curve_transf}) immediately extends to the generalized Orevkov twist as follows.

\begin{prop}\label{prop:gen_orev_numerics}
If a curve $C \subset M$ is $(p,q)$-well-placed with respect to $\calN$ with $p > q$, then $\Phi_{M}(C)$ is $(Kp-q,p)$-well-placed with respect to $\calN'$. 
\end{prop}
\NI Crucially, since $\calN'$ is itself a rational nodal anticanonical curve by Lemma~\ref{lem:N'_antican}, we can subsequently apply the generalized Orevkov twist using $\calN'$ to obtain a curve $\Phi_{M;\calN'}(\Phi_{M;\calN}(C))$ which is $(K[Kp-q]-p,Kp-q)$-well-placed with respect to $\calN''$, and so on.

\begin{rmk}\label{rmk:no_nice_biholom}
In the case of $\Phi_{\CP^2}$ there is a biholomorphism taking $\calN'$ to $\calN$, but this is not a priori clear (or needed) in the general case.
\end{rmk}

\subsection{The two twists}\label{subsec:two_twists}

It turns out to be quite fruitful to consider\footnote{This subsection is based on discussions with J. Koll\'ar after the first draft of this paper appeared. For a more thorough discussion of these ideas see the recently posted \cite{Kol}.}
the effect of the generalized Orevkov twist $\Phi_M$ on $(p,q)$-well-placed curves with $p < q$.   Equivalently, one can keep $p>q$ but interchange the roles of the branches $\Bb_-, \Bb_+$ in the definition of well-placed, thus assuming that the initial curve is $p$-fold tangent to $\Bb_+$ and $q$-fold tangent to $\Bb_-$.
Then, provided that $p>Kq$ when we blow up $K$-times along $\Bb_+$ we obtain a curve that is $(p-Kq)$-fold tangent to $\Bb_+^K$, so that after blowing down as before we get a curve with a cusp of type $(q + K(p-Kq), p-Kq)$.
Thus, we have:
\begin{prop}\label{prop:two_twists}
Given a $(p,q)$-well-placed curve $C$ in a rigid del Pezzo surface $M$, there exists 
\begin{enumerate}[label=(\alph*)]
   \item a $(Kp-q,p)$-well-placed curve $\Phi_M(C)$ in $M$ provided that $p/q > 1$, and 
   \item a $(q+K(p-Kq),p-Kq)$-well-placed curve $\Psi_M(C)$ in $M$ provided that $p/q > K$.
 \end{enumerate}
\end{prop}

Noting that $\frac{q + K(p-Kq)}{p-Kq} = \frac{1 + K(p/q- K)}{p/q-K}$, let $R: (K,\infty) \ra (K,\infty)$ denote the function 
\begin{align}
R(x) := \frac{1+K(x-K)}{x-K},
\end{align}
which is an involution which fixes $K+1$ and exchanges $(K,K+1)$ with $(K+1,\infty)$.
This function will play an important role in the case of the first Hirzebruch surface $M = F_1$ in \S\ref{subsec:unicusp_appl} below.

\begin{rmk}
Koll\'ar has observed that the generalized Orevkov twist $\Phi_M$ can be understood naturally in terms of Geiser involutions $\si_+,\si_-$, which are automorphisms of $M \setminus \calN$ dating back to \cite{geiser} (see e.g. \cite[\S8.7]{dolgachev2012classical}).
Roughly speaking, the construction of $\si_+$ (resp. $\si_-$) proceeds as in Construction~\ref{constr:gen_Orev_twist} by blowing up $K$ times at the branch $\calB_+$ (resp $\calB_-$), and then it appeals to the fact that the resulting weak del Pezzo surface has degree two and hence carries a canonical involution (see e.g. \cite[\S8]{dolgachev2012classical}). 
For more details we defer to \cite{Kol}.
\end{rmk}

\subsection{Staircase numerics and seed curves}\label{subsec:staircase_numerics}

Before completing the proof of Theorem~\ref{thmlet:outer_corner_curves}, we briefly recall some numerical aspects of the rational infinite staircases. 
Our discussion is largely informed by \cite{cristofaro2020infinite,casals2022full}, and we refer the reader to these references for more details.

Recall that associated to each rigid del Pezzo surface $M$ is a sequence of nonnegative integers $g_1,g_2,g_3,\dots$ which determines the locations of the outer and inner corners of the corresponding infinite staircase.
These sequences are explicated in \cite[Table 1.18]{cristofaro2020infinite}, which is reproduced in Table~\ref{table:CG_et_al}.
Here $J$ denotes the number of ``strands'', $K+2$ is the degree of the corresponding del Pezzo surface, and $a_\acc$ is the accumulation point, i.e. the limiting $x$-value.
More explicitly, the sequence $g_1,g_2,g_3,\dots$ determines the locations of the outer and inner corner points in the graph of $c_M(x)$ as follows:
\begin{itemize}
  \item $k$th outer corner: $x$-coordinate $\frac{g_{k+J}}{g_k}$, $y$-coordinate $\frac{g_{k+J}}{g_k + g_{k+J}}$
  \item $k$th inner corner: $x$-coordinate $\frac{g_{k+J}(g_{k+1}+g_{k+1+J})}{g_{k+1}(g_k + g_{k+J})}$, $y$-coordinate $\frac{g_{k+J}}{g_{k} + g_{k+J}}$.
\end{itemize}
In particular, if $p/q$ is the $x$-coordinate of an outer corner, then $(Kp-q)/p$ is the $x$-coordinate of the outer corner $J$ steps away. This means that the full set of outer corners is obtained by iteratively applying the recursion $p/q \mapsto (Kp-q)/p$ to the seeds $\tfrac{g_{1+J}}{g_1},\dots,\tfrac{g_{2J}}{g_J}$.
Note that the generalized Orevkov twist achieves precisely the recursion $(p,q) \mapsto (Kp-q,p)$ by Proposition~\ref{prop:gen_orev_numerics}.

\begin{example}
In the case $M = \CP^2$, the sequence $g_1,g_2,g_3,\dots$ corresponds to the odd index Fibonacci numbers.
In the case $M = \CP^1 \times \CP^1$, the even index entries of $g_1,g_2,g_3,\dots$ correspond to the odd index Pell numbers, while the odd index entries of $g_1,g_2,g_3,\dots$ correspond to the even index half-companion Pell numbers.
\end{example}

\begin{rmk}\label{rmk:y_val_of_oc_is_aut}
Note that if $M$ is endowed with its monotone symplectic form $\omega_M$, and if $C$ is a $(p,q)$-sesquicuspidal rational symplectic curve in $M$ of index zero, then by Theorem~\ref{thm:stab_obs_from_curve} we have $c_M(p/q) \geq \tfrac{p}{[\omega_M] \cdot [C]} = \tfrac{p}{c_1([C])} = \tfrac{p}{p+q}$. Meanwhile, the outer corners described above are all of the form $(x,y) = (\tfrac{p}{q},\tfrac{p}{p+q})$ for $p,q \in \Z_{\geq 1}$.
\end{rmk}

\begin{table} \caption{The sequences controlling the rational infinite staircases, reproduced from \cite[Table 1.18]{cristofaro2020infinite}.}\label{table:CG_et_al} 
\begin{tabular}{|c||c|c|c|c|c|c|}\hline 
rigid del Pezzo  & negative weight 
& $K$ & $J$ & recursion & seeds& acc. pt.  \\
surface & expansion & & &$ g_{k+2J} = Kg_{k+J}-g_k $ & $g_0,\dots,g_{2J-1}$ & $a_\acc$ \\ 
\hline\hline
$\CP^2(3)$ & $(3)$ & $7$ &2& $g_{k+4} = 7g_{k+2}-g_k$ & $2,1,1,2$ & $\tfrac{7+3\sqrt{5}}{2}$  \\
 \hline
$\CP^1(2) \times \CP^1(2) $& $(4;2,2)$ & $6$&2& $g_{k+4} = 6g_{k+2}-g_k$  & $1,1,1,3$ & $3+2\sqrt{2}$\\ 
\hline
$\CP^2(3) \#\ovl{\CP}^2(1)$&$(3;1)$ &$6$&3&$g_{k+6} = 6g_{k+3}-g_k$ & $1,1,1,1,2,4$ & $3+2\sqrt{2}$\\
 \hline
$\CP^2(3) \#^{\times 2}\ovl{\CP}^2(1)$ & $(3;1,1)$ &$5$&3&$g_{k+6} = 5g_{k+3}-g_k$& $1,1,1,1,2,3$ & $\tfrac{5+\sqrt{21}}{2}$\\ 
\hline
$\CP^2(3) \#^{\times 3}\ovl{\CP}^2(1)$ & $(3;1,1,1)$ &$4$&2& $g_{k+4} = 4g_{k+2}-g_k$ & $1,1,1,2$ & $2+\sqrt{3}$\\ 
\hline
$\CP^2(3) \#^{\times 4}\ovl{\CP}^2(1)$ & $(3;1,1,1,1)$ &$3$&2& $g_{k+4} = 3g_{k+2}-g_k$ & $1,2,1,3$& $\tfrac{3+\sqrt{5}}{2}$\\ \hline
\end{tabular}
\end{table}
\MS

By the discussion in the previous subsection, in order to complete the proof of Theorem~\ref{thmlet:outer_corner_curves} it remains to construct seed curves.
Namely, for $M$ a rigid del Pezzo surface with corresponding integer sequence $g_0,g_1,g_2,\dots$, we must construct a well-placed $(g_{k+J},g_{k})$-unicuspidal rational algebraic curve in $M$ for $k = 0,\dots,J-1$.
More explicitly, inspecting Table~\ref{table:CG_et_al}, it suffices to find a well-placed $(p,q)$-unicuspidal rational algebraic curve $C$ with $(p,q)$ ranging as follows:
\begin{itemize}
\item $M = \CP^2$: $(p,q) = (1,2),(2,1)$
\item $M = \CP^1 \times \CP^1$ : $(p,q) = (1,1),(3,1)$
\item $M = \bl^1 \CP^2$ :  $(p,q) = (1,1),(2,1),(4,1)$
\item $M = \bl^2 \CP^2$:  $(p,q) = (1,1),(2,1),(3,1)$
\item $M = \bl^3 \CP^2$: $(p,q) = (1,1),(2,1)$ 
\item $M = \bl^4 \CP^2$: $(p,q) = (1,1),(3,2)$.
\end{itemize}

As above, for $k = 1,2,3,4$ we take $\bl^k\CP^2$ to be the blowup of $\CP^2$ at $k$ points $\nn_1,\dots,\nn_k$ in general position on the standard nodal cubic 
$\calN_0 = \{X^3+Y^3=XYZ\}$ that are away from the node,
 and we take $\calN$ to be the proper transform of $\calN_0$.
Meanwhile in the case of $\CP^1 \times \CP^1$ we take $\calN$ to be the smoothing of $(\CP^1 \times \{q_1,q_2\}) \cup (\{p_1,p_2\} \times \CP^1)$ at the nodes $(p_1,q_2),(p_2,q_1),(p_2,q_2)$.
\MS

\underline{Case $M = \CP^2$}: For $(p,q) = (1,2)$ and $(p,q) = (5,1)$ (i.e. $(7\cdot 1 -2,1))$ 
we take the line $L$ and conic $Q$ respectively mentioned in Theorem~\ref{thm:orev_orig}.

\underline{Case $M = \bl^1\CP^2$}:  
\begin{itemize}[topsep=0pt]
  \item For $(p,q) = (1,1)$, we take $C$ to be the proper transform of the unique line in $\CP^2$ which passes through $\nn_1$ and the double point $\db$ of $\calN_0$. Note that we have $[C] = \ell - e \in H_2(\bl^1\CP^2)$.
  \item  For $(p,q) = (2,1)$, we take $C$ to be the proper transform of the unique line in $\CP^2$ which is tangent to $\calB_-$ at $\db$. Since this line
  does not go through $\nn_1$,  we have $[C] = \ell \in H_2(\bl^1\CP^2)$.
  \item For $(p,q) = (4,1)$, we take $C$ to be the unique conic in $\CP^2$ which satisfies $\lll \T^{(4)}_{\calB_-}\db\rrr$ and also passes through $\nn_1$ (this is easily constructed using a linear system, or by a deformation argument similar to the ones given below).
\end{itemize}

\underline{Case $M = \CP^1 \times \CP^1$}: 
\begin{itemize}[topsep=0pt]
  \item For $(p,q) = (1,1)$, we take $C$ to be the unique line in class $\ell_1$ (or alternatively $\ell_2$) which passes through the double point $\db$ of $\calN$.
  \item For $(p,q) = (3,1)$, we take $C$ to be the unique rational curve of bidegree $(1,1)$ which has contact order $3$ to a branch of $\calN$ at $\db$. To construct such a curve, we can start with a bidegree $(1,1)$ curve $D$ passing through $\db$ and two other nearby points $x_1,x_2 \in \calB_-$ (e.g. $D$ can be realized as the graph of a holomorphic map $\CP^1 \ra \CP^1$). As we move the points $x_1,x_2$ into $\db$ along $\calB_-$, $D$ correspondingly deforms into a curve of the desired kind.
\end{itemize}

\underline{Case $M = \bl^2\CP^2$}:

\begin{itemize}[topsep=0pt]
  \item For $(p,q) = (1,1)$, we take $C$ to be the proper transform of the unique line in $\CP^2$ which passes through $\db$ and $\nn_1$ (or alternatively $\nn_2$)
  \item For $(p,q) = (2,1)$, we take $C$ to be the proper transform of the unique line in $\CP^2$ which is tangent to $\calB_-$ at $\db$.
 \item For $(p,q) = (3,1)$, we take $C$ to be the proper transform of the unique conic in $\CP^2$ which satisfies $\lll \T^{(3)}_{\calB_-}\db\rrr$ and passes through $\nn_1$ and $\nn_2$.
\end{itemize}

\underline{Case $M = \bl^3\CP^2$}:

\begin{itemize}[topsep=0pt]
  \item For $(p,q) = (1,1)$, we take $C$ to be the proper transform of the unique line in $\CP^2$ which passes through $\db$ and $\nn_1$.
  \item For $(p,q) = (2,1)$, we take $C$ to be the proper transform of the unique line in $\CP^2$ which is tangent to $\calB_-$ at $\db$.
\end{itemize}

\underline{Case $M = \bl^4\CP^2$}:

\begin{itemize}[topsep=0pt]
    \item For $(p,q) = (1,1)$, we take $C$ to be the proper transform of the unique line in $\CP^2$ which passes through $\db$ and $\nn_1$.

    \item For $(p,q) = (3,2)$, we take $C$ to be the proper transform of a rational cubic in $\CP^2$ which has a $(3,2)$ cusp 
  with contact order $3$ to $\calB_-$ at $\db$ and which passes through points $\nn_1,\nn_2,\nn_3,\nn_4$ on the standard nodal cubic $\calN_0$, as guaranteed by Lemma~\ref{lem:3_2_seed} below. 

  \end{itemize}

  \sss

We end this subsection by constructing the above $(3,2)$ seed curve, which then completes the proof of Theorem~\ref{thmlet:outer_corner_curves}.
Similar to \cite{cusps_and_ellipsoids}, we will denote by $\lll \CC^{(p,q)}\db\rrr$ the constraint that a curve has a holomorphic parametrization $u: \CP^1 \ra M$ such that $u([0:1]) = \db$, and $u$ has contact order at least $p$ with $\calB_-$ at $[0:1]$ and contact order at least $q$ with $\calB_+$ at $[0:1]$. Here $\calB_-,\calB_+$ are the local branches of $\calN_0 = \{X^3 + Y^3 = XYZ\}$ near its double point $\db = [0:0:1]$.

\begin{lemma}\label{lem:3_2_seed}
Given any four pairwise distinct points $\nn_1,\dots\nn_4 \in \calN_0 \setminus \{\db\}$, with $\nn_1,\dots,\nn_4,\db$ in general position, there exists a rational cubic algebraic curve in $\CP^2$ which satisfies the cuspidal constraint $\lll \CC^{(3,2)}\db\rrr$ and passes through $\nn_1,\dots,\nn_4$.
\end{lemma}
\begin{proof}
Consider the linear system of degree three homogenous polynomials $P(X,Y,Z)$, whose coefficients identify it with $\CP^{9}$. 
For a cubic curve $C = \{P(X,Y,Z) = 0\} \subset \CP^2$, the conditions 
\begin{itemize}
  \item  $C$ passes through $\db$
  \item $(C \cdot \calB_-)_{\db} = 3$ and $(C \cdot \calB_+)_{\db} = 2$
  \item $\nn_1,\dots,\nn_4 \in C$
\end{itemize}
impose a total of $8$ linear constraints on the coefficients of $P(X,Y,Z)$, and hence they cut out a linear subspace of $\CP^9$ of dimension at least $1$. 
Note that for a polynomial in this subspace, the corresponding curve $C = \{P(X,Y,Z) = 0\}$ could have a double point at $\db$, with one branch tangent to $\calB_-$. 
However, we can force $C$ to have a cusp at $\db$ but imposing that the determinant  of the Hessian
$\bigl(\begin{smallmatrix}
 \bdy_{xx}P & \bdy_{xy} P \\ \bdy_{yx}P & \bdy_{yy}P  
 \end{smallmatrix}\bigr)$
 vanishes at $[0:0:1]$. This amounts to one additional linear constraint, so that overall we get a linear subspace of $\CP^9$ of dimension at least $0$, and hence there is at least one curve $C$ satisfying all of these constraints.

It remains to show that $C$ is irreducible, and hence rational (since it has a singularity). 
In the case of three lines, two of them must pass through $\db$, with one tangent to $\db$. Then the third line must pass through three of the points $\nn_1,\dots,\nn_4$, which is ruled out by general position.
Meanwhile, in the case of a line and a conic, the Hessian determinant condition forces them to be tangent to each other at $\db$, and hence they must both be tangent to $\calB_-$ at $\db$. But this already contributes $6$ to the intersection number $C \cdot \calN_0$, which together with the constraints $\nn_1,\dots,\nn_4$ gives $C \cdot \calN_0 \geq 10$, a contradiction.
\end{proof}

\section{Inflating along sesquicuspidal curves}\label{sec:inflate}

The main goal of this section is to prove Theorem~\ref{thmlet:inflation_from_sescusp}, using the following basic outline:
\begin{enumerate}[label=\arabic*)]
  \item construct a (partial) resolution $\wt{C}$ of $C$ in a suitable iterated blowup $\wt{M}$ of $M$
  \item apply the technique of symplectic inflation to $\wt{C}$ to modify the symplectic form on $\wt{M}$
  \item blow down again to obtain a symplectic manifold $M'$ which is symplectomorphic to $M$ and by construction contains a large symplectic ellipsoid.
\end{enumerate}
The main technicality is that we need to perform the blowdowns in families in the symplectic category, where blowups and proper transforms are more delicate than in the complex category.
Indeed, recall that whereas complex blowups are performed at a point, symplectic blowups require the data of a symplectic ball embedding $\iota: B^{2n}(R) \hooksymp M$ for some $R \in \R_{>0}$.\footnote{Strictly speaking the construction requires choosing an extension of this embedding to $B^{2n}(R+\eps)$ for some $\eps > 0$, but we will suppress this from the discussion.} The symplectic blowup $\bl_\iota M$ is then defined roughly by removing the interior of $\iota(B^{2n}(R))$ and collapsing the boundary along the fibers of the Hopf fibration.
Some precise relations between complex and symplectic blowups are detailed in \cite[\S7.1]{INTRO}.

In \S\ref{subsec:toric_p_q}, we first discuss a model for the resolution $\wt{M} \ra M$ in the case of a $(p,q)$ cusp singularity using toric moment maps, and we use this to prove Theorem~\ref{thmlet:inflation_from_sescusp}(i) in \S\ref{subsec:infl_p_q}.
In \S\ref{subsec:res_mult_Puis} we extend the discussion to cusps with multiple Puiseux pairs, and finally we prove Theorem~\ref{thmlet:inflation_from_sescusp}(ii) in \S\ref{subsec:res_mult_Puis}. Along the way we also discuss some generalities on resolutions of cusp singularities which will be needed elsewhere in the paper.

\subsection{Toric resolution of a $(p,q)$ cusp}\label{subsec:toric_p_q}

Recall that any cusp singularity of a holomorphic curve $C \subset \C^2$ can be resolved by a finite sequence of point blowups (see e.g. \cite[\S3.3]{wall2004singular}).
We will denote the exceptional divisor resulting from the $i$th blowup by $\F_i^i$, and for $j > i$ we denote its proper transform in the $j$th blowup $\bl^j \C^2$ by $\F_i^j$. 
We arrive at the minimal resolution $C^K$ after some number $K \in \Z_{\geq 1}$ of blowups, and after $L-K$ further blowups for some $L \in \Z_{\geq 1}$ we arrive at the normal crossing resolution $C^L$, in which the total transform $C^L \cup \F_1^L \cup \cdots \cup \F_L^L \subset \bl^L\C^2$ of $C$ is a normal crossing divisor.
We have $[\F_L^L] \cdot [\F_L^L] = -1$ and $[\F_i^L] \cdot [\F_i^L] \leq -2$ for $i = 1,\dots,L-1$, and the spheres $\F_1^L,\dots,\F_{L-1}^L$ 
are disjoint from $C^L$, while $\F_L^L$ intersects $C^L$ transversely in one point.
In the case of a $(p,q)$ cusp, the combinatorics of the normal crossing resolution are related to the continued fraction expansion of $p/q$ and are neatly encoded in the so-called box diagram for $(p,q)$ (see \cite[\S4.1]{cusps_and_ellipsoids}).

Let $\mu_{\C^2}: \C^2 \ra \R^2_{\geq 0}$, $\mu(z_1,z_2) = (\pi |z_1|^2,\pi |z_2|^2)$, denote the moment map for the standard torus action on $\C^2$.
Given $p > q$ coprime positive integers, let $\Delta_{(q,p)} \subset \R^2_{\geq 0}$ denote the triangle with vertices $(0,0),(q,0),(0,p)$, let $\Om_{(q,p)} \subset \R^2_{\geq 0}$ denote the closure of its complement, and let $X_{(q,p)}$ denote the corresponding toric symplectic orbifold with moment map 
$\mu_{X_{(q,p)}}: X_{(q,p)} \ra \Om_{(q,p)}$ (this can be viewed as a weighted blowup of $\C^2$).
Note that $X_{(q,p)}$ has two cyclic quotient singularities of types $\tfrac{1}{p}(1,p-q)$ and $\tfrac{1}{q}(1,q-r)$, where $r$ is the remainder when $p$ is divided by $q$,\footnote
{For more information, see \S4.1a.} and these can each be resolved by finitely many toric blowups (c.f. \cite[Rmk. 4.3.4]{cusps_and_ellipsoids}). 
On the level of moment polygons, a toric blowup at a corner adjacent to edges having primitive inward normals $(1,0),(a,b) \in \Z^2$ with $a < b$ amounts to chopping off the corner so as to introduce a new edge with inward normal $(1,1)$ (the general case reduces to this one by an integral affine transformation).
We denote by $\wt{\Om}_{(q,p)} \subset \R_{\geq 0}^2$ any (noncompact) polygon obtained from $\Om_{(q,p)}$ after resolving both of the singularities by successive toric blowups.
The corresponding (noncompact) toric symplectic manifold $\mu_{\wt{X}_{(q,p)}}: \wt{X}_{(q,p)} \ra \wt{\Om}_{(q,p)}$ is also obtained from $\C^2$ by a sequence of $L$ toric blowups.
\sss

\begin{example}\label{ex:311}
Figure~\ref{fig:toric_blowup_seq} illustrates the construction of $\wt{X}_{(2,3)}$ from $\C^2$ by $3$ toric blowups, with corresponding inward normal vectors $(1,1),(2,1),(3,2)$.
\end{example}
\begin{figure}[h]
\caption{The sequence of toric blowups starting at $\C^2$ and ending at $\wt{X}_{(2,3)}$. The green lines represent the visible $(3,2)$-cuspidal curve $C_{(3,2)} \subset \C^2$ and its resolution $\wt{C}_{(3,2)} \subset \wt{X}_{(2,3)}$.}
\label{fig:toric_blowup_seq}
\centering
\includegraphics[width=1\textwidth]{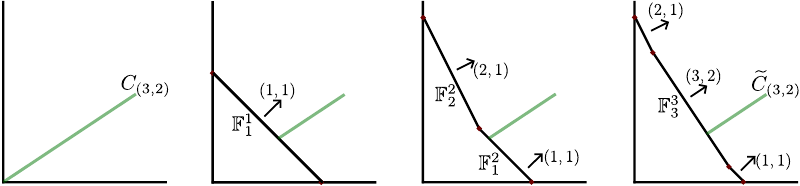}
\end{figure}

Let $\wt{\Ddiv}_{(q,p)} := \F_1^L \cup \cdots \cup \F_L^L$ denote the compact components of the toric boundary divisor in $\wt{X}_{(q,p)}$.
Note that, for $1 \leq i < j \leq L$, $\F_i^L$ and $\F_j^L$ are either disjoint or intersect symplectically orthogonally in one point.
For $i = 1,\dots,L$, the symplectic area of $\F_i^L$ is given by the affine length of the corresponding edge $\mu_{\wt{X}_{(q,p)}}(\F_i^L) \subset \bdy \wt{\Om}_{(q,p)}$, and evidently in the construction of $\wt{X}_{(q,p)}$ we can independently choose arbitrary values $\la_1,\dots,\la_L \in \R_{>0}$ for these affine lengths.

Given a neighborhood $U$ of $\mu_{\wt{X}_{(q,p)}}(\wt{\Ddiv}_{(q,p)})$ in $\wt{\Om}_{(q,p)}$, there is a corresponding neighborhood $V = U \cup \wt{\Delta}_{(q,p)}$ of the origin in $\R^2_{\geq 0}$, 
where $\wt{\Delta}_{(q,p)}$ denotes the closure of $\R^2_{\geq 0} \setminus \wt{\Om}_{(q,p)}$,
so that $\bdy U \cap \R_{>0}^2 = \bdy V \cap \R_{>0}^2$.
Then the corresponding toric domains $X_V := \mu_{\C^2}^{-1}(V)$ and $\wt{X}_{U} := \mu_{\wt{X}_{(q,p)}}^{-1}(U)$ coincide away from compact subsets.
Our model for the symplectic resolution $\wt{M} \ra M$ of a $(p,q)$ cusp singularity will roughly amount to excising $X_V$ and gluing in $\wt{X}_U$.
Note that the ellipsoid $E(q,p)$ naturally sits in $X_V$.

Now suppose that $C_1,\dots,C_L$ is any configuration of symplectically embedded two-spheres in a symplectic four-manifold $W$ which have the same respective areas as $\F_1,\dots,\F_L$ and the same intersection graph with symplectically orthogonal intersections.
Then by a version of the symplectic neighborhood theorem (see e.g. \cite[Prop. 3.5]{symington1998symplectic}), there is a neighborhood $\calW$ of $C_1 \cup \cdots \cup C_L$ in $W$ which is symplectomorphic to a neighborhood of $\F_1 \cup \cdots \cup \F_L$ in $\wt{X}_{(q,p)}$ of the form $\wt{X}_U$ for some $U \subset \wt{\Om}_{(q,p)}$ containing $\mu_{\wt{X}_{(q,p)}}(\wt{\Ddiv}_{(q,p)})$.
This means that there is a symplectic surgery of $W$ which excises $\calW$ and glues in $X_V$, with $V = U \cup \wt{\Delta}_{(q,p)}$ as above.
This gives an explicit model for the symplectic blowdown of $W$ along $C_L,\dots,C_1$, which by construction contains the ellipsoid $E(q,p)$.

\sss

Observe that there is a $(p,q)$-unicuspidal  symplectic curve $C_{(p,q)}$ in $\C^2$ whose image under $\mu_{\C^2}$ is the ray $\R_{\geq 0} \cdot (p,q)$, given explicitly by 
\begin{align}\label{eq:vis_curve_def}
C_{(p,q)} = \{ (r \sqrt{p}e^{2\pi i p t},r \sqrt{q}e^{2\pi i qt}) \in \C^2 \;|\; r \in \R_{\geq 0}, t \in [0,1]\}
\end{align}
(this is ``visible'' in the sense of \S\ref{subsubsec:vis_Lag_symp} below and can be viewed as the hyperK\"ahler twist of the Schoen--Wolfson Lagrangian discussed in \cite[Ex. 5.11]{evans2023lectures}).
Although $C_{(p,q)}$ is not equal on the nose to the model $(p,q)$ cusp $\{x^p + y^q = 0\}$, these two curves intersect the bounding sphere of a small neighborhood of the singular point in  
the same transverse torus knot,
namely the $(p,q)$ torus knot of maximal self-linking number, and hence there is a symplectic isotopy that takes one to the other. Thus
they are essentially interchangeable for the purposes of this section.
Similarly, there is a nonsingular visible symplectic curve $\wt{C}_{(p,q)}$ in $\wt{X}_{(q,p)}$ whose image under $\mu_{\wt{X}_{(q,p)}}$ is the intersection of the ray $\R_{\geq 0} \cdot (p,q)$ with $\wt{\Om}_{(q,p)}$.
In \S\ref{subsec:infl_p_q} we will take $\wt{C}_{(p,q)}$ as a model for the symplectic proper transform of $C_{(p,q)}$ in $\wt{X}_{(q,p)}$, noting that $\wt{C}_{(p,q)}$ intersects $\F_L^L$ positively in one point and is disjoint from $\F_i^L$ for $i = 1,\dots,L-1$.
The moment map images of $C_{(p,q)}$ and $\wt{C}_{(p,q)}$ are illustrated in Figure~\ref{fig:toric_blowup_seq} for the case $(p,q) = (3,2)$.

\subsection{Inflating along a curve with a $(p,q)$ cusp}\label{subsec:infl_p_q}

We first prove part (i) of Theorem~\ref{thmlet:inflation_from_sescusp}.

\begin{proof}[Proof of Theorem~\ref{thmlet:inflation_from_sescusp}(i)] 
Let $C$ be a $(p,q)$-sesquicuspidal symplectic curve in $M$ (as defined in \S\ref{subsubsec:sing_curves}), which satisfies $[C] = c\pd [\omega_M]$ and $[C] \cdot [C] \geq pq$.
After resolving any double points, we will assume that $C$ is nonsingular away from the $(p,q)$ cusp (but possibly of higher genus).
After further modifying $C$ near the cusp point, we can further assume that
\begin{itemize}
  \item  there is a neighborhood $\calD \subset M$ of the cusp which is symplectomorphic to $\eps \cdot X_V$, where $X_V = \mu_{\C^2}^{-1}(V) \subset \C^2$ as in \S\ref{subsec:toric_p_q} with $\wt{\Delta}_{(q,p)} \subset V \subset \R_{\geq 0}^2$, and $\eps \cdot X_V$ is the result after scaling the symplectic form by some $\eps > 0$ sufficiently small 
  \item $C \cap \calD$ is sent to $C_{(p,q)} \cap X_V$, with $C_{(p,q)}$ the visible symplectic curve defined in \eqref{eq:vis_curve_def}.
\end{itemize}

Let $U := V \setminus \Int \wt{\Delta}_{(q,p)}$ denote the corresponding neighborhood of the finite edges in $\wt{\Om}_{(q,p)}$, with associated domain $\wt{X}_{U} = \mu^{-1}_{\wt{X}_{(q,p)}}(U) \subset \wt{X}_{(q,p)}$.
Let $(\wt{M},\omega_{\wt{M}})$ denote the result after excising $\calD$ from $M$ and gluing in $\eps \cdot \wt{X}_U$ under the natural symplectic identification $\Op(\bdy \calD) \cong \Op(\bdy (\eps \cdot \wt{X}_U))$.
Let $\wt{C} \subset \wt{M}$ be the unique symplectic curve which agrees with $C$ outside of $\calD$ and agrees with $\wt{C}_{(p,q)}$ in $\eps \cdot \wt{X}_U$.
In other words, $\wt{M}$ is a model for the $L$-fold symplectic blowup of $M$, and $\wt{C}$ is a model for the symplectic resolution of $C$ at its cusp point.

We now symplectically inflate along $\wt{C}$ as follows. Note that $\wt{C}$ is smoothly embedded, and by assumption we have $[\wt{C}] \cdot [\wt{C}] = [C] \cdot [C] - pq \geq 0$.
Therefore, using e.g. \cite[Lem. 3.7]{mcduff1994notes}, there exists a closed two-form $\eta$ on $\wt{M}$ such that: 
\begin{itemize}
  \item $[\eta] = \pd([\wt{C}]) \in H^2(\wt{M};\R)$
  \item $\eta$ has support in a small neighborhood of $\wt{C}$ which is disjoint from $\F_1^L,\dots,\F_{L-1}^L$
  \item $\wt{\omega}_s := \omega_{\wt{M}} + s \eta$ is a symplectic form for all $s \in \R_{>0}$ 
  \item $\F_L^L$ is a symplectic submanifold of $(\wt{M},\wt{\omega}_s)$ for all $s \in \R_{>0}$.
\end{itemize}
Note that $(\wt{M},\wt{\omega}_s)$ contains the configuration of symplectic spheres $\F_1^L,\dots,\F_L^L$ which still intersect symplectically orthogonally with the same intersection pattern for all $s \in \R_{\geq 0}$, and we have 
\begin{align*}
\int_{\F_i^L}\wt{\omega}_s = 
\begin{cases}
  \int_{\F_i^L} \omega_{\wt{M}} & i = 1,\dots,L-1\\
  \int_{\F_L^L} \omega_{\wt{M}} + s & i = L.
\end{cases}
\end{align*}

Now let $(M_s,\omega_s)$ denote the result after performing the toric model for the symplectic blowdown along $\F_L^L,\dots,\F_1^L$ as described in \S\ref{subsec:toric_p_q}.
By choosing the relevant symplectic neighborhoods smoothly with $s$ and identifying $M_s$ smoothly with $M$, we view $\{\omega_s\}_{s \geq 0}$ as a smooth family of symplectic forms on $M$, such that $[\omega_s] = [\omega_M] + s\pd([C]) = (1+sc)[\omega_M]$, and such that there is a symplectic embedding of $(\eps + s)\cdot E(q,p)$ into $(M,\omega_s)$. 
By the Moser's stability theorem, the rescaled symplectic form $\tfrac{1}{1+sc}\cdot  \omega_s$ is symplectomorphic to $\omega_M$, and it admits a symplectic embedding of $\tfrac{\eps +s}{1+sc} \cdot E(q,p)$. Since $\tfrac{\eps + s}{1+sc} \ra \tfrac{1}{c}$ as $s \ra \infty$, the result now follows by taking $s$ sufficiently large.
\end{proof}

\subsection{Cusps with multiple Puiseux pairs}\label{subsec:res_mult_Puis}

In this subsection, we first recall some more generalities about cusp singularities and their resolutions,
in order to relate the blowup sequence for a $(p,q)$ cusp with that of a cusp with Puiseux pairs $(p,q),(p_2,q_2),\dots,(p_k,q_k)$.
In particular, we recall the definition of the Puiseux characteristic, which is a useful alternative to Puiseux pairs when  discussing blowups.
We then state a technical lemma relating symplectic and complex blowups which will be used in the proof of Theorem~\ref{thmlet:inflation_from_sescusp}(ii) at the end of this subsection.

According to \cite[\S2]{wall2004singular}, for $C$ any germ of a holomorphic curve near the origin in $\C^2$ which is not tangent to $\{x = 0\}$ we can find a local parametrization of the form
\begin{align*}
x = t^m, \;\;\;\;y = a_{m_1}t^{m_1} + a_{m_2}t^{m_2} + a_{m_3}t^{m_3} + \cdots,
\end{align*}
with $m \leq m_1 < m_2 < \cdots$ and $a_{m_1},a_{m_2},a_{m_3},\dots \in \C^*$, and such that $\gcd(m,m_1,m_2,\dots) = 1$.
Here $m$ is the multiplicity of $C$ at the origin. We define $\beta_1$ to be the smallest $m_i$ which is not a multiple of $m$, we put $e_1 := \gcd(m,\beta_1)$, and we put inductively 
\begin{align*}
\beta_{k+1} = \min \{m_i\;|\; e_k \notdiv m_i\},\;\;\;\;\; e_{k+1} = \gcd(e_k,\beta_{k+1}).
\end{align*}
We necessarily arrive at $e_g = 1$ for some $g \in \Z_{\geq 1}$, and the Puiseux characteristic of $C$ is by definition $(m;\beta_1,\dots,\beta_g)$.
One can show that this is independent of the coordinate representation of $C$ and is obtained from the Puiseux pairs $(n_1,d_1),\dots,(n_g,d_g)$ via
\begin{align*}
m = d_1 \cdots d_g, \;\;\;\;\; \beta_i = n_id_{i+1}\cdots d_g.
\end{align*}
In the reverse direction, given $(m;\beta_1,\dots,\beta_g)$ we can recover $n_1,\dots,n_g$ and $d_1,\dots,d_g$ via $\tfrac{\beta_i}{m} = \tfrac{n_i}{d_1\cdots d_i}$. 
In particular, note that if the first Puiseux pair is $(n_1,d_1) = (p,q)$ then the Puiseux characteristic takes the form $(kq;kp,\beta_2,\dots,\beta_g)$ with $k = d_2\cdots d_g$.
One can also show that the Puiseux characteristic determines the multiplicity sequence and vice versa (see \cite[Thm. 3.5.6]{wall2004singular}).

\begin{rmk}
Recall that, according to \cite{eisenbud1985three}, any local branch of a singular holomorphic curve in $\C^2$ is homeomorphic to the cone over an iterated torus knot, where the cabling parameters $(d_k,s_k)$ can be read off from the Puiseux pairs with $d_k \neq 1$ via $s_1 = n_1$ and $s_k = n_k - n_{k-1}d_k + d_{k-1}d_k s_{k-1}$ for $k \geq 2$ (here we follow the conventions of \cite{neumann2017topology}).
\end{rmk}

The following example is of primary relevance for Theorem~\ref{thmlet:inner_corner_curves}(b):
\begin{example}\label{ex:kp+1}
The Puiseux pairs $(p,q),(kp+1,k)$ correspond to the Puiseux characteristic $(kq;kp,kp+1)$ (and vice versa).
\end{example}

\sss

If $C$ has Puiseux characteristic $(m;\beta_1,\dots,\beta_g)$, then its proper transform after blowing up has Puiseux characteristic $(m',\beta_1',\dots,\beta_{g'}')$ given as follows (see \cite[Thm. 3.5.5]{wall2004singular}):
\begin{align}\label{eq:blowup_Puis_char}
(m';\beta_1',\dots,\beta_{g'}') = 
\begin{cases}
(m;\beta_1-m,\dots,\beta_g-m) & \beta_1 > 2m \\
(\beta_1-m;m,\beta_2-\beta_1+m,\dots,\beta_g-\beta_1 + m) & \beta_1 < 2m \text{ and } (\beta_1-m) \notdiv m \\
(\beta_1-m;\beta_2-\beta_1+m,\dots,\beta_g-\beta_1+m) & (\beta_1-m) | m.
\end{cases}
\end{align}

Using \eqref{eq:blowup_Puis_char}, the following is readily checked:
\begin{lemma}\label{lem:partial_res}
Suppose that the normal crossing resolution of a $(p,q)$ cusp requires $L$ blowups and results in negative self-intersection spheres $\F_1^L,\dots,\F_L^L$ as in \S\ref{subsec:toric_p_q}.
Let $C$ be any cusp singularity with Puiseux pairs $(n_1,d_1),(n_2,d_2),\dots,(n_g,d_g)$ with $(n_1,d_1) = (p,q)$. Then the first $L$ blowups of the resolution sequence for $C$ produce spheres $\G_1^L,\dots,\G_L^L$ having the same intersection pattern (including self-intersection numbers) as $\F_1^L,\dots,\F_L^L$.  The proper transform $\wt{C}$ of $C$ intersects $\G_L^L$ in one point with contact order $k := d_2\cdots d_g$, and is disjoint from $\G_1^L,\dots,\G_{L-1}^L$.   
\end{lemma}
\NI Note that $\wt{C}$ may itself have a residual cusp singularity.

\begin{example}[continuation of Example~\ref{ex:kp+1}]\label{ex:res_of_kq_kp_kp+1}
Let $C$ be a curve with cusp having Puiseux characteristic of the form $(kq;kp,kp+1)$ (and hence Puiseux pairs $(p,q), (kp+1,k)$), and let $L$ be as in Lemma~\ref{lem:partial_res}. Then the first $L$ blowups in the resolution sequence achieve the normal crossing resolution for a $(p,q)$ cusp, after which the proper transform $\wt{C}$ is nonsingular but intersects $\G_L^L$ in a single point of contact order $k$. Thus $L$ blowups achieve the minimal resolution for $C$, and an additional $k$ blowups achieve the normal crossing resolution for $C$.

For instance, the minimal resolution sequence for Puiseux characteristic $(3;5)$ is $(3;5) \ra (2;3) \ra (1;2)$, while that of $(9;15,16)$ is $(9;15,16) \ra (6;9,10) \ra (3;7) \ra (3;4) \ra (1,3)$ (this corresponds to $(q,p) = (3,5)$, $k=3$, and $L = 4$). 
\end{example}

We record the following for later purposes: 
\begin{lemma}\label{lem:k^2pq}
  Let $C$ be a curve with a cusp having Puiseux characteristic $(kq;kp,kp+1)$, and let $\wt{C}$ be its minimal resolution. Then we have $[\wt{C}] \cdot [\wt{C}] = [C] \cdot [C] - k^2pq$. In particular, $[C] \cdot [C] \geq k^2pq$ if and only if $[\wt{C}] \cdot [\wt{C}] \geq 0$.
\end{lemma}

We now are in a position to complete the proof of Theorem~\ref{thmlet:inflation_from_sescusp}.
In \S\ref{subsec:infl_p_q} we assumed that $C$ coincides with the visible symplectic curve $C_{(p,q)} \subset \C^2$ from \eqref{eq:vis_curve_def} locally near its cusp. 
Since this curve has no direct analogue for a cusp with multiple Puiseux pairs, we will augment the explicit toric resolution model from \S\ref{subsec:toric_p_q} with a slightly more abstract argument.

The following technical lemma relating symplectic and complex blowups will suffice for our purposes.
\begin{lemma}[{see \cite[\S7.1]{INTRO}}]\label{lem:bu_tech}
Let $(M,\omega)$ be a symplectic manifold equipped with an $\om$-tame almost complex structure $J$ which is integrable near a point $\db \in M$.
Let $(\bl_\db M,\wt{J})$ denote the complex blowup $M$ at $\db$, with exceptional divisor $\E_{\bl_\db M}$.
Then for some $\delta > 0$ there exists a symplectic embedding $\iota: (B^{2n}(\delta),\omega_\std) \hooksymp (M,\omega)$ with $\iota(0) = \db$ for which the corresponding symplectic blowup
 $(\bl_\iota M,\wt{\omega})$ admits a diffeomorphism $\Phi: \bl_\iota M \xrightarrow[]{\cong} \bl_\db M$ such that 
$\Phi^*\wt{J}$ is $\wt{\omega}$-tame and 
 $\Phi(\E_{\bl_\iota M}) = \E_{\bl_\db M}$.

Furthermore, suppose that $\Ddiv_1$ and $\Ddiv_2$ are smooth $J$-holomorphic local divisors in $M$ which intersect $\omega$-orthogonally at $\db$, and let $\wt{\Ddiv}_1,\wt{\Ddiv}_2 \subset \bl_\db M$ denote their $\wt{J}$-holomorphic proper transforms. Then we can arrange that $\Phi^{-1}(\wt{\Ddiv}_1)$ and $\Phi^{-1}\wt{\Ddiv}_2$ each intersect $\E_{\bl_\iota M}$ $\wt{\omega}$-orthogonally.
\end{lemma}

Note that, in the context of Lemma~\ref{lem:bu_tech}, if $C$ is a (singular) symplectic curve in $M$ which is preserved by $J$ near $\db$, then we can define its symplectic proper transform to be $\Phi^{-1}(\wt{C})$, where $\wt{C}$ is the $\wt{J}$-holomorphic proper transform of $C$ in $\bl_\db M$. 

\begin{proof}[Proof of Theorem~\ref{thmlet:inflation_from_sescusp}(ii)]
  
Let $C$ be a sesquicuspidal symplectic curve in $M$ with Puiseux pairs $(p,q),(p_2,q_2),\dots,(p_g,q_g)$, whose homology class satisfies $[C] = c \PD[\omega_M]$ and $[C] \cdot [C] \geq k^2 pq$ for $k = q_2\cdots q_g$.
As before, after resolving any double points we can assume that $C$ is embedded
away from the cusp point. 
By assumption there is a neighborhood $U$ of the cusp point such that $(U,C \cap U)$ is symplectomorphic to $(U',C' \cap U')$, where $U'$ is a neighborhood of the origin in $\C^2$ and $C' \subset \C^2$ is a holomorphic curve having a cusp with Puiseux pairs $(p,q)=(p_1,q_1),\dots,(p_g,q_g)$. 
By pulling back $J_\std|_{U'}$ to $U$ and extending over $M$, we can find an $\omega_M$-compatible almost complex structure $J$ on $M$ which preserves $C$ and is integrable near the cusp.

Let $(\wt{M}_\comp,\wt{J})$ denote the $L$-fold complex blowup of $(M,J)$ which achieves normal crossing resolution of a $(p,q)$ cusp singularity as in Lemma~\ref{lem:partial_res}, with negative self-intersection $\wt{J}$-holomorphic spheres $\G_1^L,\dots,\G_L^L \subset \wt{M}_\comp$.
Let $\wt{C} \subset \wt{M}_\comp$ be the corresponding proper transform of $C$ (this may be smooth or have a residual cusp).
Using Lemma~\ref{lem:bu_tech}, there is a corresponding $L$-fold symplectic blowup $(\wt{M}_\symp,\wt{\omega})$ of $(M,\omega_M)$ and a diffeomorphism $\Phi: \wt{M}_\symp \xrightarrow[]{\cong} \wt{M}_\comp$ such that $\Phi^*\wt{J}$ is $\wt{\omega}$-tame, and the symplectic spheres 
$\Phi^{-1}(\G_L^1),\dots,\Phi^{-1}(\G_L^L)$ intersect symplectically orthogonally.
After smoothing the residual cusp of $\Phi^{-1}(\wt{C})$ (if necessary) by replacing it with a perturbation of the corresponding Milnor fiber, we obtain a symplectically embedded curve $D \subset \wt{M}_\symp$ which intersects $\Phi^{-1}(\G_L^L)$ positively in $k$ points and is disjoint from $\Phi^{-1}(\G_1^L),\dots,\Phi^{-1}(\G_{L-1}^L)$. 
Note that  $[\Phi^{-1}(\wt{C})]$ has positive self-intersection number by Lemma~\ref{lem:k^2pq}.
The rest of the proof proceeds as in \S\ref{subsec:infl_p_q} by inflating along $D$, blowing down $\Phi^{-1}(\G_{L-1}^L),\dots,\Phi^{-1}(\G_1^L)$ using the same toric model from \S\ref{subsec:toric_p_q}, and finally rescaling the symplectic form and applying Moser's stability theorem.
Note that after inflating $\Phi^{-1}(\G_L^L)$ has symplectic area $\eps + ks$ since $[D] \cdot [\G_L^L] = k$, and hence the rescaled symplectic manifold $(M,\omega_s)$ admits a symplectic embedding of $\tfrac{\eps + ks}{1+sc} E(q,p)$.
\end{proof}

\section{$\Q$-Gorenstein smoothings and almost toric fibrations} \label{sec:ATF1}

In this section we collect various facts about (a) $\Q$-Gorenstein smoothings of singular toric algebraic surfaces (\S\ref{subsec:toric_surf}) and (b) symplectic almost toric fibrations (\S\ref{subsec:atfs_and_polys}).
Few if any of the results in this section are original, but our perspective is somewhat novel in that we emphasize the central role played by $T$-polygons (see \S\ref{subsubsec:T-sings}) in both algebraic and symplectic geometry.
Roughly, we associate to a $T$-polygon $Q$ both an algebraic surface $\wt{V}_Q$ (defined as a $\Q$-Gorenstein smoothing) and a symplectic four-manifold $\atf(Q_\nodal)$ (defined as the total space of an almost toric fibration).
Proposition~\ref{prop:QG_def_diff_ATF} gives a direct comparison between these two geometries, which we utilize in \S\ref{sec:singI} and \S\ref{sec:singII} in order to construct algebraic curves via symplectic techniques.

\subsection{Toric surfaces and $T$-singularities} \label{subsec:toric_surf}

In this subsection, we begin by briefly reviewing some toric algebraic geometry and singularity theory and setting up our notation. We then recall the notion of $T$-singularities and their $\Q$-Gorenstein smoothings, and define $T$-polygons. We also discuss (dual) Fano polygons and their mutations, which play an important role in the mirror symmetry approach to Fano surfaces (see e.g. \cite{galkin2010mutations,akhtar2016mirror,kasprzyk2017minimality,coates2012mirror}).

\subsubsection{Cyclic quotient singularities and toric surfaces}

For $\ka \in \Z_{\geq 1}$, let 
\begin{align*}
\roots_\ka = \{e^{2\pi\sqrt{-1}  j/\ka} \;|\; j = 0,\dots,\ka-1\}
\end{align*}
 denote the group of $\ka$th roots of unity.
Given $w_1,\dots,w_n \in \Z_{\geq 0}$, we consider the action of $\mu_\ka$ on $\C^n$ with weights $w_1,\dots,w_n$, i.e. with $\mu \cdot (z_1,\dots,z_n) = (\mu^{w_1}z_1,\dots,\mu^{w_n}z_n)$ for $\mu \in \roots_\ka$.
Note that the weights $w_1,\dots,w_n$ are only relevant modulo $\ka$.
We denote this representation of $\roots_\ka$ by $\roots_\ka^{w_1,\dots,w_n}$.

Cyclic quotient singularities are by definition quotients of the form
\begin{align*}
\tfrac{1}{\ka}(w_1,\dots,w_n) := \C^n / \roots_\ka^{w_1,\dots,w_n}
\end{align*}
for $\ka \in \Z_{\geq 1}$ and $w_1,\dots,w_n \in \Z_{\geq 1}$ coprime to $\ka$.
Here we have $\tfrac{1}{\ka}(w_1,\dots,w_n) = \tfrac{1}{\ka}(\ell w_1,\dots,\ell w_n)$ for any $\ell \in (\Z/\ka)^\times$, so we may assume $w_1 = 1$.
Note that the cyclic quotient surface singularity $\tfrac{1}{\ka}(1,w)$ is the affine toric variety $U_\si := \spec_m(\C[S_\sigma])$ corresponding to the cone $\sigma \subset \R^2$ generated by $(0,1)$ and $(\kappa,-w)$. Here $\sigma^\vee \subset \MM_\R$ is the dual cone to $\si$, $S_\sigma$ is the semigroup of lattice points in $\sigma^\vee$, and $\C[S_\sigma]$ is the associated semigroup algebra (see e.g. \cite{cox2011toric,fulton1993introduction,brasselet2004introduction,da2003symplectic} for more background on toric varieties).
\MS

Let $\NN$ be a lattice of rank $n \in \Z_{\geq 1}$, with dual lattice $\MM = \hom(\NN,\Z)$ (typically we will have $\NN = \Z^n$, but this notation is still helpful in distinguishing the roles of $\NN$ and $\MM$).
We put $\NN_\R := \NN \otimes_\Z \R$ and $\MM_\R := \MM \otimes_\Z \R$.
We will say that a polytope\footnote{By \hl{polytope} $P \subset \NN_\R$ we mean the convex hull of finitely many points in $\NN_\R$. We call this a \hl{polygon} when $\NN$ has rank two.} $P \subset \NN_\R$ is \hl{centered}
if $P$ is $n$-dimensional and contains the origin in its interior.
Given a centered polytope $P \subset \NN_\R$, the \hl{dual polytope} $P^o \subset \MM_\R$ is by definition
\begin{align*}
P^o := \{u \in \MM_\R\;|\; \lan u,v\ran \geq -1 \;\forall\; v \in P \}.
\end{align*}
Note that (unless $P$ is reflexive) $P^o$ is typically not a lattice polytope (i.e. having vertices in $\MM$), even if $P$ is.  
For a polytope $Q \subset \MM_\R$, the dual polytope $Q^o \subset \NN_\R$ is defined similarly.

We associate to $P$ its \hl{face fan} $\fan_{P}$ in $\NN_\R$, which has a cone $\si_\tau$ for each face $\tau$ of $P$, where $\si_\tau$ is generated by the vertices of $\tau$.
Equivalently, this is the \hl{normal fan} $\fan_{P^o}$ of $P^o$, which has a cone $\si_\eta$ for each face $\eta$ of $P^o$, where $\si_\eta$ is generated by the inward normal vectors of those facets of $P^o$ which contain $\eta$.
We denote by $V_\fan$ the (typically singular) toric variety associated to a fan $\fan$.
In the case $\fan = \fan_P = \fan_{P^o}$ we will also denote $V_\fan$ by $V_P$ or $V_{P^o}$ when we wish to emphasize the polytope $P$ or its dual $P^o$.\footnote{It should be clear from the context whether we are taking the face fan or normal fan since $P$ and $P^o$ live in different vector spaces.}

For a general polygon $Q \subset \MM_\R$,
the toric surface $V_{Q}$ has cyclic quotient singularities at its toric fixed points, which correspond to the vertices of $Q$ (or equivalently the maximal cones of the normal fan $\fan_{Q^o}$).
Explicitly, for each vertex $v \in Q$ there is an integral affine transformation\footnote{By \hl{integral transformation} of $\MM_\R$ we mean a map $\MM \otimes_\Z \R \ra \MM \otimes_\Z \R$ which is a group isomorphism $\MM \cong \MM$ on the first factor and the identity on the second factor. By \hl{integral affine transformation} of $\MM_\R$ we mean the composition of an integral transformation with a translation.} of $\MM_\R$ sending $v$ to the origin, so that the edge directions become $(0,1)$ and $(n,-q)$ for some coprime $n,q \in \Z_{\geq 1}$, in which case the singularity has type $\tfrac{1}{n}(1,q)$.
In the case $n=1$, this corresponds to a smooth point of $V_Q$ and we refer to $v$ as a \hl{Delzant vertex} vertex of $Q$.
If $Q$ has only Delzant vertices then it is a \hl{Delzant polygon}.

We end this subsection with a remark about the homology of a smooth toric surface.
Let $Q \subset \MM_\Q$ be a Delzant polygon with edges $e_1,\dots,e_\ell$ and corresponding toric divisors $\Ddiv_{e_1},\dots,\Ddiv_{e_\ell} \subset V_Q$.
Recall that the homology group $H_2(V_Q)$ of the associated nonsingular toric variety $V_Q$ is generated by the toric divisors $\Ddiv_{e_1},\dots,\Ddiv_{e_\ell}$. More precisely, letting $\vecn_1,\dots,\vecn_\ell \in \NN$ be the primitive inward normal vectors to the edges, we have the short exact sequence
\begin{align}\label{eq:hom_of_V_Q}
0 \ra \MM \ra \Z\lan [\Ddiv_{e_1}],\dots,[\Ddiv_{e_\ell}]\ran \ra H_2(V_Q) \ra 0,
\end{align}
where the first nontrivial map sends $u \in \MM$ to $\sum_{i=1}^\ell \langle \vecn_i,u\ran [\Ddiv_{e_i}]$; see \cite[\S5.1]{cox2011toric}.

\subsubsection{Fano and dual Fano polygons}\label{subsubsec:dual_Fano}

A polytope $P \subset \NN_\R$ is said to be \hl{Fano} if it is centered and its vertices are primitive lattice vectors (see e.g. \cite[\S3]{akhtar2012minkowski}).
In this case the corresponding toric variety $V_P$ has anticanonical divisor which is $\Q$-Cartier and ample, i.e. $V_P$ is a (typically singular) toric Fano variety. In particular, if we now restrict to the case $\dim(P) = 2$, $V_P$ is a singular toric del Pezzo surface.
We will say that a centered polygon $Q \subset \MM_\R$ is \hl{dual Fano} if the dual polygon 
$Q^o \subset \NN_\R$ is Fano, and hence in particular a lattice polytope.
Note that the dual Fano condition is equivalent to each edge $e$ of $Q$ having height one, where we define the \hl{height} of an edge to be the number 
$\height(e)$ such that $\lan \vecn,u\ran = -\height(e)$ for all $u \in e$, where $\vecn \in \NN$ is the primitive inward normal vector to $e$.

For points $v,w \in \MM_\R$, recall that the \hl{affine length} $\afflen([v,w])$ of the line segment $[v,w] \subset \MM_\R$ is $|c|$, where we put $v-w = c(v-w)_\prim$ for $c \in \R$ and $(v-w)_\prim \in \MM$ a primitive lattice vector. The following gives another characterization of the dual Fano condition (c.f. Proposition~\ref{prop:dual_Fano_monotone} for a symplectic counterpart).

\begin{lemma}\label{lem:delz_Q_dual_Fano_c1_crit}
If a centered Delzant polygon $Q \subset \MM_\R$ is dual Fano, then we have $c_1([\Ddiv_e]) = \afflen(e)$ for each edge $e$. The converse also holds if $Q$ is a lattice polygon.
\end{lemma}

\begin{proof} 
We can assume $\MM = \Z^2$ and $\MM_\R = \R^2$ without loss of generality. Let $\vv_1,\dots,\vv_\ell$ be the vertices of $Q$ ordered counterclockwise, and let $e_i$ be the edge joining $\vv_i$ and $\vv_{i+1}$ for $i = 1,\dots,\ell$ (modulo $\ell$). 
Let $\height(e_i)$ denote the height of the edge $e_i$, and let $\vecn_i \in \NN$ denote the primitive inward normal vector to the edge $e_i$ for $i = 1,\dots,\ell$.

After applying an integral transformation of $\R^2$, we can further assume that $\vecn_1 = (0,1)$ and $\vecn_\ell = (1,0)$, and hence $$
\vv_1 = (-\height(e_\ell),-\height(e_1))\;\; \mbox{ and } \;\;  \vv_2 = \vv_1 + (\afflen(e_1),0).
$$
By \eqref{eq:hom_of_V_Q}, we have
$\sum\limits_{i=1}^\ell \lan \vecn_i,\vv_2\ran \left([\Ddiv_{e_i}] \cdot [\Ddiv_{e_1}]\right) = 0$.
Further we have $[\Ddiv_{e_i}] \cdot [\Ddiv_{e_1}] = 0$ for $i \notin \{1,2,\ell\}$,
and  $[\Ddiv_{e_1}] \cdot [\Ddiv_{e_2}] = [\Ddiv_{e_1}] \cdot [\Ddiv_{e_\ell}] = 1$ 
 since $Q$ is Delzant.
Thus 
\begin{align*}
0 &= \lan \vecn_1,\vv_2 \ran \left([\Ddiv_{e_1}] \cdot [\Ddiv_{e_1}]\right) + \lan \vecn_2,\vv_2\ran + \lan \vecn_\ell,\vv_2 \ran 
\\ &= -\height(e_1)\left([\Ddiv_{e_1}] \cdot [\Ddiv_{e_1}]\right) -\height(e_2)  -\height(e_\ell) + \afflen(e_1),
\end{align*}
so that
\begin{align*}
[\Ddiv_{e_1}] \cdot [\Ddiv_{e_1}] = \frac{-\height(e_2) - \height(e_\ell) + \afflen(e_1)}{\height(e_1)}.
\end{align*}
Noting that $\Ddiv_{e_1}$ is an embedded two-sphere, by the adjunction formula and symmetry we have 
\begin{align}\label{eq:c_1_of_D_e_i}
c_1([\Ddiv_{e_i}]) = 2 + \frac{-\height(e_{i+1}) - \height(e_{i-1}) + \afflen(e_i)}{\height(e_i)}
\end{align}
for $i = 1,\dots,\ell$.

If $Q$ is dual Fano, then we have $\height(e_i) = 1$ for $i = 1,\dots,\ell$, so \eqref{eq:c_1_of_D_e_i} becomes
$c_1([\Ddiv_{e_i}]) = \afflen(e_i)$.
Conversely, if $Q$ is a lattice polygon and if $c_1([\Ddiv_{e_i}]) = \afflen(e_i)$ for $i = 1,\dots,\ell$, then \eqref{eq:c_1_of_D_e_i} becomes
$\afflen(e_i) = 2 + \frac{-\height(e_{i+1}) - \height(e_{i-1}) + \afflen(e_i)}{\height(e_i)}$. Since $\height(e_1),\dots,\height(e_\ell) \geq 1$, this is only possible if $\height(e_1) = \cdots = \height(e_\ell) = 1$.
\end{proof}

\subsubsection{$T$-singularities and polygon mutations}\label{subsubsec:T-sings}
A cyclic quotient surface singularity $\tfrac{1}{n}(1,q)$ is a \hl{$T$-singularity} if we have $n = mr^2$ and $q = mra-1$ for some $m,r,a \in \Z_{\geq 1}$ with $\gcd(r,a) = 1$.\footnote
{Notice that the singularity type depends only  on $a$ mod $r$. }
These were shown in \cite{kollar1988threefolds} to be precisely those cyclic quotient surface singularities which admit $\Q$-Gorenstein smoothings. 
Here a \hl{$\Q$-Gorenstein smoothing} of a normal surface $X$ with quotient singularities is a flat family $\calX$ over a smooth curve germ $\calS$ such that the central fiber is $X$, the general fiber is smooth, and the relative canonical divisor $K_{\calX/\calS}$ is $\Q$-Cartier (see e.g. \cite[\S2.1]{hacking2010smoothable} or \cite[\S2]{lee2011construction} and the references therein). Note that this last condition is equivalent to the 
total space being $\Q$-Gorenstein, i.e. having $\Q$-Cartier canonical divisor.

Although the local deformation theory of the cyclic quotient surface singularity $\tfrac{1}{n}(1,q)$ is quite complicated, 
the restriction to $\Q$-Gorenstein deformations is well-understood by \cite{kollar1990flips,kollar1988threefolds} (c.f. \cite[\S1]{akhtar2016mirror}).
Specializing to the case of $T$-singularities, the base of the miniversal family of $\Q$-Gorenstein deformations of the $T$-singularity $\tfrac{1}{mr^2}(1,mra-1)$ 
is isomorphic to $\C^{m-1}$, corresponding to the family of hypersurfaces
\begin{align}\label{eq:T_sing_miniversal}
\{xy = z^{rm} + C_{m-2}z^{r(m-2)} + \cdots + C_{1}z^{r} + C_0\} \subset \C^3/\roots_r^{1,-1,a}
\end{align}
for parameters $C_0,\dots,C_{m-2} \in \C$.
Note here that the central fiber is indeed isomorphic to ${\tfrac{1}{mr^2}(1,mra-1)}$ by the isomorphism
\begin{align}\label{eq:Milnor_quot_iso}
\C^2/\roots_{mr^2}^{1,mra-1} \xrightarrow{\cong} \{xy = z^{rm}\} / \mu_r^{1,-1,a},\;\;\;\;\; (z_1,z_2) \mapsto (z_1^{rm},z_2^{rm},z_1z_2).
\end{align}
The general fiber is smooth and is diffeomorphic to 
\begin{align}\label{eq:B_m_r_a}
B_{m,r,a} := \{xy = (z^r-\zeta_1)\cdots(z^r-\zeta_m)\}/\roots_r^{1,-1,a},
\end{align}
for some real $0 < \zeta_1 < \cdots < \zeta_m$,
i.e. $B_{m,r,a}$ is the quotient of the $A_{rm-1}$ Milnor fiber by $\mu_r$, and we have $H_1(B_{m,r,a};\Q) = 0$ and $\dim H_2(B_{m,r,a};\Q) = m-1$.
In the special case $m=1$, $B_{1,r,a}$ is a rational homology ball and plays an important role in constructing exotic four-manifolds with small homology groups (see e.g. \cite{fintushel1997rational}).

We will refer to a vertex $\vv \in Q$ of a polygon $Q \subset \MM_\R$ as a \hl{$T$-vertex} if the corresponding toric fixed point $\pp_\vv \in V_Q$ is a $T$-singularity, and we will call $Q$ a \hl{$T$-polygon} if all of its vertices are $T$-singularities (note this includes the case of Delzant vertices).\footnote
{
The notions of $T$-polygon and dual Fano polygon are  independent: there are $T$-polygons whose fan is not Fano and there are dual Fano polygons that are not $T$-polygons.}
Given a $T$-vertex $\vv$, there is an isomorphism $\MM_\R \cong \R^2$ of integral affine manifolds which sends $\vv$ to $(0,0)$ with edge vectors $(0,1)$ and $(mr^2,mra-1)$,
and we will refer to the image of the direction $(r,a)$ as the\footnote{One can check that this definition is unambiguous, 
since the integral affine transformation of $\R^2$ which swaps $(0,1)$ and $(mr^2,mra-1)$ fixes $(r,a)$. In particular, 
this singularity is equivalent to its reflection in the $x$-axis with edge vectors
 $,(0,-1), (mr^2,1-mra)$ and eigenray $(r,-a)$.}
\hl{eigenray emanating from $\vv$}
(this corresponds to the eigendirection of a suitable affine monodromy in \S\ref{subsubsec:nodal_int_from_poly}).\footnote{This is also often referred to as a \hl{nodal ray} in the context of almost toric fibrations as in \S\ref{subsubsec:nodal_int_from_poly}.}

Let $Q \subset \MM_\R$ be a $T$-polygon, and as before let $V_Q$ denote the corresponding toric surface with $T$-singularities. By definition $T$-singularities admit local $\Q$-Gorenstein smoothings as in \eqref{eq:T_sing_miniversal}, and according to \cite[Prop. 3.1]{hacking2010smoothable}\footnote{To apply the hypotheses of \cite[Prop. 3.1]{hacking2010smoothable} it suffices to note that the anticanonical divisor of a toric divisor is always big. Indeed, if $\Ddiv_1,\dots,\Ddiv_k$ denote the toric boundary divisors, then we can find an ample divisor of the form $\sum_{i=1}^k a_i \Ddiv_i$ for $a_1,\dots,a_k \in \Z_{\geq 1}$. Then for any positive integer $m \geq a_1,\dots,a_k$ we have that $m$ times the anticanonical divisor is
$m \sum_{i=1}^k \Ddiv_i = \sum_{i=1}^k a_i \Ddiv_i + \sum_{i=1}^k (m - a_i)\Ddiv_i$, which is a sum of an ample divisor and an effective divisor and hence big by \cite[Cor 2.2.7]{lazarsfeld2017positivity}.} 
there are no local-to-global obstructions to deformations, so in particular $V_Q$ admits a $\Q$-Gorenstein smoothing.
Thus we have:

\begin{lemma}\label{lem:QGorsmooth} For any $T$-polygon $Q \subset \MM_\R$, there is a $\Q$-Gorenstein smoothing $\wt{V}_Q$ of $V_Q$. If in addition $Q$ is  dual Fano, then $\wt{V}_Q$ is Fano for any sufficiently small smoothing, and in particular rigid if $\rk\, H^2(\wt{V}_Q, \Q) \le 5$. 
\end{lemma}

Following e.g. \cite{galkin2010mutations,akhtar2012minkowski}, there is a notion of mutation which inputs a dual Fano polygon\footnote{One can also describe the mutation in terms of the Fano polygon $Q^o \subset \NN_\R$, although for our purposes the formula using dual Fano polygons is more succinct and more directly related to mutations of almost toric fibrations. An extension of mutations to higher dimensional Fano polytopes is also defined in  \cite{akhtar2012minkowski}.} $Q \subset \MM_\R$ and a choice of vertex $\vv \in Q$ and produces a new dual Fano polygon $\mut_\vv(Q) \subset \MM_\R$. 
Namely, let $f \in \NN$ be a primitive lattice vector 
such that $\lan f,\vv\ran = 0$,
let $\vv_\prim = t\vv \in \MM$ be primitive for some $t \in \R_{>0}$,
and consider the piecewise-linear map 
\begin{align}\label{eq:psi_mut}
\psi: \MM_\R \ra \MM_\R,\;\;\;\;\; u \mapsto u - \min(\lan f,u\ran,0)\,\vv_\prim. 
\end{align}
We put $\mut_\vv(Q) := \psi(Q) \subset \MM_\R$.
In the case $\MM_\R = \R^2$, this is equivalent to
\begin{align}\label{eq:mutation}
\psi(u) = 
\begin{cases}
\shear_{\vv_\prim}(u) & \text{if}\;\; \lan f,u\ran \leq 0\\
  u & \text{if}\;\; \lan f,u\ran > 0,
\end{cases}
\end{align}
where $\shear_{\vv_\prim}(u): \R^2 \ra \R^2$ is the \hl{primitive shear} along $\vv_\prim$, defined by 
\begin{align}\label{eq:prim_shear}
\shear_{\vv_\prim}(u) = u + \det(\vv_\prim,u) \cdot \vv_{\prim},
\end{align}
with $\det(\vv_\prim,u)$ the determinant of the $2 \times 2$ matrix with columns $\vv_\prim$ and $u$.
We will say that $\mut_\vv(Q)$ is the \hl{(primitive) mutation} of the polygon $Q$ at the vertex $\vv$.

\begin{rmk}
 Strictly speaking there are two choices for $f$ in the above (i.e. in the case $\vv = (x,y) \in \R^2$ we have $f = \pm(y,-x)$), although the choice becomes unique if we choose an orientation on $\MM_\R$ and ask for $f,\vv_\prim$ to be an oriented basis. At any rate, it is easy to show using \eqref{eq:mutation} that the two choices for $f$ give mutations which are related by an integral affine transformation of $\MM_\R$.
\end{rmk}

We also define $\mut_\vv(Q)$ in the case that $Q \subset \MM_\R$ is any polygon (not necessarily dual Fano) and $\vv \in Q$ is a $T$-vertex.
Note that in this case the origin in $\MM_\R$ is not necessarily in distinguished position relative to $Q$ (it may not even be contained in $Q$).
We put
\begin{align}\label{eq:mutation_non_Fano}
\mut_\vv(Q) := \tau^{-1}(\mut_{\tau(\vv)}(\tau(Q))),
\end{align}
where $\tau: \MM_\R \ra \MM_\R$ is any translation sending a point on the eigenray emanating from $\vv$ to the origin (c.f. \S\ref{subsubsec:T-sings}), and 
$\mut_{\tau(\vv)}(\tau(Q))$ is defined as above. One can check that this coincides with the previous definition when $Q$ is a dual Fano $T$-polygon (in that case the eigenray emanating from $\vv$ passes through the origin).

If $\vv\in Q$ is a $T$-vertex of type $\tfrac{1}{mr^2}(1,mra-1)$ with $m=1$, 
then it is straightforward to check that the mutated polytope $\mut_\vv(Q)$ no longer has a vertex at $\vv$, but there is 
a (possibly new) vertex at the point where the eigenray emanating from $\vv_\prim$ meets a side of $Q$. On the other hand, if 
$m>1$ then $\vv$ is still a vertex of $\mut_\vv(Q)$, but now with parameters $(m-1,r,a)$.
Thus for all $1\le k \le m$  
the \hl{$k$-fold mutation} $\mut^k_\vv(Q) := \psi^k(Q)$ of $Q$
at $\vv$ is well-defined, and we define the \hl{full mutation} of $Q$ at $\vv$ to be $\mut^\full_\vv(Q) := \mut^m_\vv(Q)$. We have:

\begin{lemma}\label{lem:full_mut_pres_sides}
If $\vv$ is a $T$-vertex of a polygon $Q \subset \MM_\R$ with $k$ sides, then $\mut^\full_\vv(Q)$ is a polygon with either $k$ or $k-1$ sides. In particular, if $Q$ is a triangle then so is $\mut^\full_\vv(Q)$.
\end{lemma}

It is shown in \cite[Lem. 7]{akhtar2016mirror} (building on \cite{ilten2012mutations}) that if two dual Fano polygons $Q,Q' \subset \MM_\R$ are mutation equivalent (i.e. $Q'$ can be obtained from $Q$ by a sequence of mutations), then the corresponding singular del Pezzo surfaces $V_Q$ and $V_{Q'}$ are $\Q$-Gorenstein deformation equivalent (in the sense of \cite[Def. 2]{akhtar2016mirror}). We discuss an approach to this in \S\ref{subsec:QG_pencils}. In particular, if $Q$ (and hence $Q'$) is also a $T$-polygon,  the smoothings $\wt{V}_Q$ and $\wt{V}_{Q'}$ as in Lemma~\ref{lem:QGorsmooth} are also $\Q$-Gorenstein deformation equivalent.
Conversely, \cite[Conj. A]{akhtar2016mirror} conjectures that two dual Fano polygons $Q,Q' \subset \MM_\R$ are mutation equivalent if the corresponding singular toric del Pezzos $V_Q,V_{Q'}$ are $\Q$-Gorenstein deformation equivalent.
The specialization of this conjecture to dual Fano $T$-polygons is proved in \cite[Thm. 1.2]{kasprzyk2017minimality}.
By the classification of smooth del Pezzo surfaces, it follows that there are precisely $10$ mutation equivalence classes of dual Fano $T$-polygons.
Representatives of these equivalence classes (or rather their duals) are shown in \cite[Fig. 1]{akhtar2016mirror}.
For the $6$ equivalence classes corresponding to rigid smooth del Pezzo surfaces, representatives which are particularly useful for constructing ellipsoid embeddings are shown in Figure~\ref{fig:smooth_Fano_polygons} (which is a reproduction of \cite[Fig. 5.12]{cristofaro2020infinite}). 
The construction of these representatives is based on almost toric techniques as in \cite{vianna2017infinitely}.

\begin{figure}
\centering
\includegraphics[scale=.85]{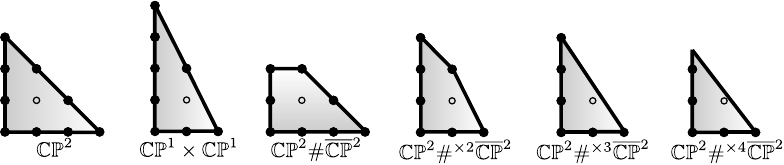}
\caption{
Dual Fano $T$-polygons $Q$ such that $\wt{V}_Q$ is a smooth rigid del Pezzo surface. Although these are not unique (due to the possibility of mutation), each representative has a Delzant vertex and the minimal possible number of sides.
Note that the polygon for $\CP^2 \#^{\times 4}\ovl{\CP}^2$ is not a lattice polygon (the top left vertex is $(-1,3/2)$.)}
\label{fig:smooth_Fano_polygons} 
\end{figure}

\sss

If $Q$ is a $T$-triangle, then by Lemma~\ref{lem:full_mut_pres_sides} it remains so under successive full mutations. 
As explained in \cite[\S H.2]{evans2023lectures}, if the vertices have types $\tfrac{1}{m_ir_i^2}(1,m_ir_ia_i-1)$ for $i=1,2,3$, then this data satisfies a generalized Markov equation
\begin{align}\label{eq:gen_Mark}
m_1r_1^2 + m_2r_2^2 + m_3r_3^2 = C\sqrt{m_1m_2m_3}\,r_1r_2r_3,
\end{align}
where $C \in \R_{>0}$ is invariant under full mutations (one can check that the set $\{m_1,m_2,m_3\}$ is also preserved under full mutations).
We thus have a bijection between
\begin{enumerate}[label=(\alph*)]
  \item the set of $T$-triangles which are full mutation equivalent to $Q$, and
  \item the set of triples $(r_1,r_2,r_3) \in \Z_{\geq 1}^3$ satisfying \eqref{eq:gen_Mark}.
\end{enumerate}
The specialization $r_1 = 1$ corresponds to triangles with a Delzant vertex (c.f. Figure~\ref{fig:smooth_Fano_polygons}).
We will see later that these taken together encode all of the relevant unicuspidal curves for the two-stranded rigid del Pezzo infinite staircases (c.f. Proposition~\ref{prop:inner_corners_vis_ell} and Proposition~\ref{prop:out_cor_vis}).

\begin{example}
When $Q$ is the first triangle in Figure~\ref{fig:smooth_Fano_polygons}, \eqref{eq:gen_Mark} becomes the classical Markov equation $r_1^2 + r_2^2 + r_3^2 = 3r_1r_2r_3$. The solutions are well-known to form an infinite trivalent tree, with edges corresponding to Markov mutations $(r_1,r_2,r_3) \mapsto (r_1,r_2,3r_1r_2-r_3)$ and their permutations (see e.g. \cite{aigner2015markov}).  
The solutions with $r_1 = 1$ correspond to pairs of consecutive odd index Fibonacci numbers. 
\end{example}

\subsection{Almost toric fibrations and polygons}\label{subsec:atfs_and_polys}
In this section we discuss symplectic almost toric fibrations from the point of view of $T$-polygons.
After some preliminaries on Lagrangian torus fibrations, we define almost toric fibrations, discuss their construction from $T$-polygons, and compare these with the $\Q$-Gorenstein smoothings from the previous subsection.

\subsubsection{Abstract almost toric fibrations and almost toric bases}\label{sec:alm_tor_bases}

By definition a closed symplectic manifold $M^{2n}$ is toric if it carries an effective Hamiltonian $\TT^n$-action.
In this case the image of the corresponding moment map $\pi: M \ra \R^n$ is a convex polytope $\pi(M)$ (see \cite{atiyah1982convexity,guillemin1982convexity}), such that $\pi$ is a regular Lagrangian torus fibration over the interior of $\pi(M)$.
Meanwhile, $\pi$ has toric singularities over the boundary of $\pi(M)$, with the fiber over a point in the interior of a $k$-dimensional face of $\pi(M)$ being an isotropic torus $\TT^{2(n-k)}$ of dimension $2(n-k)$.
Almost toric fibrations extend this picture (in dimension four) by allowing $\pi$ to have additional focus-focus singularities.

An \hl{almost toric fibration} (see e.g. \cite{symington71four,Leung-Symington,evans2023lectures}) is a smooth proper surjective map $\pi: M^4 \ra B^2$ with connected fibers, where $M^4$ is a symplectic
 four-manifold and $B^2$ is a smooth compact two-manifold with corners,
  such that
\begin{itemize}
  \item at each regular point $\pp \in M^4$ the kernel $\ker d_\pp \pi \subset T_\pp M$ is a Lagrangian subspace 
  \item for each critical point $\pp \in M$ of $\pi$ there are Darboux coordinates $x_1,y_1,x_2,y_2$ for $M$ near $\pp$ and smooth coordinates $b_1,b_2$ for $B$ near $\pi(\pp)$ such that $\pi$ has one of the following local normal forms:
  \begin{enumerate}
     \item $\pi(x_1,y_1,x_2,y_2) = (x_1,x_2^2+y_2^2)$ \;(corank one elliptic)
     \item $\pi(x_1,y_1,x_2,y_2) = (x_1^2+y_1^2,x_2^2+y_2^2)$ \;(corank two elliptic)
     \item $\pi(x_1,y_1,x_2,y_2) = (x_1y_1+x_2y_2,x_1y_2 - x_2y_1)$ \;(focus-focus).
   \end{enumerate}
\end{itemize}

Note that the regular fibers are necessarily two-dimensional tori by the Arnold--Liouville theorem.
The first two types of singularities are modeled on the singularities of a moment map of a toric symplectic manifold over an edge or vertex respectively of the moment polytope. The third type of singularity is topologically equivalent to a critical point of a Lefschetz fibration, but here the fibers are Lagrangian rather than symplectic. 
The images of focus-focus singularities are called \hl{base-nodes}.

Given an almost toric fibration $\pi: M^4 \ra B^2$, observe that $\pi$ restricts to a regular Lagrangian torus fibration over the regular values $B^\reg$ of $\pi$. 
Thus  $B^\reg$ inherits an integral affine structure as follows.
Given $\pp \in M$ and $\bb = \pi(\pp) \in B^\reg$, the symplectic form on $M$ induces a nondegenerate pairing
\begin{align}\label{eq:LTF_pairing}
\langle -,-\rangle: T_\bb B \times T^\vert_\pp M \ra \R, \;\;\;\;\; \langle u,v\rangle = \om(\wt{u},v),
 \end{align}
where $\wt{u} \in T_\pp M$ is any lift of $u \in T_\bb B$ (i.e. $\pi_*\wt{u} = u)$ and $T^\vert_\pp M = T_\pp \pi^{-1}(\bb)$ is the vertical tangent space at $\pp$.
In particular, for each covector in $T_\bb^*B$ there is a corresponding vector field along $\pi^{-1}(\bb)$, and taking its time-$1$ flow gives an action of $T_\bb^*B$ on the fiber $\pi^{-1}(\bb)$.
The set of covectors in $T^*_\bb B$ which act trivially on $\pi^{-1}(\bb)$ defines a lattice $\La_\bb^* \subset T_\bb^* B$, with dual lattice $\La_\bb \subset T_\bb B$.
The corresponding lattice bundle $\La^* = \bigcup\limits_{\bb \in B^\reg} \La^*_\bb \subset T^*B^{reg}$ gives a natural symplectomorphism 
\begin{align}\label{eq:symp_pres_of_reg_LTF}
T^*B^\reg / \La^* \cong M^\reg,
\end{align}
while the dual lattice bundle $\La = \bigcup\limits_{\bb \in B} \La_\bb \subset TB^{reg}$ defines the integral affine structure on $B^\reg$.

This integral affine structure on $B^\reg$ gives rise to an \hl{affine monodromy} map $\mon: \pi_1(B^\reg,*) \ra \aut(H_1(\pi^{-1}(*);\Z))$ which measures how the integral affine structure twists as we go around loops in $B^\reg$ (here $* \in B^\reg$ is a basepoint).
By picking a basis we can also view this as a map $\pi_1(B^\reg) \ra \gl_2(\Z)$ which is defined up to global conjugation by an element of $\gl_2(\Z)$.

For a small loop $\ga$ in $B^\reg$ surrounding a base-node $\bb_0$, the corresponding affine monodromy $\mon(\ga) \in \gl_2(\Z)$ is conjugate to the matrix
$\bigl(\begin{smallmatrix}
  1 & k \\ 0 & 1
\end{smallmatrix}\bigr)$
with eigenvalue $1$,
where $k$ is the number of focus-focus critical points in the fiber $\pi^{-1}(\bb_0)$.
Thus for any $\bb' \neq \bb_0$ in a small neighborhood of $\bb_0$ in $B$ there is a well-defined eigenline in $T_{\bb'}B$ for the affine monodromy around $\bb_0$,
and these limit to a line $\eig_{\bb_0} \subset T_{\bb_0}B$, which we call the \hl{eigenline at $\bb_0$}.

Recall that there exists a closed smooth toric symplectic manifold with moment polygon $Q \subset \R^n$ if and only if $Q$ is a Delzant polytope (\cite{delzant1988hamiltoniens}).
In dimension four we expect to associate almost toric fibrations $\pi: M^4 \ra B^2$ to more general polygons with non-Delzant vertices (even though $M$ is assumed to be smooth).
However, some care is needed in formulating the almost toric analogue of the moment polygon, since there is no global torus action and thus no global moment map.

Given an almost toric fibration $\pi: M^4 \ra B^2$, the set of regular values $B^\reg \subset B$ naturally inherits an integral affine structure, and this extends over the toric critical values to $B \setminus \{\bb_1,\dots,\bb_\ell\}$, where $\bb_1,\dots,\bb_\ell \in B$ are the base-nodes.
Thus we have a \hl{nodal integral affine surface}, i.e. a triple $(B,\{\bb_i\},\calA)$, where $B$ is a smooth two-dimensional manifold with corners equipped with a subset $\{\bb_1,\dots,\bb_\ell\} \subset B$ (the base-nodes) and an integral affine structure $\calA$ on $B \setminus \{\bb_1,\dots,\bb_\ell\}$, such that for each $i = 1,\dots,\ell$, a punctured neighborhood of $\bb_i$ is integral affine isomorphic to a punctured neighborhood of the origin in the model nodal integral structure $(V_{\op{mod}}^{k_i},\calA_{\op{mod}}^{k_i})$ from \cite[\S4.4]{symington71four}, for some nonzero $k_i \in \Z$.
In particular, this means that the affine monodromy around a small loop $\ga_i$ surrounding $\bb_i$ is conjugate to 
$\bigl(\begin{smallmatrix}
  1 & k_i \\ 0 & 1
\end{smallmatrix}\bigr)$.
Given an almost toric fibration $\pi: M^4 \ra B^2$, we will refer to the associated nodal integral affine surface $(B,\{\bb_i\},\calA)$ as its \hl{almost toric base}.

\begin{rmk}
 In slightly more detail, $(V_{\op{mod}}^k,\calA_{\op{mod}}^k)$ is defined as follows (see \cite[\S4.4 and \S5.1]{symington71four} for a fuller discussion).
 Pick a nonzero vector $u \in \R^2$, put $A := \bigl(\begin{smallmatrix}
  1 & 1 \\ 0 & 1
\end{smallmatrix}\bigr)$, and consider the sector in $\R^2$ bounded by the rays $\R_{\geq 0} \cdot u$ and $\R_{\geq 0} \cdot A^k u$ whose (interior) angle lies in $(\pi,2\pi]$.
Then $(V_{\op{mod}}^k,\calA_{\op{mod}}^k)$ is the integral affine manifold given by puncturing this sector at the origin and gluing together its two bounding rays via $A^k$. One can check that this is independent of $u$ up to integral affine isomorphism, and, according to \cite[Prop. 4.14]{symington71four}, any base-node of an almost toric fibration is locally of this form.
\end{rmk}

We will say that two nodal integral affine surfaces $(B,\{\bb_i\},\calA)$ and $(B',\{\bb_i'\},\calA')$ are isomorphic if there is a diffeomorphism $B \xrightarrow{\cong} B'$
which restricts to a bijection $\{\bb_i\} \xrightarrow{\cong} \{\bb_i'\}$ and preserves the integral affine structures on the complements of the base-nodes.
It is shown in \cite[Cor. 5.4]{symington71four} that if two almost toric fibrations $\pi_1: M_1 \ra B_1$ and $\pi_2: M_2 \ra B_2$ have isomorphic almost toric bases with nonempty boundaries then their total spaces are symplectomorphic.

Moreover, by \cite[Thm. 5.2]{symington71four}, a nodal integral affine surface $(B,\{\bb_i\},\calA)$ is the base of an almost toric fibration if and only if each point $\bb \in B \setminus \{\bb_i\}$ has a neighborhood which is integral affine isomorphic to a neighborhood of a point $\R_{\geq 0}^2$ (with its standard integral affine structure). 
In particular, given such a nodal integral affine surface $(B,\{\bb_i\},\calA)$ with $B$ compact, there is an associated closed symplectic four-manifold which we denote by $\atf(B,\{\bb_i\},\calA)$ and which is well-defined up to symplectomorphism.

\subsubsection{Nodal integral affine surfaces from $T$-polygons}\label{subsubsec:nodal_int_from_poly}

For the convenience of the reader, the beginning of this section (until Remark~\ref{rmk:nodal_trade}) reviews
the construction of
a nodal integral affine surface $Q_\nodal = (B,\{\bb_i\},\calA)$ from a $T$-polygon $Q \subset \MM_\R$; for more details see for example \cite[\S7.3-4,Rmk 8.8]{evans2023lectures}.
Here $B$ is the smooth surface with boundary given by smoothing the corners of the polygon $Q$, while $\calA$ is given roughly by implanting various nodes into the integral affine structure on $Q$ (i.e. the one induced from $\MM_\R$).
Together with the discussion in \S\ref{sec:alm_tor_bases}, this construction associates to any $T$-polygon $Q$ a closed symplectic four-manifold $\atf(Q_\nodal)$ which carries an almost toric fibration
\begin{align*}
\pi: \atf(Q_\nodal) \ra Q_\nodal.
\end{align*}
Strictly speaking $Q_\nodal$ depends on some auxiliary parameters giving the locations of the base-nodes (varying these is called a \hl{nodal slide}), but using Moser's stability theorem as in \cite[Thm.8.10]{evans2023lectures} one can show that $\atf(Q_\nodal)$ is independent of these choices up to symplectomorphism.

In more detail, we assume $\MM = \Z^2$ for simplicity, and let $Q \subset \R^2$ be a $T$-polygon with vertices $\vv_1,\dots,\vv_\ell$ ordered counterclockwise.
Since each vertex $\vv_i$ is a $T$-vertex, the corresponding toric fixed point has type $\tfrac{1}{m_ir_i^2}(1,m_ir_ia_i-1)$ for some $m_i \in \Z_{\geq 1}$ and coprime $r_i,a_i \in \Z_{\geq 1}$.
For $i =1,\dots,\ell$, let $-h_i \in \MM$ denote the primitive integral vector pointing in the direction of the eigenray emanating from $\vv_i$ (as in \S\ref{subsubsec:T-sings}).
Note that by our conventions $-h_i$ points inward towards $Q$.

For $i = 1,\dots,\ell$, pick $\eps_i > 0$ sufficiently small, and let $\ga_i \subset Q$ be the line segment ${\{\vv_i - th_i\;|\; t \in [0,\eps_i]\}}$.
Pick $\eps = t_i^1 > \cdots > t_i^{m_i} > 0$, and let $\bb_i^1,\dots,\bb_i^{m_i}$ be the corresponding points along $\ga_i$ (ordered towards $\vv_i$) given by $\bb_i^j := \vv_i - t_i^jh_i$.
Let $\shear_i \in \gl_2(\Z)$ be the primitive shear along $-h_i$ as in \S\ref{subsubsec:T-sings}, i.e. $\shear_i(u) = u + \det(-h_i,u) \cdot (-h_i)$, and let $\tau_i: \R^2 \ra \R^2$ be any translation sending a point on the eigenray emanating from $\vv_i$ to the origin.  
 Let $s_i^1,\dots,s_i^{m_i}$ be the components of $\ga_i \setminus \{\bb_i^1,\dots,\bb_i^{m_i}\}$ ordered towards $\vv_i$.

We now modify the integral affine surface $Q \setminus \{\bb_i^j\}$ to obtain a new one by cutting $Q$ along $s_i^j$ and regluing via the transformation $\tau_i \circ \shear_i^j \circ \tau_i^{-1} \in \gl_2(\Z)$ for $i = 1,\dots,\ell$ and $j =1,\dots,m_i$ (here $\shear_i^j$ denotes the $j$th power of $\shear_i$).
More precisely, if $(s_i^j)^-,(s_i^j)^+$ are the boundary segments of $Q \setminus \bigcup\limits_{i=1}^\ell \ga_i$ (in counterclockwise order) arising from cutting along $s_i^j$, then the gluing identifies $u \in (s_i^j)^+$ with $(\tau_i \circ \shear_i^j\circ \tau_i^{-1})(u) \in (s_i^j)^-$.
Since $\shear_i^j$ fixes $\ga_i$ pointwise, the resulting topological space $B_\nodal$ is naturally identified with $Q \setminus \{\bb_i^j\}$, but by construction it carries an integral affine structure $\calA$ which has affine monodromy around each $\bb_i^j$ conjugate to 
$\bigl(\begin{smallmatrix}
  1 & 1 \\ 0 & 1
\end{smallmatrix}\bigr)$.

\begin{lemma}\label{lem:Q_nodal_loc_tor}
The glued integral affine structure $\calA$ on $Q \setminus \{\bb_i^j\}$ 
is locally isomorphic to $\R^2$ near any interior point $\bb \in \Int Q \setminus \{\bb_i^j\}$, and it is locally isomorphic to a boundary point of $\R \times \R_{\geq 0}$ near any point in $\bdy Q$. In particular, the smooth structure on $B_\nodal$ is such that there are no corner points.
\end{lemma}
\NI To complete the construction of $Q_\nodal$, we let $B$ be the smooth surface with boundary given by filling in the punctures of $B_\nodal$. Note that $B$ is diffeomorphic to a two-dimensional closed disk.

\begin{proof}[Proof of Lemma~\ref{lem:Q_nodal_loc_tor}]
  
It suffices to analyze the integral affine structure near a vertex $\vv_i$. After an integral affine transformation we can assume that $\vv_i = (0,0)$ with incoming edge vector $(0,-1)$ and outgoing edge vector $(m_ir_i^2,m_ir_ia_i-1)$, and $-h_i = (r_i,a_i)$.
Let $C \subset \R_{\geq 0}^2$ be the cone spanned by $(0,1),(m_ir_i^2,m_ir_ia_i-1)$, and let $C^-,C^+ \subset C$ be the subcones spanned by $(0,1),(r_i,a_i)$ and $(r_i,a_i),(m_ir_i^2,m_ir_ia_i-1)$ respectively.
Then we have an integral affine isomorphism from a neighborhood of $\vv_i$ in $B_\nodal$ to a neighborhood of $(0,0)$ in $\R_{\geq 0} \times \R$, given by
\begin{align*}
u \mapsto 
\begin{cases}
  u & u \in C^- \\
  \shear_i^{m_i}(u) & u \in C^+.
\end{cases}
\end{align*}
Indeed, it suffices to check that the vector $(m_ir_i^2,m_ir_ia_i-1)$ gets sent to $(0,-1)$, i.e. $\shear_i^{m_i}(m_ir_i^2,m_ir_ia_i-1) = (0,-1)$, and this is the content of the following lemma.
\end{proof}

\begin{lemma}
  If $\shear \in \gl_2(\Z)$ is the primitive shear along $(r,a)$, we have $\shear^{m}(mr^2,mra-1) = (0,-1)$. 
\end{lemma}
\begin{proof}
Using \eqref{eq:prim_shear} we have $\shear^{m}(x,y) = (x,y) + m\det((r,a),(x,y)) \cdot (r,a)$, from which the result directly follows.
\end{proof}

\begin{figure}[ht]
    \centering
    \begin{minipage}[b]{0.29\linewidth}
        \centering

\begin{tikzpicture}[scale=1,cross/.style={cross out, draw, minimum size=2*(#1-\pgflinewidth), inner sep=0pt, outer sep=0pt}, cross/.default={3pt}]

    \coordinate (A) at (-1,-1);
    \coordinate (B) at (3,-1);
    \coordinate (C) at (-1,1);
    \coordinate (O) at (0,0);

    \draw (A) -- (B) -- (C) -- cycle;

    \node[below] at (A) {\scriptsize(-1,-1)};
    \node[below] at (B) {\scriptsize (3,-1)};
    \node[above] at (C) {\scriptsize (-1,1)};

    \coordinate (A1) at ($(A)+.5*(1,1)$);
    \coordinate (B1) at ($(B)+.7*(-3,+1)$);
    \coordinate (C1) at ($(C)+.4*(1,-1)$);    
    \coordinate (C2) at ($(C)+.75*(1,-1)$);

    \draw[dashed] (B) -- (B1);
    \draw[dashed] (C) -- (C1);
    \draw[dashed] (C1) -- (C2);    

    \draw[line width=.8pt] ($(B1)-(2pt,2pt)$) -- ++(4pt,4pt);
    \draw[line width=.8pt] ($(B1)-(2pt,-2pt)$) -- ++(4pt,-4pt);

    \draw[line width=.8pt] ($(C1)-(2pt,2pt)$) -- ++(4pt,4pt);
    \draw[line width=.8pt] ($(C1)-(2pt,-2pt)$) -- ++(4pt,-4pt);

    \draw[line width=.8pt] ($(C2)-(2pt,2pt)$) -- ++(4pt,4pt);
    \draw[line width=.8pt] ($(C2)-(2pt,-2pt)$) -- ++(4pt,-4pt);

    \node[below] at ($(B1)+0.05*(0,-1)$) {\scriptsize $\bb_1^1$}; 
    \node[below] at ($(C1)+0.025*(-4,-1)$) {\scriptsize $\bb_2^2$}; 
    \node[below] at ($(C2)+0.025*(-4,-1)$) {\scriptsize $\bb_2^1$};     

    \fill (0,0) circle (1pt); 
    \node[anchor=west] at (0,0) {\scriptsize (0,0)}; 

        \end{tikzpicture}

    \end{minipage}
    \hfill
    \begin{minipage}[b]{0.29\linewidth}
        \centering

\begin{tikzpicture}[scale=1,cross/.style={cross out, draw, minimum size=2*(#1-\pgflinewidth), inner sep=0pt, outer sep=0pt}, cross/.default={3pt}]

    \coordinate (A) at (-1,-1);
    \coordinate (B) at (3,-1);
    \coordinate (C) at (-1,1);
    \coordinate (O) at (0,0);

    \draw (A) -- (B) -- (C) -- cycle;

    \node[below] at (A) {\scriptsize (-1,-1)};
    \node[below] at (B) {\scriptsize (3,-1)};
    \node[above] at (C) {\scriptsize (-1,1)};

    \coordinate (A1) at ($(A)+.5*(1,1)$);
    \coordinate (B1) at ($(B)+.4*(-3,+1)$);
    \coordinate (C1) at ($(C)+.35*(1,-1)$);    
    \coordinate (C2) at ($(C)+.7*(1,-1)$);

    \draw[dashed] (B) -- (B1);
    \draw[dashed] (C) -- (C1);
    \draw[dashed] (C1) -- (C2);    

    \draw[line width=1.5pt,color=violet] (O) -- (-1,0);    
    \draw[line width=1.5pt,color=violet] (O) -- (0,-1);        
    \draw[line width=1.5pt,color=violet] (O) -- (1/5,2/5);            

    \node[below] at ($(O)!0.5!(-1,0)$) {\scriptsize (1)};
    \node[right] at ($(O)!0.5!(0,-1)$) {\scriptsize (2)};
    \node[below right] at ($(O)!0.8!(1/5,2/5)$) {\scriptsize (1)};        

    \draw[line width=.8pt] ($(B1)-(2pt,2pt)$) -- ++(4pt,4pt);
    \draw[line width=.8pt] ($(B1)-(2pt,-2pt)$) -- ++(4pt,-4pt);

    \draw[line width=.8pt] ($(C1)-(2pt,2pt)$) -- ++(4pt,4pt);
    \draw[line width=.8pt] ($(C1)-(2pt,-2pt)$) -- ++(4pt,-4pt);

    \draw[line width=.8pt] ($(C2)-(2pt,2pt)$) -- ++(4pt,4pt);
    \draw[line width=.8pt] ($(C2)-(2pt,-2pt)$) -- ++(4pt,-4pt);
  
    \fill (0,0) circle (1pt);

        \end{tikzpicture}

    \end{minipage}
    \hfill 
        \begin{minipage}[b]{0.29\linewidth}
        \centering

\begin{tikzpicture}[scale=1,cross/.style={cross out, draw, minimum size=2*(#1-\pgflinewidth), inner sep=0pt, outer sep=0pt}, cross/.default={3pt}]

    \fill[fill=orange!30, fill opacity=1] (-1,-1) -- (-1,1/3) -- (3,-1) -- cycle;

    \coordinate (A) at (-1,-1);
    \coordinate (B) at (3,-1);
    \coordinate (C) at (-1,1);
    \coordinate (O) at (0,0);

    \draw (A) -- (B) -- (C) -- cycle;

    \node[below] at (A) {\scriptsize (-1,-1)};
    \node[below] at (B) {\scriptsize (3,-1)};
    \node[above] at (C) {\scriptsize (-1,1)};

    \coordinate (A1) at ($(A)+.5*(1,1)$);
    \coordinate (B1) at ($(-3/5,1/5)$);
    \coordinate (C1) at ($(C)+.4*(1,-1)$);    
    \coordinate (C2) at ($(C)+.75*(1,-1)$);

    \draw[dashed] (B) -- (B1);
    \draw[dashed] (C) -- (C1);
    \draw[dashed] (C1) -- (C2);    
    \draw[line width=1.5pt,blue] (A) -- (B1);

    \draw[line width=.8pt] ($(B1)-(2pt,2pt)$) -- ++(4pt,4pt);
    \draw[line width=.8pt] ($(B1)-(2pt,-2pt)$) -- ++(4pt,-4pt);

    \draw[line width=.8pt] ($(C1)-(2pt,2pt)$) -- ++(4pt,4pt);
    \draw[line width=.8pt] ($(C1)-(2pt,-2pt)$) -- ++(4pt,-4pt);

    \draw[line width=.8pt] ($(C2)-(2pt,2pt)$) -- ++(4pt,4pt);
    \draw[line width=.8pt] ($(C2)-(2pt,-2pt)$) -- ++(4pt,-4pt);

        \end{tikzpicture}
    \end{minipage}
    
    \caption{Left: $Q_\nodal$, where $Q$ is the triangle with vertices $\vv = (3,-1), \vv_2 = (-1,1), \vv_3 = (-1,-1)$. Middle: a tropical representation of a curve intersecting each toric divisor once as in Corollary~\ref{cor:curve_hitting_each_edge_once}. Right: a visible symplectic unicuspidal curve as in \S\ref{subsubsec:vis_Lag_symp}. The outer corner obstruction given by this curve implies that the shaded triangle represents an optimal ellipsoid embedding.}
    \label{fig:ATF_3_pics}
\end{figure}
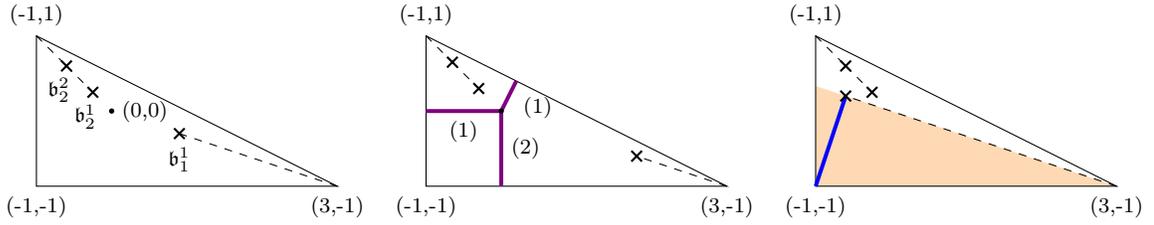

\begin{rmk}\label{rmk:nodal_trade}
Gluing in a node at a Delzant vertex is called a \hl{nodal trade} and it does not change the resulting symplectic four-manifold up to symplectomorphism (see e.g. \cite[Thm. 6.5]{symington71four} or \cite[\S8.2]{evans2023lectures}).
 In the above construction of $Q_\nodal$ we have chosen to glue in nodes at all of the Delzant vertices of $Q$ (i.e. those of type $\tfrac{1}{mr^2}(1,mra-1)$ with $m=r=1$) only for uniformity of exposition, but we can equally well leave the integral affine structure alone near some or all of the Delzant vertices. When constructing sesquicuspidal curves it will be beneficial to have one Delzant vertex without a nodal trade.
\end{rmk}

\begin{rmk}\label{rmk:nodal_ray}
For an almost toric fibration $\pi: \atf(Q_\nodal) \ra Q_\nodal$ as above, the polygon $Q$ decorated by its base-nodes $\{\bb_i^j\}$ and the eigenrays at its vertices is sometimes called an \hl{almost toric base diagram}. Our approach here is to keep track of only the $T$-polygon $Q \subset \MM_\R$, since the eigenrays and number of base-nodes are uniquely determined by $Q$ and the locations of the base-nodes are immaterial up to symplectomorphism of the total space.  

In \S\ref{sec:singII}, we construct $(p,q)$-unicuspidal curves in $\atf(Q_\nodal)$ 
in $T$-polygons that have one smooth vertex. The types $(p,q)$ of their cusps do not depend simply on the types
 $\frac1{q^2}(1, pq-1)$ of the $T$-singularities of the vertices of $Q$ (where $p$ is only well-defined mod $q$), but rather on  the relation between  the eigenrays at the vertices of $Q$ and the smooth corner at the origin.  For a vertex of type
$\frac1{q^2}(1, pq-1)$ on the $y$-axis the adjacent edges of $Q$ have directions $(0,-1), (q^2, 1- pq)$ with eigenray $(q,-p)$.
Recall also from \S\ref{subsubsec:T-sings} that eigenrays are invariant under the reflection that interchanges its two adjacent edges. 
\end{rmk}

\begin{example}
  Figure~\ref{fig:ATF_3_pics} left illustrates $Q_\nodal$ in the case that $Q$ is a triangle with vertices $\vv_1 = (3,-1), \vv_2 = (-1,1), \vv_3 = (-1,-1)$. The vertices are:
  \begin{itemize}
    \item smooth at $(3,-1)$, with a nodal trade
    \item type $\tfrac{1}{2}(1,1)$ at $(-1,1)$
    \item smooth at $(-1,-1)$, without a nodal trade.
  \end{itemize}
Here we have $-h_1 = (-3,1)$ and $-h_2 = (1,-1)$.
In this example the symplectic four-manifold $\atf(Q_\nodal)$ is symplectomorphic to $\CP^1(2) \times \CP^1(2)$.
\end{example}

\sss

We can equivalently describe $\atf(Q_\nodal)$ by starting with the symplectic toric orbifold with moment polygon $Q$ and performing a cut-and-paste operation near each toric fixed point.
Namely, near the $i$th toric fixed point we excise a neighborhood which is symplectomorphic to a neighborhood in $\tfrac{1}{m_ir_i^2}(1,m_ir_ia_i-1) \cong \{xy = z^{r_im_i}\} / \roots_{r_i}^{1,-1,a_i}$ and we glue in a neighborhood in $B_{m_i,r_i,a_i} = \{xy = (z^{r_i}-\zeta_1)\cdots(z^{r_i}-\zeta_{m_i})\} / \roots_{r_i}^{1,-1,a_i}$ (c.f. \S\ref{subsubsec:T-sings}).
Since the group actions above are unitary, these spaces inherit symplectic forms from the standard one on affine space.
Here $B_{m,r,a}$ carries the Auroux-type almost toric fibration
\begin{align}\label{eq:pi_aur}
\pi_\aur: B_{m,r,a} \ra \C,\;\;\;\;\; \pi_\aur(x,y,z) = \left(|z|^2,\tfrac{1}{2}|x|^2 - \tfrac{1}{2}|y|^2 \right),
\end{align}
which has $m_i$ focus-focus critical points mapping to distinct base-nodes.
Note that $\pi_\aur$  does indeed descend to the quotient by $\roots_r^{(1,-1,a)}$ (see \cite[\S7.4]{evans2023lectures} and \S\ref{subsec:aur_vis} below).

By comparing the local description of $\Q$-Gorenstein smoothings of $T$-singularities as in \eqref{eq:T_sing_miniversal} with the above cut-and-paste description of $\atf(Q_\nodal)$, one can show that, for any sufficiently close $\Q$-Gorenstein smoothing $\wt{V}_Q$ of $V_Q$, there is a diffeomorphism $\Phi: \wt{V}_Q \ra \atf(Q_\nodal)$ such that $\Phi_*(\wt{J})$ tames the symplectic form on $\atf(Q_\nodal)$, where $\wt{J}$ is the integrable almost complex structure on $\wt{V}_Q$.
More explicitly, we have:
\begin{prop}\label{prop:QG_def_diff_ATF}
Let $Q$ be a $T$-polygon with corresponding (singular) toric surface $V_Q$.
Let $\pi: \calX \ra \D_\eps$ be a proper holomorphic map, where $\calX$ is a complex manifold and $\D_\eps := \{|z| < \eps\}$ denotes the open disk in $\C$ of radius $\eps > 0$.
Assume that the family $\pi$ is flat, $\Q$-Gorenstein, and submersive over $\D_\eps \setminus \{0\}$, with $X_0 := \pi^{-1}(0)$ biholomorphic to $V_Q$ and $X_t := \pi^{-1}(t)$ nonsingular for all $t \neq 0$.
Then, for all $|t| > 0$ sufficiently small, there exists a diffeomorphism $\Phi_t: X_t \ra \atf(Q_\nodal)$ such that $(\Phi_t)_*(J_{t})$ tames the symplectic form on $\atf(Q_\nodal)$, where $J_{t}$ denotes the integrable almost complex structure induced on $X_t$ as a submanifold of $\calX$.
\end{prop}

\begin{proof}[Sketch of proof]
Let $\pp_1,\dots,\pp_k$ denote the singular points of $X_0$, where $\pp_i$ is a $T$-singularity of the form $\tfrac{1}{m_ir_i^2}(1,m_ir_ia_i-1)$ for $i = 1,\dots,k$.
By miniversality of the family \eqref{eq:T_sing_miniversal}, we can find some $\de > 0$ such that, for each $i = 1 ,\dots,k$ and all $|t| < \de$, we have
\begin{itemize}
  \item open subsets $U_i^t \subset X_t$, where $U_i^t \cap U_j^t = \nil$ for $i \neq j$ and $U_i^0 \cap \{\pp_1,\dots,\pp_k\} = \{\pp_i\}$, such that the closure of $U_i^t$ has smooth boundary which varies smoothly with $t$
    \item complex affine varieties $B_i^t = \{xy = z^{r_im} + C_{i,m-2}^{t}z^{r_i(m-2)} + \cdots + C_{i,1}^{t}z^{r_i} + C_{i,0}^{t}\} \subset \C^3/\roots_{r_i}^{1,-1,a_i}$  for some parameters $C_{i,j}^{t} \in \C$ which vary holomorphically with $t$ and vanish for $t = 0$, such that $B_i^t$ is nonsingular for $t \neq 0$
\item holomorphic embeddings $\iota_i$ from $\calU_i := \bigcup\limits_{|t| < \de} U_i^t \subset \calX$ into $\C^3 / \roots_{r_i}^{1,-1,a_i} \times \D_\de$ which restrict to embeddings $\iota_i^t: U_i^t \hookrightarrow B_i^t$. 
\end{itemize}

Let us equip $X_0$ with the compatible symplectic form coming from its identification with $V_Q$ (strictly speaking as an orbifold, but we can ignore the behavior near the singularities), and equip $\calU_i$ with the symplectic structure induced from its embedding into $\C^3 / \roots_{r_i}^{1,-1,a_i} \times \D_\de$.
We consider the resulting symplectic connection on $\calU_1 \cup \cdots \cup \calU_k$ (i.e. whose horizontal tangent spaces are symplectically orthogonal to the vertical tangent spaces) and extend it to an arbitrary connection on $\calX$.
Fix small open subsets $R_i^0,S_i^0 \subset X_0$ satisfying $\pp_i \in R_i^0 \Subset S_i^0 \subset U_i^0$ for $i = 1,\dots,k$,
and let $R_i^t \Subset S_i^t \subset X_t$ denote the neighborhoods in $X_t$ corresponding to $R_i^0 \Subset S_i^0 \subset X_0$ under parallel transport (here $\Subset$ denotes inclusion as a precompact subset).
Note that parallel transport then induces diffeomorphisms 
\[\Psi_t: X_0 \setminus \left( R_1^0 \cup \cdots \cup R_k^0\right ) \cong X_t \setminus \left( R_1^t \cup \cdots \cup R_k^t\right ),\]
and by construction $\Psi_t$ restricts to a symplectic embedding $S_i^0 \setminus R_i^0 \hookrightarrow S_i^t \setminus R_i^t$ for each $i = 1,\dots,k$ and $|t| < \de$.

Observe that, up to symplectomorphism, we have
\begin{align*}
\atf(Q_\nodal) = X_0 \setminus \left(R_1^0 \cup \cdots \cup R_k^0\right) \underset{\Psi_t}{\cup} \left(R_1^t \cup \cdots \cup R_k^t\right),
\end{align*}
 where the gluing uses $\Psi_t$ to symplectically identify $S_i^0 \setminus R_i^0$ with $S_i^t \setminus R_i^t$ for $i=1,\dots,k$.
This follows essentially from the above cut-and-paste description of $\atf(Q_\nodal)$, noting that $B_i^t$ is symplectomorphic to $B_{m_i,r_i,a_i}$ as defined in \eqref{eq:B_m_r_a} and it carries a natural Auroux-type almost toric fibration similar to \eqref{eq:pi_aur}. 

We now define the diffeomorphism 
$$\Phi_t: X_t \ra X_0 \setminus (R_1^0 \cup \cdots \cup R_k^0) \underset{\Psi_t}{\cup} (R_1^t \cup \cdots \cup R_k^t)$$
to be $\Psi_t^{-1}$ on $X_t \setminus (R_1^t \cup \cdots \cup R_k^t)$,
and the identity on $R_1^t \cup \cdots \cup R_k^t$.
To see that $(\Phi_t)_*J_t$ tames the symplectic form on the target, note that since $J_{0}$ is compatible with the symplectic form on $X_0$ and tameness is an open condition, we can assume (possibly after further shrinking $\de$) that 
$(\Psi_t)^*(J_{t})$ tames the symplectic form on $X_0 \setminus (R_1^0 \cup \cdots \cup R_k^0)$ for all $|t| < \de$.
Meanwhile, $J_{t}$ is compatible with the symplectic form on $R_i^t$ for $i = 1,\dots,k$, both structures being induced by the standard ones on $\C^3$.
\end{proof}

\subsubsection{Mutations of almost toric fibrations}\label{subsubsec:ATF_mutations}

Let $Q \subset \MM_\R$ be a $T$-polygon.
For almost toric fibrations of the form $\atf(Q_\nodal)$ as in \S\ref{subsubsec:nodal_int_from_poly}, mutating $Q$ at a vertex $\vv$ as in \S\ref{subsubsec:T-sings} recovers the familiar notion of mutation for almost toric fibrations (see e.g. \cite[\S8.4]{evans2023lectures}).
One can check that, up to nodal slides (i.e. moving the base-nodes), $(\mut_\vv(Q))_\nodal$ and $Q_\nodal$ are isomorphic nodal integral affine polygons, differing in only their presentations in terms of polygons with branch cuts (roughly speaking, a primitive mutation at a vertex $\vv$ corresponds to rotating the branch cut at one of the nearby base-nodes by $180$ degrees).
In particular, we have:
\begin{prop}\label{prop:mut_imply_symp}
If $Q, Q' \subset \MM_\R$ are mutation equivalent $T$-polygons, then the symplectic four-manifolds $\atf(Q_\nodal)$ and $\atf(Q'_\nodal)$ are symplectomorphic.
\end{prop}

\section{Singular algebraic curves in almost toric manifolds I}\label{sec:singI}

In this section, we first discuss in \S\ref{subsec:vis_geom} various geometric features of a symplectic four-manifold which are ``visible'' from the base of an almost toric fibration, such as (singular) symplectic and Lagrangian subspaces and ellipsoid embeddings. Then, in \S\ref{subsec:rat_curves_tor_surf} we construct some explicit rational holomorphic curves with prescribed cusp singularities in (singular) toric surfaces.
 Finally, in \S\ref{subsec:inflatable_curves}, we explain how to transport these holomorphic curves into symplectic rigid del Pezzo surfaces in order to construct inner corner curves and prove Theorem~\ref{thmlet:inner_corner_curves}.

\subsection{Visible geometry in almost toric fibrations}\label{subsec:vis_geom}

\subsubsection{Visible Lagrangians and symplectic curves}\label{subsubsec:vis_Lag_symp}

Suppose that $\pi: M^4 \ra B^2$ is a four-dimensional regular Lagrangian torus fibration,
and let $C^2 \subset M^4$ be a compact 
two-dimensional submanifold which projects to a regular path $\ga \subset B^2$.
Put $C_\bb := C \cap \pi^{-1}(\bb)$ for each $\bb \in B$. Let $\lan -,-\ran$ denote the pairing from \eqref{eq:LTF_pairing}.
The following is readily checked using the symplectomorphism \eqref{eq:symp_pres_of_reg_LTF}:
\begin{itemize}
  \item $C$ is Lagrangian if and only if $\langle u,v\rangle = 0$ for any $\pp \in C$, $u \in T_{\pi(\pp)}\ga$, and $v \in T_\pp C_{\pi(\pp)}$
  \item $C$ is symplectic if and only if $\langle u,v\rangle \neq 0$ for any $\pp \in C$, $u \in T_{\pi(\pp)}\ga$, and $v \in T_\pp C_{\pi(\pp)}$. 
\end{itemize}
In these situations we will say that $C$ is a \hl{visible} Lagrangian or symplectic submanifold of $M$.
For $\bb \in \ga$, we will say that the fiber $C_\bb$ is \hl{straight} if it is an orbit of the action of $T_\bb^* B$ on $\pi^{-1}(\bb)$.
For a visible Lagrangian $C$, it is easy to check that each fiber $C_\bb$ is straight, and thus each tangent vector to $\ga$ is a multiple of a lattice vector, so $\ga$ is also straight with respect to the integral affine structure on $B$.

For a visible symplectic curve $C$ the fibers $C_\bb$ need not be straight, but at each point $\bb \in \ga$ there is a unique (up to sign) primitive covector $\al_\bb \in \La_\bb^*$ such that $C_\bb$ is homologous to an orbit of the vector field along $\pi^{-1}(\bb)$ corresponding to $\al_\bb$ (c.f. the discussion after \eqref{eq:LTF_pairing}).
Thus for a visible symplectic curve $C$ there is a section $\al$ of the pullback of $\La^*$ along $\ga$, such that $\al_\bb(v) \neq 0$ for every $\bb \in \ga$ and $0 \neq v \in T_\bb \ga$.
We will refer to the pair $(\ga,\al)$ as a \hl{covector-decorated path}. 
For future reference, we note that the symplectic area of the visible symplectic curve $C$ is naturally computed by  integrating the covector field $\al$ along $\ga$:
\begin{align}\label{eq:area_vis_symp}
\op{area}(C) = \int \al(\ga'(t))dt.
\end{align}

\begin{example}
Suppose that $B$ is $\R^2$ with its standard integral affine structure and coordinates $x_1,x_2$, and let $\ga: (a,b) \mapsto \R^2$ be the straight line segment $t \ra (pt,qt)$ for some primitive lattice vector $(p,q) \in \Z^2$.
 Then we can take $\al$ to be the constant covector field  $pdx_1 + qdx_2$, and the area of the corresponding visible symplectic cylinder $C_{(\ga,\al)} \subset T^*\mathbb{T}^2$ is $(b-a)(p^2+q^2)$.
\end{example}

One can extend the above discussion to define visible Lagrangian and symplectic submanifolds in an almost toric fibration $\pi: M^4 \ra B^2$.
Over $B^\reg$ the situation is identical, and with some care we can also allow paths in $B$ which end on toric boundary points or base-nodes.
Roughly, as $\bb' \in B^\reg$ approaches an interior point $\bb$ of an edge $e$ of a moment polygon, the torus fibers $\pi^{-1}(\bb')$ collapse to circles along a direction determined by $e$, and, similarly, as $\bb'$ approaches a Delzant vertex $\vv$ (having no base-nodes) the torus fibers collapse to a point.\footnote{Note that a general corank one or two elliptic value $\bb \in B$ is locally integral affine isomorphic to one of these situations.} Meanwhile, as $\bb' \in B^\reg$ approaches a base-node $\bb$, the torus fibers $\pi^{-1}(\bb')$ get pinched along a circle which depends on the eigenline $\eig_\bb \in T_\bb B$.
In slightly more detail, we have the following building blocks:

\begin{enumerate}[label=(\roman*)]
  \item a smooth Lagrangian disk over any straight line segment $\ga$ ending on a base-node $\bb$, provided that $\ga$ is tangent to the eigenline $\eig_\bb$ at $\bb$
  \item a symplectic disk over any covector-decorated path $(\ga,\al)$ ending perpendicularly to  an interior point $\bb$ of an edge $e$, provided that the covector $\al_\bb \in T_\bb^*B$ vanishes on the tangent space $T_\bb e$ to the edge
  \item a symplectic disk over any covector-decorated path $(\ga,\al)$ ending on a base-node $\bb$ perdendicularly to its eigenline $\eig_\bb$, provided that the covector $\al_\bb$ vanishes on $\eig_\bb$
  \item a $(p,q)$-unicuspidal  symplectic disk over any covector-decorated path $(\ga,\al)$ that has one end  on a Delzant vertex $\vv$,
  where $p$ and $q$ are obtained by evaluating $\al$ on the primitive tangent vectors to the edges adjacent to $\vv$ (see Example~\ref{ex:311}).
\end{enumerate}
\NI 
The symplectic disks in (ii), (iii) are smooth if the paths are straight near their endpoints.
Other local models for visible singular Lagrangians are also discussed e.g. in \cite[\S5.1 and \S6.4]{evans2018markov}.

Let us now specialize to the case of an almost toric fibration $\pi: \atf(Q_\nodal) \ra Q_\nodal$ associated to a $T$-polygon $Q \subset \R^2$ as in \S\ref{subsubsec:nodal_int_from_poly}.
Here we assume that the vertex $\vv_i$ of $Q$ is of type $\tfrac{1}{m_ir_i^2}(1,m_ir_ia_i-1)$.
Note that a path in $Q_\nodal$ is 
smooth  in the usual sense in $\R^2$, except that whenever it crosses some nodal line it bends by the appropriate shear.
In this case we have for example:
\begin{enumerate}[label=(\alph*)]
  \item\label{item:vis_Lag_sphere} a visible Lagrangian two-sphere $\LL_{[\bb_i^j,\bb_i^{j+1}]}$ over the line segment $[\bb_i^j,\bb_i^{j+1}] \subset Q_\nodal$ for each $i = 1,\dots,\ell$ and $j=1,\dots,m_i-1$ (c.f. \cite[Fig. 7.8]{evans2023lectures})
  \item\label{item:vis_Lag_pin} a visible \hl{Lagrangian $(r_i,a_i)$-pinwheel} $\LL_{[\bb_i^{m_i},\vv_i]}$ (see \cite[Def. 3.1]{khodorovskiy2013symplectic}) over the line segment $[\bb_i^{m_i},\vv_i]$ for each $i=1,\dots,\ell$ (c.f. again \cite[Fig. 7.8]{evans2023lectures})
  \item\label{item:vis_symp_inner} a visible symplectic two-sphere $C$ with $c_1(C) = 2$ over any straight line segment in $Q_\nodal \setminus \{\bb_i^j\}$ with both ends ending perpendicularly on interior points of edges (see the violet path in Figure~\ref{fig:vis_curves} left)
  \item\label{item:vis_symp_outer} a visible nonsingular symplectic two-sphere  $C$ with $c_1(C) = 1$  over any straight line segment $\ga$ in $Q_\nodal$ with one end ending perpendicularly on an interior point of an edge and the other end ending on a base-node perpendicularly to the associated eigenline (see the blue path in Figure~\ref{fig:vis_curves} right) 
  \item\label{item:vis_symp_uni} a visible $(p,q)$-unicuspidal symplectic two-sphere $C_{p,q}$ with $c_1(C) = p+q$ over any straight line segment $\ga$ in $Q_\nodal$ of slope $p/q$ with one end at a Delzant vertex (having no base-nodes) and the other end ending on a base-node perpendicularly to the associated eigenray (see the blue path in Figure~\ref{fig:ATF_3_pics} right). 
\end{enumerate}

The above claims about the first Chern class $c_1(C)$ are justified by work of Symington, quoted in Lemma~\ref{lem:c_1_is_bdy} below.
As illustrated in Figure~\ref{fig:vis_curves} below (see also  Figure~\ref{fig:toric_blowup_seq}), the normal crossing resolution $\wt{C}_{p,q}$ of the cuspidal curve $C_{p,q}$ is a curve of type (d) that  is visible in an appropriate blowup of $\atf(Q_\nodal)$ at its smooth point.
Note also that a Lagrangian $(r,a)$-pinwheel is homeomorphic to the closed unit two-disk $\ovl{\D}^2$ with its boundary quotiented out by the equivalence relation $z \sim e^{2\pi a\sqrt{-1}/r}z$ for $z \in \bdy \ovl{\D}^2$. In particular, a Lagrangian $(1,1)$-pinwheel is just an embedded Lagrangian disk and a Lagrangian $(2,1)$-pinwheel is an embedded Lagrangian real projective plane $\mathbb{RP}^2$.
The visible Lagrangians \ref{item:vis_Lag_sphere} and \ref{item:vis_Lag_pin} will be useful for describing the symplectic form on $\atf(Q_\nodal)$ in the sequel, while the visible symplectic curves \ref{item:vis_symp_inner}, \ref{item:vis_symp_outer}, and \ref{item:vis_symp_uni} give symplectic analogues of some of the algebraic curves which we construct in \S\ref{sec:singI} and \S\ref{sec:singII} respectively.

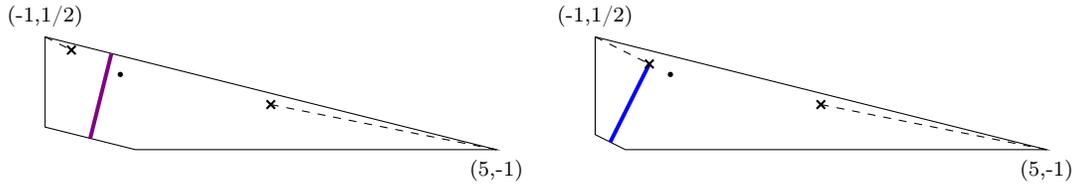
\begin{figure}[ht]
    \centering
    \begin{minipage}[b]{0.49\linewidth}
        \centering

\begin{tikzpicture}[cross/.style={cross out, draw, minimum size=2*(#1-\pgflinewidth), inner sep=0pt, outer sep=0pt}, cross/.default={3pt}]

    \coordinate (A) at (-1,-1);
    \coordinate (B) at (5,-1);
    \coordinate (C) at (-1,1/2);
    \coordinate (O) at (0,0);

    \def\epsvar{0.2} 

    \coordinate (Ox) at ($(A)!\epsvar!(B)$);
    \coordinate (Oy) at ($(A)!\epsvar!(C)$);

    \draw (Oy) -- (Ox) -- (B) -- (C) -- cycle;

    \coordinate (D) at (${1-\epsvar}*(A) + .5*\epsvar*(B) + .5*\epsvar*(C)$);
    \coordinate (E) at ($(D) + .281*(1,4)$);

    \node[below] at (B) {\scriptsize (5,-1)};
    \node[above] at (C) {\scriptsize (-1,1/2)};

    \coordinate (A1) at ($(A)-.5*(A)$);
    \coordinate (B1) at ($(B)-.6*(B)$);
    \coordinate (C1) at ($(C)-.35*(C)$);

    \draw[dashed] (B) -- (B1);
    \draw[dashed] (C) -- (C1);

    \draw[line width=1.5,color=violet] (D) -- (E);

    \fill (0,0) circle (1pt); 

\begin{scope}[scale=.8]

    \draw[line width=.8pt] ($(B1)-(2pt,2pt)$) -- ++(4pt,4pt);
    \draw[line width=.8pt] ($(B1)-(2pt,-2pt)$) -- ++(4pt,-4pt);

    \draw[line width=.8pt] ($(C1)-(2pt,2pt)$) -- ++(4pt,4pt);
    \draw[line width=.8pt] ($(C1)-(2pt,-2pt)$) -- ++(4pt,-4pt);
\end{scope}
        \end{tikzpicture}

    \end{minipage}
    \hfill 
    \begin{minipage}[b]{0.49\linewidth}
        \centering

\begin{tikzpicture}[cross/.style={cross out, draw, minimum size=2*(#1-\pgflinewidth), inner sep=0pt, outer sep=0pt}, cross/.default={3pt}]

    \coordinate (A) at (-1,-1);
    \coordinate (B) at (5,-1);
    \coordinate (C) at (-1,1/2);
    \coordinate (O) at (0,0);

    \def\epsvar{0.2} 

    \coordinate (Ox) at ($(A)+\epsvar*(2,0)$);
    \coordinate (Oy) at ($(A)+\epsvar*(0,1)$);

    \draw (Oy) -- (Ox) -- (B) -- (C) -- cycle;

    \coordinate (D) at ($(Ox)!0.5!(Oy)$);
    \coordinate (E) at ($(D) + .52*(1,2)$);

    \node[below] at (B) {\scriptsize (5,-1)};
    \node[above] at (C) {\scriptsize (-1,1/2)};

    \coordinate (A1) at ($(A)-.5*(A)$);
    \coordinate (B1) at ($(B)-.6*(B)$);
    \coordinate (C1) at ($(C)-.72*(C)$);    

    \draw[dashed] (B) -- (B1);
    \draw[dashed] (C) -- (C1);

    \draw[line width=1.5,color=blue] (D) -- (E);

    \fill (0,0) circle (1pt); 

\begin{scope}[scale=.8]

    \draw[line width=.8pt] ($(B1)-(2pt,2pt)$) -- ++(4pt,4pt);
    \draw[line width=.8pt] ($(B1)-(2pt,-2pt)$) -- ++(4pt,-4pt);

    \draw[line width=.8pt] ($(C1)-(2pt,2pt)$) -- ++(4pt,4pt);
    \draw[line width=.8pt] ($(C1)-(2pt,-2pt)$) -- ++(4pt,-4pt);
\end{scope}
        \end{tikzpicture}

    \end{minipage}
    \hfill 

   \caption{Some visible symplectic curves in almost toric fibrations. The curve $\wt{C}_{1,2}$ in the right diagram is the normal crossing resolution of the visible cuspidal curve $C_{1,2}$}
    \label{fig:vis_curves}
\end{figure}

\subsubsection{Visible ellipsoid embeddings}\label{subsubsec:vis_ell}

Let $Q \subset \MM_\R$ be a $T$-polygon, and let $\pi: \atf(Q_\nodal) \ra Q_\nodal$ denote the corresponding almost toric fibration as in \S\ref{subsubsec:nodal_int_from_poly}.
Noting that the base-nodes of $Q_\nodal$ can be pushed arbitrarily close to the boundary by nodal slides (and recalling Remark~\ref{rmk:nodal_trade}), we have:
\begin{prop}[{\cite[Prop. 2.35]{cristofaro2020infinite}}]\label{prop:vis_ell_emb}
Let $\vv$ be a Delzant vertex of $Q$ such that the adjacent edges have affine lengths $a$ and $b$.
Then for any $\eps > 0$ we have a symplectic embedding $E(\tfrac{a}{1+\eps},\tfrac{b}{1+\eps}) \hooksymp \atf(Q_\nodal)$.
\end{prop}
We will refer to an embedding as in Proposition~\ref{prop:vis_ell_emb} as a \hl{visible ellipsoid embedding}.
Note that if $Q$ has a smooth corner $\vv_\op{sm}$, then the same is true of its mutations $\mut_\vv(Q)$ at vertices $\vv \neq \vv_\op{sm}$.
Since by Proposition~\ref{prop:mut_imply_symp} $Q$ and $\mut_\vv(Q)$ encode symplectic four-manifolds which are symplectomorphic, this gives a mechanism for constructing many ellipsoid embeddings into a fixed target space.
Indeed, the ellipsoid embeddings corresponding to the inner corner points of the rigid del Pezzo infinite staircases can all be described in this way:

\begin{prop}[{\cite[\S5]{cristofaro2020infinite}}]\label{prop:inner_corners_vis_ell}
Let $M$ be a rigid del Pezzo surface, and let $Q$ be the corresponding dual Fano $T$-polygon (either a triangle or a quadrilateral) in Figure~\ref{fig:smooth_Fano_polygons}.
Then each inner corner point of the infinite staircase in the ellipsoid embedding function $c_M(x)$
corresponds to a visible ellipsoid embedding after applying a sequence of full mutations to $Q$.
\end{prop}

If $Q$ is a triangle then the ellipsoidal embedding in Proposition~\ref{prop:inner_corners_vis_ell} fills the entire volume of $ M$ and hence is clearly maximal.  We will explain in 
 \S\ref{subsec:singII_intro} (see in particular Proposition~\ref{prop:out_cor_vis})
 why this is still the case when $Q$ is a quadrilateral via the notion of visible symplectic obstructions.  
 It will follow that the numerics of these staircases 
can be completely understood in terms of structures that are visible  in  suitable families of almost toric bases.

\subsubsection{Cohomology class of the symplectic form}\label{subsubsec:atf_cohom}

For $Q$ a Delzant polygon with edges $e_1,\dots,e_\ell$, recall that the associated toric surface $V_Q$ carries a natural K\"ahler form $\omega$ for which $Q$ is the moment polygon of a Hamiltonian $\TT^2$-action (see e.g. \cite[\S6.6]{da2003symplectic}).
Using \eqref{eq:hom_of_V_Q} together with the fact that the toric divisor $\Ddiv_{e_i}$ has symplectic area $\afflen(e_i)$ for $i = 1,\dots,\ell$, this characterizes the cohomology class $[\omega] \in H^2(V_Q;\R)$.

We seek to extend this to almost toric fibrations of the type $\pi: \atf(Q_\nodal) \ra Q_\nodal$ constructed in \S\ref{subsubsec:nodal_int_from_poly}.
Under the homeomorphic identification $Q \cong Q_\nodal$, the vertices $\vv_1,\dots,\vv_\ell$ of $Q$ now correspond to corank one elliptic values in $Q_\nodal$.
In particular, the edge preimages $\pi^{-1}(e_i) \subset \atf(Q_\nodal)$ are symplectic annuli rather than two-spheres, so they do not a priori represent homology classes.
However, note that $r_i$ times the circle $\pi^{-1}(\vv_i)$ bounds the Lagrangian pinwheel $\LL_{[\bb_i^{m_i},\vv_i]}$ discussed in \S\ref{subsubsec:vis_Lag_symp}.
Therefore the rational cycle
\begin{align}
\wt{\Ddiv}_{e_i} := \pi^{-1}(e_i) - \tfrac{1}{r_i}\cdot \LL_{[\bb_i^{m_i},\vv_i]} + \tfrac{1}{r_{i+1}}\cdot \LL_{[\bb_{i+1}^{m_{i+1}},\vv_{i+1}]}
\end{align}
defines a homology class $[\wt{\Ddiv}_{e_i}] \in H_2(\atf(Q_\nodal);\Q)$.

The following results are easily checked.  Note also that the cycle 
$[\wt{\Ddiv}_{e_i}]$ has symplectic area equal to the integral affine length of the edge $e_i$.
Using the cut-and-paste discussion in \S\ref{subsubsec:nodal_int_from_poly} (similar to Symington's Euler computation for almost toric fibrations given in \cite[\S8]{symington71four}), we have:

\begin{lemma}\label{cor:symp_form_in_ATF0}
 The rational homology group $H_2(\atf(Q_\nodal);\Q)$ is generated by $[\wt{\Ddiv}_{e_i}]$ for $i=1,\dots,\ell$ and 
 the Lagrangian sphere classes $[\LL_{[\bb_i^j,\bb_i^{j+1}]}]$ for $i=1,\dots,\ell$ and $j = 1,\dots,m_i-1$.
\end{lemma}

\begin{cor}\label{cor:PD_to_symp_form_in_ATF}
A homology class $A \in H_2(\atf(Q_\nodal);\R)$ is Poincar\'e dual to the cohomology class of the symplectic form if and only if we have $A \cdot [\wt{\Ddiv}_{e_i}] = \afflen(e_i)$ for $i=1,\dots,\ell$ and $A \cdot [\LL_{[\bb_i^j,\bb_i^{j+1}]}] = 0$ for $i=1,\dots,\ell$ and $j = 1,\dots,m_i-1$.
\end{cor}

\sss

By Lemma~\ref{lem:delz_Q_dual_Fano_c1_crit}, if $Q$ is a Delzant polygon  the toric symplectic four-manifold $V_Q$ is monotone with monotonicity constant $1$ (i.e. $[\omega] = \pd(c_1)$) if and only if $Q$ is dual Fano.
This extends to  almost toric manifolds as follows:

\begin{prop}\label{prop:dual_Fano_monotone}
If $Q$ is a dual Fano $T$-polygon, then the symplectic four-manifold $\atf(Q_\nodal)$ is monotone with monotonicity constant $1$.
\end{prop}

Before beginning the proof, we recall the following result.

\begin{lemma}[{\cite[Prop. 8.2]{symington71four}}] \label{lem:c_1_is_bdy}
 For a general $T$-polygon $Q$ with edges $e_1,\dots,e_\ell$, 
the full boundary preimage $\pi^{-1}(\bdy Q_\nodal) = \bigcup\limits_{i=1}^\ell \pi^{-1}(e_i)$ is Poincar\'e dual to $c_1(\atf(Q_\nodal))$.
\end{lemma}

\begin{proof}[Proof of Proposition~\ref{prop:dual_Fano_monotone}]
By \cite[Thm. 1.2]{kasprzyk2017minimality} there are precisely $10$ mutation equivalence classes of dual Fano $T$-polygons, corresponding to the $10$ topological types of smooth del Pezzo surfaces.
Since mutations imply symplectomorphisms (see Proposition~\ref{prop:mut_imply_symp}), it suffices to check the result for one representative in each of the $10$ equivalence classes.

 In fact, \cite{vianna2017infinitely} shows that $8$ of the $10$ mutation equivalence classes of dual Fano $T$-polygons have a triangular representative. Moreover, the two exceptions, $\CP^1(3) \# \ovl{\CP}^2(1)$ and $\CP^1(3) \#^{\times 2} \ovl{\CP}^2(1)$, instead have Delzant representatives, for which the result follows directly by Lemma~\ref{lem:delz_Q_dual_Fano_c1_crit} since in the smooth case $\atf(Q) = V_Q$.

Now suppose that $Q \subset \R^2$ is a dual Fano $T$-triangle. 
By Lemma~\ref{lem:c_1_is_bdy} and the preceding discussion, the boundary preimage $[\pi^{-1}(\bdy Q_\nodal)]$ together with the visible Lagrangian spheres $[\LL_{[\bb_i^j,\bb_i^{j+1}]}]$ for $i=1,\dots,\ell$ and $j=1,\dots,m_i-1$ form a basis for $H_2(\atf(Q_\nodal);\Q)$, and since $c_1([\LL_{[\bb_i^j,\bb_i^{j+1}]}]) = [\LL_{[\bb_i^j,\bb_i^{j+1}]}] \cdot [\pi^{-1}(\bdy Q_\nodal)] = 0$, it follows that $\atf(Q_\nodal)$ is monotone.

To see that the monotonicity constant is $1$, we will define a test homology class in $H_2(\atf(Q_\nodal);\R)$ and check that its symplectic area agrees with its Chern number.
Let $e_1,e_2,e_3$ denote the edges of $Q$, where $e_i$ has primitive outward normal vector $(p_i,q_i) \in \Z^2$ and affine length $\ell_i \in \R_{>0}$, so that we have $\sum\limits_{i=1}^3 \ell_i \cdot (p_i,q_i) = (0,0)$.
 Let $\ga_i \subset \R^2$ denote a straight line segment
 which starts at the origin and ends on  the interior of the edge $e_i$, let $\al_i$ denote the covector field along $\ga_i$ with constant value $p_idx_1 + q_idx_2$, and let $C_i := C_{(\ga_i,\al_i)} \subset \atf(Q_\nodal)$ denote the corresponding visible symplectic disk as in \S\ref{subsubsec:vis_Lag_symp}.
Using \eqref{eq:area_vis_symp} and the fact that $Q$ is dual Fano, the symplectic area of $C_i$ is $1$.
  Also, note that $\bdy C_i$ is a circle in the Lagrangian torus fiber $\pi^{-1}(\vec{0})$ for $i=1,2,3$, and we have 
 \begin{align*}
 \sum\limits_{i=1}^3 \ell_i [\bdy C_i] = 0 \in H_1(\pi^{-1}(\vec{0});\R).
 \end{align*}
Thus we can find a Lagrangian $\R$-chain $\calL$ in $\pi^{-1}(\vec{0})$ such that 
$\sum\limits_{i=1}^3 \ell_i C_i + \calL$ is an $\R$-cycle in $\atf(Q_\nodal)$, say representing a class $A \in H_2(\atf(Q_\nodal);\R)$.
The symplectic area of $A$ is $\sum\limits_{i=1}^3 \ell_i\, \op{area}(C_i) = \ell_1 + \ell_2 + \ell_3$, and by Lemma~\ref{lem:c_1_is_bdy} its Chern number is also $A \cdot [\pi^{-1}(\bdy Q_\nodal)] = \ell_1 + \ell_2 + \ell_3$.
\end{proof}

\subsection{Rational curves in toric surfaces}\label{subsec:rat_curves_tor_surf}

We now turn our attention to the construction of rational algebraic curves in toric surfaces, typically of strictly positive index.
Our approach here is to write down explicit formulas for rational curves in the dense complex torus $(\C^*)^2 \subset V_Q$, and then take their closures to obtain rational curves in $V_Q$.
This could be seen as a step in the direction of a tropical-to-holomorphic correspondence as in \cite{Mikhalkin_JAMS}, but here we focus primarily on those curves which we will need in \S\ref{subsec:inflatable_curves}, keeping the discussion as elementary and explicit as possible.

As before, let $\MM$ be a rank two lattice with dual lattice $\NN$, and put $\MM_\R = \MM \otimes_\Z \R$ and $\NN_\R = \NN \otimes_\Z \R$. In the following, we will say that a polygon $Q \subset \MM_\R$ is \hl{rational} if all of its vertices lie in $\MM_\Q := \MM \otimes_\Z \Q$. Note that any rational polygon becomes a lattice polygon after scaling by some positive integer.

\begin{prop}\label{prop:curve_in_rat_surf}
Let $Q \subset \MM_\R$ be a rational polygon with edges $e_1,\dots,e_\ell$ and corresponding primitive inward normals $\vecn_1,\dots,\vecn_\ell \in \NN$. 
Suppose that we are given a set (possibly empty) of positive integers $J_{e_i} = \{k^i_1,\dots,k^i_{s_i}\}$ for each $i=1,\dots,\ell$, such that the following balancing condition holds
\begin{align}\label{eq:rat_curve_prop_bal}
\sum_{i=1}^\ell \sum\limits_{j=1}^{s_i} k^i_j \vecn_i = \vec{0},
\end{align}
and such that the quantities $d_i :=\sum\limits_{j=1}^{s_i}k^i_j$ satisfy $\gcd(d_1,\dots,d_\ell) = 1$.
Then there exists a rational
algebraic curve $C$ in $V_Q$ such that, for $i=1,\dots,\ell$, $C$ intersects $\Ddiv_{e_i}$ in its interior in precisely $s_i$ local branches, with  contact orders $k^i_1,\dots,k^i_{s_i}$, and $C$ is otherwise disjoint from $\Ddiv_{e_i}$.
\end{prop}

\begin{example}
Figure~\ref{fig:ATF_3_pics} middle gives a tropical representation of the curve $C$ in Proposition~\ref{prop:curve_in_rat_surf} in the case that $Q$ is a triangle with vertices $(-1,-1),(-1,1),(3,-1)$, with $J_{e_1} = \{2\}$, $J_{e_2} = \{1\}$, $J_{e_3} = \{1\}$.
\end{example}

Let $\Si$ denote the punctured Riemann surface $\CP^1 \setminus \{w_1,\dots,w_\vk\}$ for some pairwise distinct $w_1,\dots,w_\vk \in \CP^1$.
Given $\vecfrako = (\frako_1,\dots,\frako_\vk) \in \Z^{\vk}$, there is a unique (up to choice of constant $A \in \C^*$) holomorphic function $f_{\Si,\vecfrako}: \Si \ra \C^*$ with zero of order $\frako_i$ (i.e. pole of order $-\frako_i$) at the puncture $w_i$ for $i = 1,\dots,\vk$ and no other zeros or poles, given explicitly by
\begin{align}
f_{\Si,\vecfrako}(z) = A(z-w_1)^{\frako_1}\cdots (z-w_\vk)^{\frako_\vk}.
\end{align}
For simplicity we will usually take $A=1$.

By \hl{ordered toric degree} we will mean a tuple 
$\tordeg = (\vecv_1,\dots,\vecv_\vk)$ 
for some $\vk \in \Z_{\geq 2}$, with 
$\vecv_1,\dots,\vecv_\vk \in \NN \setminus \{\vec{0}\}$ such that $\sum_{i=1}^\vk \vecv_i = \vec{0}$.
Let $T_\NN \subset V_Q$ denote the complex torus $\NN \otimes_\Z \C^*$, which is identified with $(\C^*)^n$ after choosing a basis $\vecb_1,\dots,\vecb_n$ for $\NN$.
Given an ordered toric degree $\tordeg = (\vecv_1,\dots,\vecv_\vk)$, we have the holomorphic function $f_{\Si,\tordeg}: \Si \ra T_\NN$ whose $j$th component with respect to the chosen basis is $f_{\Si,(v_1^j,\dots,v_\vk^j)}$, where we put $\vecv_i = \sum_{j=1}^n v_i^j \vecb_j$ for $i = 1,\dots,\vk$.  
More explicitly, in the case $\NN_\R = \R^2$ with $\vecb_1,\vecb_2$ the standard basis, we put
\begin{align}\label{eq:f_Si_tor}
f_{\Si,\tordeg}(z) = (f_x(z),f_y(z)) = \left((z-w_1)^{v_1^x}\cdots (z-w_\ell)^{v_\ell^x},(z-w_1)^{v_1^y}\cdots (z-w_\ell)^{v_\ell^y} \right),
\end{align}
with $\vecv_i = (v_i^x,v_i^y)$ for $i=1,\dots,\vk$.

\begin{proof}[Proof of Proposition~\ref{prop:curve_in_rat_surf}]

Put $\tordeg := (k_1^1\vecn_1,\dots,k_{s_1}^1\vecn_1,\dots,k^\ell_1 \vecn_\ell,\dots,k^\ell_{s_\ell}\vecn_\ell)$, let $\Si$ be $\CP^1$ minus $\sum\limits_{i=1}^\ell \sum\limits_{j=1}^{s_i}k_j^i$ punctures, and let $C$ be the closure of the image of $f_{\Si,\tordeg}: \Si \ra (\C^*)^2 \subset V_Q$. 
We will show that at the first puncture, $w_1$, $C$ intersects $\Ddiv_{e_1}$ in its interior with contact order $k_1^1$, the situation being similar at the other punctures by symmetry.
We may assume that our basis for $\NN$ is chosen such that $\vecn_1 = (0,1)$.
Let $\si = \R_{\geq 0}\lan \vecn_1 \ran \subset \NN_\R$ be the cone generated by $\vecn_1$, with the dual cone $\si^\vee \subset \MM_\R$ generated by $(0,1),(1,0),(-1,0)$.
The corresponding affine toric variety $U_\si$ is identified with $\{(z_1,z_2,z_3) \in \C^3\;|\; z_1z_3 = 1\}$, with $\Ddiv_{e_1} \cap U_\si = \{(z_1,z_2,z_3) \in U_\si\;|\; z_2 = 0\}$ and 
 with the inclusion map $\iota: (\C^*)^2 \ra U_\si$ given by $(x,y) \mapsto (x,y,x^{-1})$.
We thus have $(\iota \circ f_{\Si,\tordeg})(z) = (f_x(z),f_y(z),f_x(z)^{-1})$, and therefore
\begin{align*}
\lim\limits_{z \ra w_1} (\iota \circ f_{\Si,\tordeg})(z) &= \lim\limits_{z \ra w_1} (A (z-w_1)^{k_1^1 n_1^x},B (z-w_1)^{k_1^1 n_1^y},A^{-1} (z-w_1)^{-k_1^1 n_1^x}) \\&= \lim\limits_{z \ra w_1} (A (z-w_1)^{k_1^1 \cdot 0},B (z-w_1)^{k_1^1 \cdot 1},A^{-1} (z-w_1)^{-k_1^1 \cdot 0}) \\&= (A,0,A^{-1})
\end{align*}
for some constants $A,B \in \C^*$.
The corresponding contact order with $\Ddiv_{e_1}$ is given by the vanishing order of $f_y(z)$ as $z \ra w_1$, which is $k_1^1$.
\end{proof}

\begin{cor}\label{cor:curve_hitting_each_edge_once}
Assume that $Q \subset \MM_\R$ is a lattice polygon such that the edge affine lengths $\afflen(e_1),\dots,\afflen(e_\ell) \in \Z_{\geq 1}$ are coprime.
Then there exists a
 rational algebraic curve $C$ in $V_Q$ such that
$C$ intersects $\Ddiv_{e_i}$ in a single point in its interior with multiplicity $\afflen(e_i)$ for $i=1,\dots,\ell$.
\end{cor}
\begin{proof}
Put $J_{e_i} = \{\afflen(e_i)\}$ for $i=1,\dots,\ell$.
Since $Q$ is a closed polygon, we have $\sum\limits_{i=1}^\ell e_i = \vec{0}$, which implies the balancing condition \eqref{eq:rat_curve_prop_bal}.
\end{proof}

The following is an algebraic counterpart of the visible symplectic curves over straight lines discussed in \S\ref{subsubsec:vis_Lag_symp}:
\begin{cor}\label{cor:rat_curve_par_edges}
 Let $Q \subset \MM_\R$ be a rational polygon having two parallel edges $e_+,e_-$.
Then there exists a nonsingular rational algebraic curve $C$ in $V_Q$ which intersects each of $\Ddiv_{e_+},\Ddiv_{e_-}$ transversely in one point and is disjoint from the other toric divisors. 
\end{cor}

We will also need to know that the curves constructed above are suitably robust  under deformations of the complex structure.
Let $\calM(J_{e_1},\dots,J_{e_\ell})$ denote the (uncompactified) moduli space of curves $C$ as in Proposition~\ref{prop:curve_in_rat_surf}.
Here we view curves in $\calM(J_{e_1},\dots,J_{e_\ell})$ as holomorphic maps $\CP^1 \ra V_Q$ (modulo biholomorphic reparametrization) having specified intersection pattern with the toric divisors $\Ddiv_{e_1},\dots,\Ddiv_{e_\ell}$.
\begin{lemma}\label{lem:mod_sp_is_reg}
  The moduli space $\calM(J_{e_1},\dots,J_{e_\ell})$ is regular.
\end{lemma}
\begin{proof}
Let $u: \CP^1 \ra V_Q$ be a curve in  $\calM(J_{e_1},\dots,J_{e_\ell})$.
Then $u$ is regular by automatic transversality (see \cite[Thm. 1]{Wendl_aut}) provided that we have $\ind_\R(u) > 2Z(du) - 2$, where $Z(du)$ is the total (complex)
vanishing order of the derivative of $u$ at all of its critical points.
Note that we have $\ind_\R(u) = 2\vk - 2$, where $\vk = \sum\limits_{i=1}^\ell s_i$ is the number of punctures of $\Si$.
Meanwhile, we have 
\begin{align*}
\frac{f_x'(z)}{f_x(z)}=\tfrac{d}{dz} \log(f_x(z)) = \sum\limits_{i=1}^\vk \frac{v_i^x}{z-w_i} = \frac{P(z)}{(z-w_1)\cdots(z-w_\vk)},
\end{align*}
where $P(z)$ is a polynomial of degree at most $\vk-1$ (actually at most $\vk-2$ since $\sum\limits_{i=1}^\vk v_i^x = 0$), so we have $Z(du) \leq \vk-1$ and thus
$\ind_\R(u) = 2\vk -2 > 2\vk - 4 \geq 2Z(du)-2$.
 \end{proof}

\begin{rmk}
In the case $\vk=3$, $f_{\Si,\tordeg}$ is in fact an immersion, i.e. $Z(du) = 0$.
Indeed, without loss of generality we can take $w_1 = 0$, $w_2 = 1$, $w_3 = \infty$, so that \eqref{eq:f_Si_tor} becomes
\begin{align*}
f_{\Si,\tordeg}(z) = \left( z^{v_1^x}(z-1)^{v_2^x},z^{v_1^y}(z-1)^{v_2^y}\right).
\end{align*}
Observe that we have $f_{\Si,\tordeg} = \Phi \circ  g$, 
where $g: \Si \ra (\C^*)^2$, $g(z) = (z,z-1)$ is a parametrization of the (nonsingular) pair of pants $\{x = y+1\} \subset (\C^*)^2$ and $\Phi: (\C^*)^2 \ra (\C^*)^2$, $\Phi(x,y) = (x^{v_1^x}y^{v_2^x},x^{v_1^y}y^{v_2^y})$ is a degree $|\det(\vecv_1,\vecv_2)|$ 
holomorphic covering map (c.f. \cite[\S G.2]{evans2023lectures}).
\end{rmk}

\subsection{Inflatable sesquicuspidal curves and the inner corners}\label{subsec:inflatable_curves}

The goal of this subsection is to prove the following result, which we then use to deduce Theorem~\ref{thmlet:inner_corner_curves}.
Recall that we associate to a $T$-polygon $Q$ an almost toric fibration $\pi: \atf(Q_\nodal) \ra Q_\nodal$ as in \S\ref{subsubsec:nodal_int_from_poly}. 

\begin{thm}\label{thm:sesqui_main}
Let $Q \subset \MM_\R$ be a dual Fano lattice $T$-polygon with a Delzant vertex $\vv_\sm$. Let $e,e'$ be the edges adjacent to $\vv_\sm$, and put $\afflen(e) = kp$ and $\afflen(e') = kq$, where $k,p,q \in \Z_{\geq 1}$ satisfy $\gcd(p,q) = 1$ and $p \geq q$. Assume also that the affine lengths of the edges of $Q$ are coprime.
Then there exists a rational $\Jint$-holomorphic curve $C$ in $\atf(Q_\nodal)$, where:
\begin{itemize}
  \item $C$ has a cusp with Puiseux characteristic $(q;p)$ if $k=1$ and $(kq;kp,kp+1)$ if $k \geq 2$
  \item $C$ is Poincar\'e dual to the symplectic form on $\atf(Q_\nodal)$
    \item $[C] \cdot [C] \geq k^2pq$ 
       \item $\Jint$ is a tame integrable almost complex structure on $\atf(Q_\nodal)$ such that $(\atf(Q_\nodal),\Jint)$ is biholomorphic to a $\Q$-Gorenstein smoothing $\wt{V}_Q$ of $V_Q$.
\end{itemize}
Moreover, when $Q$ is a triangle we can assume that $C$ is unicuspidal.
\end{thm}

Note that in the symplectic category we can easily perturb $C$ to make it sesquicuspidal (i.e. positively immersed away from the cusp), although this is not guaranteed as a $\Jint$-holomorphic curve.
In particular, after such a perturbation $C$ satisfies the hypotheses of Theorem~\ref{thmlet:inflation_from_sescusp} with $c=1$, i.e. inflating along $C$ gives a symplectic embedding $E(\tfrac{kq}{c'},\tfrac{kp}{c'}) \hooksymp M$ for any $c' > 1$.
In other words, any visible ellipsoid embedding (in the sense of Proposition~\ref{prop:vis_ell_emb}) can be obtained by symplectic inflation along a sesquicuspidal rational symplectic curve.
Theorem~\ref{thmlet:inner_corner_curves} upgrades this to algebraic curves in the case of rigid del Pezzo surfaces.

\begin{proof}[Proof of Theorem~\ref{thmlet:inner_corner_curves}]

Let $M$ be a rigid del Pezzo surface and let $(x=p/q,y)$ be an inner corner point on the graph of the corresponding ellipsoid embedding function $c_M(x)$.
According to Proposition~\ref{prop:inner_corners_vis_ell},
there exists a dual Fano $T$-polygon $Q \subset \R^2$ such that
\begin{itemize}
  \item $\atf(Q_\nodal)$ is symplectomorphic to $M$ 
  \item $Q$ has a Delzant vertex $\vv_\sm$ with adjacent edges $e,e'$ satisfying $\afflen(e) = \tfrac{1}{y}$ and $\afflen(e') = \tfrac{p}{qy}$
  \item $Q$ is a triangle in the cases $M = \CP^2,\CP^1 \times \CP^2,\CP^2 \#^{\times j}\ovl{\CP}^2$, $j =3,4$, and $Q$ is a quadrilateral in the cases $M = \CP^2 \#^{\times j} \ovl{\CP}^2$, $j = 1,2$.
\end{itemize}
Let $s \in \R_{>0}$ be the minimal scaling factor such that $s \cdot Q \subset \R^2$ is a lattice polygon.
Then $s \cdot Q$ satisfies the hypotheses of Theorem~\ref{thm:sesqui_main}, with $\afflen(s \cdot e) = kq$ and $\afflen(s \cdot e') = kp$ for $k := \tfrac{s}{qy} \in \Z_{\geq 1}$.
Let $C$ be the resulting curve in $\atf(s \cdot Q_\nodal)$.
Here $\atf(s \cdot Q_\nodal)$ is naturally identified as a symplectic manifold with $\atf(Q_\nodal)$ after scaling the symplectic form by $s$. 
In particular, $C$ corresponds to a curve $C'$ in $\atf(Q_\nodal)$ which is Poincar\'e dual to $s$ times the symplectic form.
Note that $(\atf(Q_\nodal),\Jint)$ is a rigid Fano complex surface and hence is necessarily biholomorphic to $M$.
Also, observe when $Q$ is a triangle we must have $k = 1$, since then evidently $s = qy$ is minimal such that $s \cdot Q$ is a lattice triangle.
Thus the curve $C'$ verifies Thereom~\ref{thmlet:inner_corner_curves}.
\end{proof}

Our proof of Theorem~\ref{thm:sesqui_main} will proceed roughly in the following steps:
\begin{enumerate}[label=\arabic*)]
  \item construct a rational curve in a weighted blowup of $V_Q$ using the results in \S\ref{subsec:rat_curves_tor_surf} 
  \item blow down to get a curve with a distinguished cusp in $V_Q$
   \item push this curve into a $\Q$-Gorenstein smoothing $\wt{V}_Q$ of $V_Q$ 
     \item identify $\wt{V}_Q$ diffeomorphically with $\atf(Q_\nodal)$.
\end{enumerate}

We now proceed with the details. Let $Q$ as in Theorem~\ref{thm:sesqui_main}, and let $\vv_1,\dots,\vv_\ell$ denote the vertices of $Q$, with corresponding edges $e_i = [\vv_i,\vv_{i+1}]$ for $i=1,\dots, \ell$ (here $i$ is taken modulo $\ell$).
Here we take $\vv_1 = \vv_\sm$ to be the Delzant vertex, with adjacent edges $e = e_1$ and $e' = e_\ell$.

\begin{lemma}\label{lem:sesqui_curve_in_orbi} With $Q$ as above, there
 exists a rational algebraic curve $C$ in $V_Q$ such that 
\begin{itemize}
  \item for $i = 2,\dots,\ell-1$, $C$ intersects $\Ddiv_{e_i}$ in a single point in its interior with multiplicity $\afflen(e_i)$
  \item $C \cap \Ddiv_{e_1} = C \cap \Ddiv_{e_\ell} = \Ddiv_{e_1} \cap \Ddiv_{e_\ell}$
  \item $C$ has a cusp at the point $\Ddiv_{e_1} \cap \Ddiv_{e_\ell}$ with Puiseux characteristic $(q;p)$ if $k = 1$ and 
  $(kq;kp,kp+1)$ if $k \geq 2$ 
  \item $[C] \cdot [C] \geq k^2pq$.
\end{itemize}
Moreover, when $Q$ is a triangle (i.e. $\ell = 3$) we can assume that $C$ is unicuspidal.
\end{lemma}

\begin{proof}[Proof of Lemma~\ref{lem:sesqui_curve_in_orbi}]

Let $Q'$ be the polygon with vertices $\vv_1',\vv_2,\dots,\vv_\ell,\vv_1''$, where $\vv_1' = \tfrac{1}{2} \vv_1 + \tfrac{1}{2}\vv_2$ and $\vv_1'' =  \tfrac{1}{2}\vv_\ell + \tfrac{1}{2} \vv_1$ (i.e. $Q'$ is obtained from $Q$ by ``chopping off'' the vertex $\vv_1$).
Denote the corresponding edges of $Q'$ by 
$e_1',e_2,\dots,e_{\ell-1},e_\ell',e_\slant$,
where $e_1' = [\vv_1',\vv_2]$, $e_{\ell}' = [\vv_\ell,\vv_1'']$, and $e_\slant = [\vv_1'',\vv_1']$.
Let $\Ddiv_{e_1'},\Ddiv_{e_2},\dots,\Ddiv_{e_{\ell-1}},\Ddiv_{e_\ell'},\Ddiv_{e_\slant}$ denote the corresponding toric boundary divisors of the associated toric surface $V_{Q'}$, which is a $(p,q)$-weighted blowup of $V_Q$.
By Proposition~\ref{prop:curve_in_rat_surf}, we can find a rational algebraic curve $C' \subset V_{Q'}$ such that
\begin{itemize}
  \item $C'$ intersects $\Ddiv_{e_\slant}$ in a single point in its interior with multiplicity $k$
  \item for $i = 2,\dots,\ell-1$, $C'$ intersects $\Ddiv_{e_i}$ in a single point in its interior with multiplicity $\afflen(e_i)$
    \item $C'$ is disjoint from $\Ddiv_{e_1'}$ and $\Ddiv_{e_\ell'}$.
\end{itemize}
Note that, in the case $\ell=3$, $C'$ is nonsingular as in Corollary~\ref{cor:rat_curve_par_edges}.

We now consider the image of $C$ under the weighted blowdown $V_{Q'} \ra V_Q$ along $\Ddiv_{e_\slant}$.
More explicitly, we first consider the iterated toric blowup $V^{\op{res}}$ of $V_{Q'}$ which minimally resolves the singularities of $V_{Q'}$ at the toric fixed points corresponding to $\vv'_1$ and $\vv''_1$.
Let $C^{\op{res}}$ denote the proper transform of $C'$ in $V^\op{res}$.
These blowups result in a collection of negative self-intersection spheres $\F_1,\dots,\F_{L}$, where $\F_L$ is the proper transform of $\Ddiv_{e_\slant}$, such that $C^\op{res}$ intersects $\F_L$ in a single point with multiplicity $k$ and is disjoint from $\F_1,\dots,\F_{L-1}$.
In the case $k=1$, the collection $\F_1,\dots,\F_L,C^\op{res}$ has precisely the same intersection pattern as in the normal crossing resolution of a $(p,q)$ cusp (c.f. \S\ref{subsec:toric_p_q}), whence we can blow down along $\F_L,\dots,\F_1$ to obtain a curve $C$ in $V_Q$ with a $(p,q)$ cusp.
Similarly, in the case $k \geq 2$, a comparison with Example~\ref{ex:res_of_kq_kp_kp+1} shows that the blown down curve $C$ has a cusp with Puiseux characteristic $(kq;kp,kp+1)$.

To establish $[C] \cdot [C] \geq k^2pq$, note that $c_1([C^\op{res}])$ is given by the homological intersection number of $[C^\op{res}]$ with the toric boundary divisor of $V^{\op{res}}$. This is evidently at least $2$, so using the adjunction formula we have $[C^{\op{res}}] \cdot [C^{\op{res}}] \geq 0$, which by Lemma~\ref{lem:k^2pq} is equivalent to $[C] \cdot [C] \geq k^2pq$.

Finally, the last sentence of the lemma follows by Corollary~\ref{cor:rat_curve_par_edges}.
\end{proof}

\begin{proof}[Proof of Theorem~\ref{thm:sesqui_main}]

Let $C$ be a rational algebraic curve in $V_Q$ as constructed by Lemma~\ref{lem:sesqui_curve_in_orbi}. In particular, $C$ has a distinguished cusp and satisfies $[C] \cdot [C] \geq k^2pq$.
We view $C$ as having a $J$-holomorphic parametrization $u: \CP^1 \ra V_Q$ with a prescribed cusp singularity at $\Ddiv_{e_1} \cap \Ddiv_{e_\ell}$,
where $J$ is the preferred integrable almost complex structure on $V_Q$.
Since $C$ is regular (see Lemma~\ref{lem:mod_sp_is_reg}) and disjoint from the singularities of $V_Q$, it deforms to a nearby curve $\wt{C}$ with the same type of cusp in a sufficiently small $\Q$-Gorenstein smoothing $\wt{V}_Q$ of $V_Q$.\footnote{Strictly speaking Lemma~\ref{lem:mod_sp_is_reg} says that a resolution $C' \subset V_{Q'}$ is regular (without any cusp condition), which suffices for our purposes since we can readily pass between curves with prescribed cusps and their resolutions (c.f. \cite[\S4.3]{cusps_and_ellipsoids}).}
By Proposition~\ref{prop:QG_def_diff_ATF}, there is a diffeomorphism $\Phi: \wt{V}_Q \ra \atf(Q_\nodal)$ such that $\Phi_*(\wt{J})$ tames the symplectic form on $\atf(Q_\nodal)$ (here $\wt{J}$ is the integrable almost complex structure on $\wt{V}_Q$).
Put $\wt{C}' := \Phi(\wt{C})$. We can assume that the smoothing $\wt{V}_Q$ is such that $[\wt{C}'] \cdot [\wt{\Ddiv}_{e_i}] = [C] \cdot [\Ddiv_{e_i}]$ for $i =1,\dots,\ell$ (here $[\wt{\Ddiv}_{e_1}],\dots,[\wt{\Ddiv}_{e_\ell}] \in H_2(\atf(Q_\nodal);\Q)$ are the homology classes from \S\ref{subsubsec:atf_cohom}), and hence $\wt{C}'$ is Poincar\'e dual to the symplectic form of $\atf(Q_\nodal)$ by Corollary~\ref{cor:PD_to_symp_form_in_ATF}.
Note also that we have $[\wt{C}'] \cdot [\wt{C}'] = [\wt{C}] \cdot [\wt{C}] = [C] \cdot [C] \geq k^2pq$.
Thus $\wt{C}'$ satisfies all of the conclusions of Theorem~\ref{thm:sesqui_main} with $\Jint := \Phi_*(\wt{J})$.
\end{proof}

\section{Singular algebraic curves in almost toric manifolds II}\label{sec:singII}

In this section we develop techniques to construct index zero unicuspidal rational algebraic curves in $\Q$-Gorenstein smoothings of singular toric surfaces. These are closely parallel to visible symplectic curves in almost toric fibrations which pass through a focus-focus singularity (as in \S\ref{subsubsec:vis_Lag_symp}).
In particular, we prove Theorem~\ref{thm:uni_alg_curve}, which is a (slight strengthening of a) restatement of Theorem~\ref{thmlet:nodal}.
The main technique is an explicit construction of $\Q$-Gorenstein pencils associated to polygon mutations as in \cite{ilten2012mutations,akhtar2016mirror}.

More specifically, in \S\ref{subsec:singII_intro} we introduce the notion of visible ellipsoid obstructions, which are carried by symplectic and in fact algebraic unicuspidal curves, and we observe that all obstructions for the rigid del Pezzo infinite staircases are of this type.
After a brief interlude on the Auroux model in \S\ref{subsec:aur_vis}, we discuss $\Q$-Gorenstein pencils in \S\ref{subsec:QG_pencils}, and we use these to construct explicit algebraic curves in $\Q$-Gorenstein smoothings in \S\ref{subsec:expl_uni_alg}.
Finally, in \S\ref{subsec:unicusp_appl} we discuss the classification of index zero unicuspidal rational algebraic curves in the first Hirzebruch surface and prove Theorem~\ref{thmlet:F_1_classif}.

\subsection{Unicuspidal curves from $T$-polygons and the outer corners}\label{subsec:singII_intro}

\subsubsection{Visible ellipsoid obstructions and unicuspidal symplectic curves}

We begin by discussing ellipsoid embedding obstructions which come from visible unicuspidal symplectic curves in the base of an almost toric fibration.
Fix $m,r,a \in \Z_{\geq 1}$ with $\gcd(r,a) = 1$.
Note that we do not assume $a < r$, but we can uniquely write $a = a' + \vs r$ for some $a' \in \{1,\dots,r-1\}$ and $\vs \in \Z_{\geq 0}$.
Let $Q \subset \MM_\R$ be a $T$-polygon with a Delzant vertex $\vv$ adjacent to vertices $\uu$ and $\ww$, where $\ww$ has type $\tfrac{1}{mr^2}(1,mra-1)$, such that the eigenray emanating from $\ww$ intersects the line segment $[\uu,\vv]$ in a point $\pp$.
We will assume that $\MM_\R = \R^2$, $\vv = (0,0)$, $\ww$ lies on the positive $y$-axis, $\uu$ lies on the positive $x$-axis, and the edges emanating from $\ww$ are $(0,-1)$ and $(mr^2,1-mra)$.
In particular, the eigenray emanating from $\ww$ points in the direction $(r,-a)$.

Put $\ell_1 := \afflen([\vv,\ww])$ and $\ell_2 := \afflen([\vv,\pp])$.
Note that we have 
$\tfrac{\ell_1}{\ell_2} = \tfrac{a}{r}$,
 and, because the nodal ray from $\ww$ meets the side $[\uu, \vv]$,
there is a visible ellipsoid embedding in the sense of \S\ref{subsubsec:vis_ell} (c.f. Figure~\ref{fig:ATF_3_pics} right\footnote{Strictly speaking the roles of the $x$ and $y$ axes are swapped in Figure~\ref{fig:ATF_3_pics} right compared with our conventions in this section. There is another (unshaded) ellipsoid obstruction which is visible in this same figure, which is associated with the eigenray emanating from the other vertex (at $(-1,1)$).})  
\begin{align}\label{eq:vis_ell_emb_eray}
E(\tfrac{\ell_1}{c'},\tfrac{\ell_2}{c'}) \hooksymp \atf(Q_\nodal)
\end{align}
for any $c' > 1$. 

\begin{prop}\label{prop:uni_symp_curve} Let $Q$ be a $T$-polygon as above such that the eigenray emanating from $\ww$ is in direction $(r,-a)$.  Then there is an index zero rational $(r,a)$-unicuspidal symplectic curve $C$ in $\atf(Q_\nodal)$ with  area  
  $r\ell_1 = a\ell_2$.
 \end{prop}
\begin{proof}
We have $\vv = (0,0)$, $\ww = (0,\ell_1)$ and $\pp = (\ell_2,0)$, with the eigenray emanating from $\ww$ pointing in the direction $(r,-a)$.
Let $(\ga,\al)$ be the covector-decorated path where:
\begin{itemize}
  \item $\ga$ is the straight line segment starting on $\vv$ and ending at a point 
  $(ta,tr)$ on $[\ww,\pp]$ for some $t \in \R_{>0}$ 
  \item $\al$ is the lattice covector field along $\ga$ which vanishes on the eigenray direction $(r,-a)$, i.e. $\al = a dx_1 + r dx_2$.
  \end{itemize}
After a nodal slide, we can assume that $Q_\nodal$ has a base-node at $(ta,tr)$.
Let $C$ be the corresponding visible $(r,a)$-unicuspidal symplectic curve in $\atf(Q_\nodal)$ as in item \ref{item:vis_symp_uni} in \S\ref{subsubsec:vis_Lag_symp}.
By ~\eqref{eq:area_vis_symp}, the symplectic area of $C$ is given by
$\int_\ga \al = \lan(ta,tr),(a,r)\ran = \lan \ww,(a,r) \ran = r\ell_1$.
\end{proof} 
\begin{rmk}\label{rmk:cusp_type_depends_on_vs}
Note that the cusp type $(r,a) = (r,a' + \vs r)$ appearing in Proposition~\ref{prop:uni_symp_curve} depends not just on the vertex type $\tfrac{1}{mr^2}(1,mra-1) = \tfrac{1}{mr^2}(1,mra'-1)$ of $\ww$ but also on the parameter $\vs$, which controls how the eigenray of $\ww$ meets the Delzant vertex $\vv$.
 One can also relax the assumption that $\ww$ is adjacent to $\vv$, provided that the line segment joining $\vv$ perpendicularly to the eigenray of $\ww$ is contained in $Q$.
\end{rmk}

\begin{rmk} \label{rmk:vis_uni_symp_homol}
We can arrange that the curve $C$ in Proposition~\ref{prop:uni_symp_curve} is disjoint from the Lagrangian spheres $\LL_{[\bb_i^j,\bb_i^{j+1}]}$ (see \S\ref{subsubsec:vis_Lag_symp}) and also from the Lagrangian pinwheels $\LL_{[\bb_i^{m_i},\vv_i]}$ unless $\vv_i = \ww$, in which case there is a single transverse intersection point.
Let $[\ww,\qq]$ be the other edge containing $\ww$, where $\qq = \uu$ in the case that $Q$ is a triangle.
Then, using the generators from \S\ref{subsubsec:atf_cohom}, the homology class $[C]$ is characterized by
$[C] \cdot [\LL_{[\bb_i^j,\bb_i^{j+1}]}] = 0$, $[C] \cdot [\wt{\Ddiv}_{[\uu,\vv]}] = r$, $[C] \cdot [\wt{\Ddiv}_{[\vv,\ww]}] = a + 1/r$, $[C] \cdot [\wt{\Ddiv}_{[\ww,\qq]}] = -1/r$, and $[C] \cdot [\wt{\Ddiv}_e] = 0$ for all other edges $e$ of $Q$. In particular, $c_1([C]) = a + r$.

For example if $Q$ is a triangle such that $\atf(Q_\nodal) = \CP^1\times \CP^1$, then, writing $\ell_1,\ell_2$ for the usual generators of $H_2(\atf(Q_\nodal),\Z)$,  we have $[C] = d (\ell_1 + \ell_2)$ where $4d = r+a$ if $(r,-a)$ is the nodal ray of multiplicity $m=1$, and  otherwise $[C] = d_1\ell_1 + d_2\ell_2$ where $2(d_1 +d_2) = r+a$ and $|d_1-d_2| = 1$.
\end{rmk}

 Together with the obstruction provided by Theorem~\ref{thm:stab_obs_from_curve}, Proposition~\ref{prop:uni_symp_curve} immediately gives:
\begin{cor}\label{cor:vis_ell_obs}
Let $Q$ be a $T$-polygon as above such that the eigenray emanating from $\ww$ meets the edge $[\uu,\vv]$. Then
 the embedding \eqref{eq:vis_ell_emb_eray} is optimal in the sense that there is no such embedding for $c' < 1$ (even after stabilizing by $\C^{N \geq 1}$).
\end{cor}

 We will refer to an obstruction as in Corollary~\ref{cor:vis_ell_obs} as a \hl{visible ellipsoid obstruction}, since it is read off visually from the polygon $Q$ as in Figure~\ref{fig:ATF_3_pics} right.
Note that it follows that for any $\pp' \in [\pp,\uu]$ the visible ellipsoid embedding corresponding to the triangle with vertices $\vv,\pp',\ww$ is also optimal. 
In particular, this determines the ellipsoid embedding function for $M = \atf(Q_\nodal)$ on an appropriate closed interval as follows:\footnote{The slightly awkward phrasing division into cases 
comes from the fact that $c_M(x)$ is defined in terms of ellipsoids $E(1,x)$ with $x \geq 1$. 
The case $\ell_2 \leq \ell_1 \leq \ell_2 + \ell_3$ does not occur in nontrivial cases.}

\begin{cor}\label{cor:vis_ell_obs_func}
Let $Q$ be a $T$-polygon as above, with $\ell_1 = \afflen([\vv,\ww]), 
\ell_2 = \afflen([\vv,\pp])$ and  $\atf(Q_\nodal)$ symplectomorphic to $M$, and 
put $\ell_3 := \afflen([\pp,\uu])$. 
We have:
 \begin{enumerate}[label=(\Alph*)]
  \item
$c_M(x) = x/\ell_1$ for all $x \in [\tfrac{\ell_1}{\ell_2+\ell_3},\tfrac{\ell_1}{\ell_2}]$ if $\ell_1 > \ell_2 + \ell_3$;
\item $c_M(x) = \ell_1$ for all $x \in [\tfrac{\ell_2}{\ell_1},\tfrac{\ell_2+\ell_3}{\ell_1}]$ if $\ell_1 < \ell_2$.
\end{enumerate} 
\end{cor}
\begin{proof} Both of these claims
 hold by considering the family of embeddings represented by the visible triangles in $Q_\nodal$ with vertices $\ww, \vv, $ and $\uu'$ where $\uu'$ lies on the line between $\pp$ and $\uu$. The fact that the embedding with $\uu'=\pp$ is optimal (by Corollary~\ref{cor:vis_ell_obs}) implies that all the embeddings with $\uu'$ on the line $[\pp, \uu]$ are also optimal.
 If $\ell_1 < \ell_2 $ and $\ell_1 x \in [\ell_2, \ell_2+\ell_3]$, these give rise to optimal embeddings of $E(\ell_1,\ell_1x)$ into
 $\atf(Q_\nodal)$, which implies that $c_M(x)$ is constant over the corresponding interval as in (B).
On the other hand,  if  $\ell_2 \le z \le \ell_2 + \ell_3< \ell_1$ we obtain optimal embeddings   of $E(z,\ell_1)$ into       
  $\atf(Q_\nodal)$.      Setting $x = \frac{\ell_1}{z}> 1$ this translates to an optimal embedding of $E(\frac{\ell_1}x, \ell_1)$
  into       
  $\atf(Q_\nodal)$, which implies (A).
 \end{proof}

We will see in Proposition~\ref{prop:out_cor_vis} below that each of the rigid del Pezzo infinite staircases may be entirely described in this way, with (A) accounting for the sloped line preceding an outer corner and (B) accounting for the horizontal line following an outer corner.

\begin{rmk}
   Let us explain briefly why the obstruction in Corollary~\ref{cor:vis_ell_obs} holds from the perpsective of exceptional homology classes. 
By Proposition~\ref{prop:uni_symp_curve} there is a visible $(p,q)$-unicuspidal symplectic curve $C_{p,q}$ in $\atf(Q_\nodal)$.
After a sequence of blowups at the toric fixed point corresponding to the origin, we arrive at the normal crossing resolution $\wt{C}_{p,q}$, which is an exceptional curve. 
Recall that an exceptional class in a closed symplectic four-manifold $(M,\om)$ has a symplectic representative for every symplectic form $\om'$ that is deformation equivalent to $\om$, and must always have positive symplectic area.
Now consider the visible ellipsoid embedding \eqref{eq:vis_ell_emb_eray} for $c' > 1$, which corresponds to the subtriangle $T' \subset T$ which is a $\tfrac{1}{c'}$-scaling of the shaded triangle in Figure~\ref{fig:ATF_3_pics} right.
The symplectic area of the exceptional class $[\wt{C}_{p,q}]$ is a function of the distance between this slant edge of $T'$ and the eigenray emanating from $\ww$. 
Since this tends to zero as $c' \ra 1$, it follows that the embedding 
\eqref{eq:vis_ell_emb_eray} is optimal.
\end{rmk}

\subsubsection{Visible unicuspidal algebraic curves and outer corner curves}

The primary goal of this section is to show that visible ellipsoid obstructions in fact come from {\em algebraic} curves.  We prove the following Theorem in \S\ref{subsec:expl_uni_alg}.

\begin{thm}\label{thm:uni_alg_curve}  Let $Q \subset \MM_\R$ be a $T$-polygon which contains consecutive edges pointing in the directions $(-mr^2,mra-1),(0,-1),(1,0)$ for $m,r,a \in \Z_{\geq 1}$ with $\gcd(r,a) = 1$.
Then there is an index zero 
  $(r,a)$-unicuspidal rational symplectic curve $C$ in $\atf(Q_\nodal)$ which is $\Jint$-holomorphic, where $\Jint$ is a tame integrable almost complex structure on $\atf(Q_\nodal)$ such that $(\atf(Q_\nodal),\Jint)$ is biholomorphic to a sufficiently small $\Q$-Gorenstein smoothing $\wt{V}_Q$ of $V_Q$. 
Furthermore, we can assume that $C$ is $(r,a)$-well-placed with respect to a $\Jint$-holomorphic rational nodal anticanonical divisor $\calN$.
\end{thm}

\begin{rmk}
 One can check that the curve $C$ in Theorem~\ref{thm:uni_alg_curve} has the same homology class as that in Remark~\ref{rmk:vis_uni_symp_homol}, and in particular it has symplectic area 
$\lan \ww,(a,r)\ran$, where $\ww$ is the vertex lying on the positive $y$-axis. 
 Thus in the case that the eigenray emanating from $\ww$ intersects the edge of $Q$ in direction $(1,0)$, $C$ carries the same visible ellipsoid obstruction as in Corollary~\ref{cor:vis_ell_obs}.
\end{rmk}

\begin{rmk}
The utility of the last part of Theorem~\ref{thm:uni_alg_curve} is that, at least in the rigid del Pezzo case, we can apply iteratively the generalized Orevkov twist from \S\ref{sec:twist}, i.e. for each such curve $C$ we get a whole sequence of index zero unicuspidal rational algebraic curves $\Phi_M(C),\Phi_M^2(C),\Phi_M^3(C),\dots$.

 Recall that the preimage of $\bdy Q_\nodal$ under the almost toric fibration $\pi: \atf(Q_\nodal) \ra Q_\nodal$ is an anticanonical nodal symplectic divisor. If we assume that $Q_\nodal$ is constructed such that there are base-nodes at every vertex except for the Delzant vertex $\vv$ (c.f. Remark~\ref{rmk:nodal_trade}), then $\calN := \pi^{-1}(\bdy Q_\nodal)$ is rational with a single node at $\vv$. 
The visible symplectic curve $C$ in Proposition~\ref{prop:uni_symp_curve} is by construction $(r,a)$-well-placed with respect to $\calN$, and the last part of Theorem~\ref{thm:uni_alg_curve} is an algebraic analogue of this.
\end{rmk}

\sss

Theorem~\ref{thmlet:nodal} follows quickly from Theorem~\ref{thm:uni_alg_curve}, after a preliminary lemma.

\begin{lemma}\label{lem:pert_non_Fano}
Let $V$ be a smooth complex projective surface which is diffeomorphic to a rigid del Pezzo surface $M$. Then $V$ has an arbitrarily small deformation $\wt{V}$ which is biholomorphic to $M$.
\end{lemma}
\begin{proof}
Observe that $V$ is necessarily rational (see \cite[Cor. 0.2]{Friedman-Qin}) and hence by the Enriques--Kodaira classification it is an iterated blowup of $\CP^2$ or a Hirzebruch surface $F_{k}$. 
Recall that the Hirzebruch surface $F_k$ deforms (by an arbitrarily small deformation) to $\CP^1 \times \CP^1$ or $F_1$ 
for any $k \in \Z_{\geq 2}$.
Combining this with a generic perturbation of the blowup centers, we arrive at a smooth del Pezzo surface $\wt{V}$ which is diffeomorphic and hence (by the classification of del Pezzo surfaces) biholomorphic to $M$.
\end{proof}

\begin{proof}[Proof of Theorem~\ref{thmlet:nodal}]
Let $C$ be the $(r,a)$-unicuspidal $\Jint$-holomorphic curve in $\atf(Q_\nodal)$ guaranteed by Theorem~\ref{thm:uni_alg_curve}.
Here $(\atf(Q_\nodal),\Jint)$ is biholomorphic to a $\Q$-Gorenstein smoothing $\wt{V}_Q$ of $V_Q$, and by assumption $\atf(Q_\nodal)$ is diffeomorphic to a rigid del Pezzo surface $M$. 
If $\wt{V}_Q$ is Fano then it is necessarily biholomorphic to $M$.
Otherwise, by Lemma~\ref{lem:pert_non_Fano} we can still find an arbitrarily small deformation of $\wt{V}_Q$ which is Fano and hence biholomorphic to $M$, and similar to the proof of Theorem~\ref{thm:sesqui_main} we can push $C$ into this del Pezzo surface.
\end{proof}

We end this subsection by showing that the  staircases in the  rigid del Pezzo surfaces are entirely visible in terms of triangles in certain almost toric base polygons as detailed in Corollaries~\ref{cor:vis_ell_obs} and~\ref{cor:vis_ell_obs_func}.
In particular, we use this to give an alternative proof of Theorem~\ref{thmlet:outer_corner_curves} (recall that our proof in \S\ref{sec:twist} was based on the generalized Orevkov twist).

\begin{prop}\label{prop:out_cor_vis}
For $M$ be a rigid del Pezzo surface, and let $Q$ be the corresponding dual Fano $T$-polygon in Figure~\ref{fig:smooth_Fano_polygons}.
Then, each outer corner point $(x_k, y_{k})$ of the infinite staircase in the ellipsoid embedding function $c_M(x)$
corresponds to a visible ellipsoid obstruction in some polygon $Q'$ obtained from $Q$
by a sequence of full mutations. Moreover,  as in Corollary~\ref{cor:vis_ell_obs_func}, we may choose $Q'$ so that 
 this polygon $Q'$ determines the value of  $c_M(x)$ for all $x\in [x_k, x_{k+1}]$.
\end{prop}

Let $Q$ be one of the polygons in Figure~\ref{fig:smooth_Fano_polygons}, 
and let $J$ denote the number of strands of the corresponding infinite staircase $c_M(x)$, i.e. $J=3$ for $\CP^2 \# \ovl{\CP}^2$ and $\CP^2 \#^{\times 2}\ovl{\CP}^2$ and $J=2$ in the remaining cases (c.f. \S\ref{subsec:staircase_numerics}).  Let $Q_k$ denote the result after $k \in \Z_{\geq 0}$ successive full mutations of $Q$, each time along the eigenray  emanating from its top left vertex.
It is shown in \cite{cristofaro2020infinite} that this eigenray 
always meets the horizontal edge of $Q_k$, so that the non-Delzant vertices cyclically permute under successive full mutations of this kind.
Let $\vv^k_1,\dots,\vv^k_{J+1}$ denote the vertices of $Q_k$ (ordered counterclockwise), where $\vv_1^k$ is the Delzant vertex with edge vectors $(1,0),(0,1)$, $\vv_2^k$ lies on the positive $x$-axis, and $\vv_{J+1}^k$ lies on the positive $y$-axis.
A complete description of the vertices and eigenrays of $Q_k$ may be found in \cite[\S5]{cristofaro2020infinite}.
In particular, for each $J$-tuple of successive staircase steps there is a mutation
$Q_k$ whose vertices have  $T$-singularities of the corresponding types.
 Further, 
putting  $K = [\calN] \cdot [\calN]-2$, where $[\calN] \in H_2(\atf(Q_\nodal))$ is the anticanonical class (c.f. Table~\ref{table:CG_et_al}), we find that the vertex types of $Q_k$ transform under $J$-full mutations
by the same recursion as the generalized Orevkov twist:

\begin{lemma}\label{lem:rigid_dP_mut_T_sings} 
Suppose that $\vv^k_i$ has type $\tfrac{1}{m_ir_i^2}(1,m_ir_ia_i-1)$ for $i=2,\dots,\ell+1$.
Then $\vv^{k+J}_i$ has type $\tfrac{1}{m_i'(r_i')^2}(1,m_i'r_i'a_i'-1)$, where $(m_i',r_i',a_i') = (m_i,Kr_i-a_i,r_i)$.
\end{lemma}

\begin{proof}[Proof of Proposition~\ref{prop:out_cor_vis}]
Let $(x_k,y_k) = (\tfrac{g_{k+J}}{g_k},\tfrac{g_{k+J}}{g_k + g_{k+J}})$ be the $k$th outer corner point on the graph of $c_M$ (recall \S\ref{subsec:staircase_numerics}). By \cite[\S5]{cristofaro2020infinite},
the eigenray emanating from the top left vertex $\vv^k_{J+1}$ of $Q_k$ points in the direction 
$(g_{k},-g_{k+J})$ and meets the edge between $\vv^k_{1}$ and $\vv^k_{2}$. 
According to Proposition~\ref{prop:uni_symp_curve} there is a visible index zero $(g_{k+J},g_k)$-unicuspidal rational symplectic curve $C$ in $M$, and this gives precisely the outer corner obstruction $c_M(x_k) \geq y_k$. 
Now observe that, because $Q_k$ is obtained by mutation along the nodal ray from the top left vertex $\vv^{k-1}_{J+1}$ of $Q_{k-1}$, this nodal ray must meet the $x$-axis at  the vertex $\vv^k_2$ of $Q_k$ so that the nodal ray of $Q_k$ at 
$\vv^k_2$  points in the direction $(-g_{k-1}, g_{k-1+J})$.
If we now apply Corollary~\ref{cor:vis_ell_obs_func} to $Q_k$, then because $\ell_1 > \ell_2+\ell_3$
we can calculate $c_M$ on the interval that ends on the $k$th outer corner point $x_k = \tfrac{g_{k+J}}{g_k}$.
Similarly, by applying this result to the reflection of $Q_k$ about the line $x=y$, we are in the case $\ell_1 < \ell_2$ and hence can deduce that $c_M$ is constant on the interval between $x_{k-1}$ and the corner point. 
\end{proof}

\begin{proof}[Alternative proof of Theorem~\ref{thmlet:outer_corner_curves}]
By Proposition~\ref{prop:out_cor_vis}, every outer corner in $c_M(x)$ corresponds to a visible ellipsoid obstruction in some polygon $Q$ with $\atf(Q_\nodal)$ symplectomorphic to $M$, and by Theorem~\ref{thmlet:nodal} this comes from an index zero rational algebraic unicuspidal curve.
\end{proof}

\subsection{Visible curves in the Auroux-type model}\label{subsec:aur_vis}

As a preliminary to constructing curves in $\Q$-Gorenstein smoothings of singular toric surfaces, we first discuss the affine case, which is modeled on the Auroux-type system from \cite[\S7.3]{evans2023lectures}.
Recall from \S\ref{subsubsec:T-sings} that the $T$-singularity $\tfrac{1}{mr^2}(1,mra-1) = \{xy=z^{rm}\} / \roots_r^{1,-1,a}$
smooths to $$
B_{m,r,a} \cong \{xy = (z^r-\zeta_1)\cdots(z^r-\zeta_m)\}/\roots_r^{1,-1,a}
$$
 (for any fixed $0 < \zeta_1 < \cdots < \zeta_m$), and we have the Auroux-type almost toric fibration $$
 \pi_\aur: B_{m,r,a} \ra \C, \quad \pi_\aur(x,y,z) = \left(|z|^2,\tfrac{1}{2}|x|^2 - \tfrac{1}{2}|y|^2 \right).
 $$ 
The critical values of $\pi_\aur$ are $(\zeta_1^{2/r},0),\dots,(\zeta_m^{2/r},0)$.
By analogy with the visible unicuspidal symplectic curves appearing in the proof of Proposition~\ref{prop:uni_symp_curve}, we seek algebraic curves in $B_{m,r,a}$ which project via $\pi_\aur$ to vertical rays in $\R_{\geq 0} \times \R$ emanating from a critical value.
To this end, putting $\intC_i^+ := \{y = 0\} \subset B_{m,r,a}$ and $\intC_i^- := \{x = 0\} \subset B_{m,r,a}$, note that we have indeed $$\pi_\aur(\intC_i^\pm) = \{(\zeta_i^{2/r},t) \;|\; \pm t \in \R_{\geq 0}\}.
$$

Now suppose that $Q$ is a polygon with a $T$-vertex $\ww$ of type $\tfrac{1}{mr^2}(1,mra-1)$,
let $\wt{V}_Q$ be a $\Q$-Gorenstein smoothing of the singular toric surface $V_Q$, and assume that we have a holomorphic embedding $\iota: B_{m,r,a} \hookrightarrow \wt{V}_Q$.\footnote{Strictly speaking it will suffice to have an embedding of a suitable open subset of $B_{m,r,a}$ (c.f. the proof of Proposition~\ref{prop:QG_def_diff_ATF}), but for ease of exposition will we usually assume that we have a full embedding of $B_{m,r,a}$, which will indeed be the case in our explicit constructions.}
Roughly speaking, we will show below that the closure of $\iota(\intC_i^+)$ in $\wt{V}_Q$ 
is a $(r,a)$-unicuspidal rational algebraic curve, and this underpins Theorem~\ref{thm:uni_alg_curve}.

\subsection{Pencils from polygon mutations}\label{subsec:QG_pencils}

As we briefly recalled in \S\ref{subsubsec:T-sings}, singular toric surfaces $V_Q$ and $V_{Q'}$ are $\Q$-Gorenstein deformation equivalent if (and conjecturally only if) the corresponding dual Fano polygons are mutation equivalent.
More precisely, we have:

\begin{thm}[{\cite[Lem. 7]{akhtar2016mirror}, following \cite[Thm. 1.3]{ilten2012mutations}}]
Let $Q$ be a dual Fano polygon, and let $Q' = \mut_\ww(Q)$ be its mutation at a vertex $\ww$.
There exists a $\Q$-Gorenstein pencil\footnote{That is, a flat family whose total space is $\Q$-Gorenstein.} $\penc: \calX \ra \CP^1$ such that $\penc^{-1}
(0) \cong V_Q$ and $\penc^{-1}(\infty) \cong V_{Q'}$.
\end{thm}

In order to construct unicuspidal algebraic curves, we will describe an explicit model for (at least a part of) the $\Q$-Gorenstein pencil $\penc$ in the analogous case that $Q$ is a $T$-polygon (not necessarily dual Fano) with a Delzant vertex.
We first discuss the case that $Q$ is a triangle, and then obtain the case of a general polygon by a local birational modification.

Let $Q \subset \MM_\R$ be a triangle with a Delzant vertex $\vv$ adjacent to a $T$-vertex $\ww$.
After an integral affine transformation we can assume that $\MM_\R = \R^2$ and the vertices are $\vv := (0,0),\ww:= (0,mra-1),\uu := (mr^2,0)$, for some $m,r,a \in \Z_{\geq 1}$ with $\gcd(r,a) = 1$. In particular, $V_Q$ is isomorphic to the weighted projective space $\CP(1,mra-1,mr^2)$.
Similar to Remark~\ref{rmk:cusp_type_depends_on_vs}, we do not assume $a < r$, but we can write $a = a' + \vs r$ for some $a' \in \{1,\dots,r-1\}$ and $\vs \in \Z_{\geq 1}$.
Then the vertex $\ww$ has type $\tfrac{1}{mr^2}(1,mra-1) = \tfrac{1}{mr^2}(1,mra'-1)$, and the vertex $\uu$ has type $\tfrac{1}{mra-1}(1,mr^2)$ (this is not necessarily a $T$-singularity).

With respect to the vertex $\ww$, the mutated triangle $Q' := \mut^\full_\ww(Q)$ has vertices $\vv' := (0,0),\ww' := (0,mra),\uu' := (\tfrac{r}{a}\cdot (mra-1),0)$, and we have 
$V_{Q'} \cong \CP(1,mra-1,ma^2)$.
Note that $\vv'$ is smooth, $\ww'$ has type $\tfrac{1}{ma^2}(1,mra-1)$, and $\uu'$ has type 
$\tfrac{1}{mra-1}(1,ma^2)$, which is the same singularity type as $\uu$:

\begin{lemma}\label{lem:T_sings_agree}
For $m,r,a \in \Z_{\geq 1}$ with $\gcd(r,a) = 1$, we have $\tfrac{1}{mra-1}(1,mr^2) = \tfrac{1}{mra-1}(r,a) = \tfrac{1}{mra-1}(1,ma^2)$.  
\end{lemma}

\begin{proof}
 Modulo $mra-1$ we have $0 \equiv (mra-1)(mra+1) = m^2r^2a^2 - 1$, and hence $mr^2 = 1/(ma^2)$. 
\end{proof}

 Now consider the weighted projective $3$-space $\CP(1,mra-1,r,a)$ with homogeneous coordinates $[x:y:z:w]$, and consider the hypersurface
\begin{align}\label{eq:penc_tri}
S_t := \{xy = \tfrac{1}{1+t}w^{mr} + \tfrac{t}{1+t} z^{ma}\} \subset \CP(1,mra-1,r,a)
\end{align}
for $t \in \CP^1$.

\begin{prop}\label{prop:pencil_tri_case} Suppose as above that the triangle $Q$ has 
a Delzant vertex $\vv$ adjacent to a $T$-vertex $\ww$. Then
the family $\{S_t\}_{t \in \CP^1}$ defines a $\Q$-Gorenstein 
pencil such that
\begin{itemize}
  \item we have isomorphisms $S_0 \cong \CP(1,mra-1,mr^2)$ and $S_\infty \cong \CP(1,mra-1,ma^2)$
  \item for $t \neq 0,\infty$, $S_t$ has a singularity at the point $[0:1:0:0]$ of type $\tfrac{1}{mra-1}(r,a)  $ and is otherwise nonsingular.
\end{itemize}
\end{prop}

 \begin{rmk}
  When $Q$ is a $T$-triangle, the family $\{S_t\}$ is an algebraic counterpart of a family of almost toric fibrations interpolating between the singular toric surfaces $V_Q$ and $V_{Q'}$,
 corresponding to a family of nodal integral affine structures on $Q$ with base-nodes limiting to the vertex $\ww$ as $t \ra 0$ and limiting to $\uu'$ as $t \ra \infty$ (here $\uu'$ is the other
 point where the eigenray emanating from $\ww$ intersects $\bdy Q$).
 \end{rmk}

\begin{proof}[Proof of Proposition~\ref{prop:pencil_tri_case}]
Let $U_x$ denote the (singular) affine chart $\{x=1\} \subset \CP(1,mra-1,r,a)$, with $U_y,U_z,U_w \subset \CP(1,mra-1,r,a)$ defined similarly.
We will analyze the intersection of $S_t$ with each of the charts $U_x,U_y,U_z,U_w$.
To start,  observe that, for any $t \in \CP^1$, $S_t \cap U_x = \{y = \tfrac{1}{1+t}w^{mr} + \tfrac{t}{1+t}z^{ma}\} \cong \C_{z,w}^2$ is smooth.
We have also 
\begin{align}
S_t \cap U_y = \{x = \tfrac{1}{1+t}w^{mr} +\tfrac{t}{1+t}z^{ma}\} / \roots_{mra-1}^{1,r,a} \cong \C^2_{z,w}/\roots_{mra-1}^{r,a},
\end{align}
where the latter is the model $\tfrac{1}{mra-1}(r,a)$ singularity.

Next, we have
\begin{align*}
S_t \cap U_z = \{xy = \tfrac{1}{1+t}w^{mr}+\tfrac{t}{1+t}\} / \roots_{r}^{1,-1,a}.
\end{align*}
Comparing with \eqref{eq:T_sing_miniversal}, we see that the family $\{S_t \cap U_z\}$ is a $\Q$-Gorenstein smoothing of $S_0 \cap U_z \cong \tfrac{1}{mr^2}(1,mra-1)$, and in particular $S_t \cap U_z \cong B_{m,r,a}$ for $t \neq 0,\infty$.
Note that $S_\infty \cap U_z = \{xy = 1\}/\roots_{r}^{1,-1,a}$ is smooth.

Finally, we have 
\begin{align*}
S_t \cap U_w = \{xy = \tfrac{1}{1+t} + \tfrac{t}{1+t}z^{ma}\} / \roots_a^{1,-1,r},
\end{align*}
i.e. the family $\{S_t \cap U_w\}$ is a $\Q$-Gorenstein smoothing of $S_\infty \cap U_w \cong \tfrac{1}{ma^2}(1,mra-1)$, and we have $S_t \cap U_w \cong B_{m,a,r}$ for $t \neq 0,\infty$, with $S_0 \cap U_w = \{xy = 1\}/\roots_a^{1,-1,r}$ smooth.
Note that we have covered the total space of the deformation $\{S_t\}$ by $\Q$-Gorenstein neighborhoods (each is either smooth, equivalent to the miniversal model \eqref{eq:T_sing_miniversal}, or a trivial family of cyclic quotient singularities), and hence $\{S_t\}$ is a $\Q$-Gorenstein deformation.

Let $[s:t:u]$ be homogeneous coordinates on $\CP(1,mra-1,mr^2)$, and define the map 
\begin{align*}
\iota_0: \CP(1,mra-1,mr^2) \ra S_0,\;\;\;\;\; [s:t:u] \mapsto [s^{mr}:t^{mr}:u:st]. 
\end{align*}
Note that this is well-defined since 
$\iota_0([\la s: \la^{mra-1}t:\la^{mr^2}u]) = [\la^{mr}s^{mr}:\la^{mr(mra-1)}t^{mr}:\la^{mr^2}u:\la^{mra}st] = [s^{mr}:t^{mr}:u:st]$, 
and the restriction $\CP(1,mra-1,mr^2) \cap U_u \ra S_0 \cap U_z$ agrees with the isomorphism \eqref{eq:Milnor_quot_iso}.
Similarly, we define the map
\begin{align*}
\iota_\infty: \CP(1,mra-1,ma^2) \ra S_\infty,\;\;\;\;\; [s:t:u] \mapsto [s^{ma}:t^{ma}:st:u]. 
\end{align*}
We leave it to the reader to verify that the maps $\iota_0$ and $\iota_\infty$ are isomorphisms.
\end{proof}

\begin{example}
Let $(p_1,p_2,p_3) \in \Z_{\geq 0}^3$ be a triple which satisfies the Markov equation $p_1^2 + p_2^2 + p_3^2 = 3p_1p_2p_3$, 
and let $(p_1,p_2,p_3' := 3p_1p_2-p_3)$ be its Markov mutation.
Put $\MM := \{(m_1,m_2,m_3) \in \Z^3\;|\; \sum\limits_{i=1}^3 p_i^2m_i = 0\}$.
Then the dual Fano polygon $$
Q = \{(x_1,x_2,x_3) \in \R^3 \;|\; \sum\limits_{i=1}^3 p_i^2x_i = 0, x_1,x_2,x_3 \geq -1\} \subset \MM_\R$$
 satisfies $V_Q \cong \CP(p_1^2,p_2^2,p_3^2)$ and has $T$-vertices of types $\tfrac{1}{p_i^2}(p_{i+1}^2,p_{i+2}^2)$ for $i=1,2,3$ (mod $3$).
Let $\ww$ denote the vertex of type $\tfrac{1}{p_3^2}(p_1^2,p_2^2)$ and put $Q' := \mut^\full_\ww(Q)$, so that we have $V_{Q'} \cong \CP(p_1^2,p_2^2,(p_3')^2)$.
Then the   $\Q$-Gorenstein pencil $\{S_t\}_{t \in \CP^1}$    given by
\begin{align}
S_t := \{xy = \tfrac{1}{1+t}w^{p_3} + \tfrac{t}{1+t}z^{p_3'}\} \subset \CP(p_1^2,p_2^2,p_3,p_3')
\end{align}
satisfies  $S_0 \cong \CP(p_1^2,p_2^2,p_3^2)$, $S_\infty \cong \CP(p_1^2,p_2^2,(p_3')^2)$ and $S_t \cong \CP^2$ for $t \neq 0,\infty$.
This fits into the above framework after specializing to the case $p_1 = 1$, so that $Q$ has a Delzant vertex $\vv$.
\end{example}

Now let $Q \subset \MM_\R$ be any polygon having a Delzant vertex $\vv$ adjacent to a $T$-vertex $\ww$.
As in the triangle case considered above, we will assume that $\MM_\R = \R^2$ with $\vv = (0,0)$ and $\ww = (0,mra-1)$, where the edge vectors at $\vv$ are $(1,0),(0,1)$ and the edge vectors at $\ww$ are $(0,-1),(mr^2,1-mra)$, for some $m,r,a \in \Z_{\geq 1}$ with $\gcd(r,a) = 1$.
We will further assume that the eigenray emanating from $\ww$ hits the edge of $Q$ which lies on the $x$-axis.\footnote{Without this assumption we will still have $\wt{S}_0 \cong V_Q$, which suffices to prove Theorem~\ref{thm:uni_alg_curve}, but we will not generally have $\wt{S}_\infty \cong V_{Q'}$. Indeed, since the complex variety $V_Q$ depends only on the normal fan of $Q$ we can always adjust $Q$ so as to achieve this property without changing its normal fan. However, the mutation $Q'$ depends on $Q$ itself and not just its normal fan.}
Put $Q' := \mut_\ww^\full(Q)$.
By reducing to the triangular case above, we now construct a $\Q$-Gorenstein pencil $\{\wt{S}_t\}_{t \in \CP^1}$ which interpolates between $V_Q$ and $V_{Q'}$ and smooths the toric fixed point $\pp_\ww$ in a general fiber.
Let $Q_\tri$ denote the triangle with vertices $\vv,\ww$, and $\uu = (mr^2,0)$, so that we have $Q \subset Q_\tri$.

Let $\Si_Q,\Si_{Q_\tri}$ denote the normal fans to $Q,Q_\tri$ respectively, and note that $\Si_Q$ is a refinement of $\Si_{Q_\tri}$.
We denote by $\tau_{[\vv,\ww]},\tau_{[\vv,\uu]},\tau_{[\uu,\ww]} \subset \NN_\R$ the rays of $\Si_{Q_\tri}$ spanned by the inward normals to the edges $[\vv,\ww],[\vv,\uu],[\uu,\ww]$ respectively. Let $\si \subset \NN_\R$ denote the two-dimensional cone spanned by $\tau_{[\vv,\uu]}$ and $\tau_{[\uu,\ww]}$, with corresponding distinguished point $\pp_\si \in 
V_{Q_\tri}$ (c.f. \cite[\S3.2]{cox2011toric}).
Note that the rays of $\Si_{Q} \setminus \Si_{Q_\tri}$ all lie in the interior of $\si$.
The induced toric morphism $\pi: V_Q \ra V_{Q_\tri}$ is a proper birational map which restricts to an isomorphism 
$V_Q \setminus \pi^{-1}(\pp_\si) \ra V_{Q_\tri} \setminus \pp_\si$.
We denote by $U_\si \subset V_{Q_\tri}$ the affine toric variety associated to the cone $\si$, which we identify with $\tfrac{1}{mra-1}(r,a)$, and we put $\wt{U}_\si := \pi^{-1}(U_\si) \subset V_Q$.
Thus $V_Q$ is obtained from $V_{Q_\tri}$ by excising $U_\si$ and gluing in $\wt{U}_\si$, which we could view as sequence of generalized blowups at (singular) points.

Let $\{S_t\}_{t \in \CP^1}$ denote the $\Q$-Gorenstein pencil associated to $Q_\tri$ and its mutation at $\ww$ as in Proposition~\ref{prop:pencil_tri_case}.
For $t \in \CP^1$, let $\jmath_t: \C^2_{z,w}/\roots_{mra-1}^{r,a} \xrightarrow{\cong} S_t \cap U_y$ denote the natural isomorphism $(z,w) \mapsto (\tfrac{1}{1+t}w^{mr} +\tfrac{t}{1+t}z^{ma},1,z,w)$.
For each $t \in \CP^1$, let $\wt{S}_t \ra S_t$ be the birational modification corresponding to excising the image of $U_\si$ under $\jmath_t$ and gluing in $\wt{U}_\si$.
We have, essentially by construction, the following generalization of Proposition~\ref{prop:pencil_tri_case}.

\begin{prop}\label{prop:pencil:gen_case} \hfill

    \begin{enumerate}[label=(\roman*)]

      \item
                 Suppose as above
that $Q \subset \MM_\R$ is a polygon having a Delzant vertex $\vv=(0,0)$ adjacent to a $T$-vertex $\ww =(0,mra-1)$ and such that
 the eigenray $(r,-a)$ emanating from $\ww$ hits the edge of $Q$ which lies on the $x$-axis.
Then the  family $\{\wt{S}_t\}_{t \in \CP^1}$ described above defines a $\Q$-Gorenstein pencil such that
\begin{itemize}
  \item we have isomorphisms $\wt{S}_0 \cong V_Q$ and $S_\infty \cong V_{Q'}$
  \item for $t \neq 0,\infty$, the singularities of $\wt{S}_t$ correspond naturally to the singular toric fixed points of $V_Q$ excluding $\pp_\ww$.
\end{itemize}
\item   If in the above situation the eigenray $(r,-a)$ emanating from $\ww$ hits some other edge of $Q$
then the $\Q$-Gorenstein pencil $\wt{S}_t$ with $\wt{S}_0 \cong V_Q$  is defined for $t$ close to $0$ and $\wt{S}_t$ has singularities as described above.
\end{enumerate}
\end{prop}

\begin{rmk}\label{rmk:mutation}
In \cite{akhtar2016mirror} the authors construct a family of hypersurfaces in a toric $3$-fold $V_{\wt{Q}}$, where $\wt{Q}$ is a three-dimensional polytope admitting projections to $Q$ and the mutated polytope $Q'$.  This polytope has a smooth corner that projects
to the smooth corners of $Q$ and $Q'$, and
the dually induced toric morphisms $V_Q,V_{Q'} \hookrightarrow V_{\wt{Q}}$  embed $V_Q$ and $V_{Q'}$ as hypersurfaces that  span a $\Q$-Gorenstein pencil.   This construction applies whether or not
the eigenray from $\ww$ meets the edge of $Q$ on the $x$-axis, and it follows that there is a family of surfaces $(\wt{S}_t)_{ t\in \C P^1}$ with $\wt{S}_0 = V_Q$ and $\wt{S}_\infty = V_{Q'}$.
 In this construction, each additional facet of $Q$ gives rise to an extra facet of $\wt{Q}$, so that the effect of blowing up $Q$ is to blow up $\wt{Q}$.  However this blowup of $\wt{Q}$ may not correspond to a family of blowups of the $\wt{S}_t$.  
 As an example, one can consider the effect of a mutation on the quadrilateral illustrated in \cite[Fig. 1.3]{magill2022unobstructed}.
 Here the original quadrilateral has the property that the nodal ray from $\vv_Y$ hits the edge $\vv_V - \vv_X$.
 One can check that this has the effect that 
 one of the edges of the  mutation of $Q_{\tri}$ is completely excised from the mutation of   $Q$, which means the latter mutation is not a blowup of the former mutation.
Nevertheless, in this case
 one can  with some effort  give explicit formulas in the spirit of \eqref{eq:penc_tri} for the hypersurface $\wt{S}_t \subset V_{\wt{Q}}$ in terms of Cox coordinates on $V_{\wt{Q}}$ thought of as a GIT quotient.
\end{rmk}

\subsection{Explicit unicuspidal algebraic curves}\label{subsec:expl_uni_alg}

Let $Q$ be a polygon as in Proposition~\ref{prop:pencil:gen_case}, with associated $\Q$-Gorenstein pencils $\{\wt{S}_t\}_{t \in \CP^1}$.
We now construct a $(r,a)$-unicuspidal rational algebraic curve $C_t^+$ in $\wt{S}_t$ for all $t \neq 0,\infty$.
Since $\{\wt{S}_t\}_{t \in \CP^1}$ is given by a fiberwise birational modification of the pencil $\{S_t\}_{t \in \CP^1}$ associated to $Q_\tri$, we first consider the case that $Q = Q_\tri$ is a triangle. 
Notice that the smooth toric fixed point of $\CP(1,mra-1,mr^2)$ maps by $\iota_0$ to the point $[1:0:0:0] \in S_0$, which lies in $S_t$ for all $t \in \CP^1$.
The curve $C_t^+$ will be constructed to have a $(r,a)$ cusp at $[1:0:0:0]$, and will furthermore be well-placed with respect to an anticanonical divisor $\calN_t$ having a node at $[1:0:0:0]$. 

Recall that for $t \neq 0,\infty$ we have 
$S_t \cap U_z = \{xy = \tfrac{1}{1+t}w^{mr} + \tfrac{t}{1+t}\} / \roots_r^{1,-1,a} \cong B_{m,r,a}$.
Following \S\ref{subsec:aur_vis}, we consider the curves $\intC^+_t = \{y = 0\} \subset S_t \cap U_z$ and $\intC^-_t = \{x = 0\} \subset S_t \cap U_z$ and their closures $C^+_t,C^-_t$ in $S_t$.
We will focus on $C_t^+$ since $C^-_t$ does not pass through $[1:0:0:0]$.
Strictly speaking $\intC^+_t$ has $m$ components (since we have combined $C_1^+,\dots,C_m^+$ from \S\ref{subsec:aur_vis} into a single curve), so we take one of the components
\begin{align}
C^+_t := \{y = 0, w^r = \zeta z^a\} \subset S_t
\end{align}
for fixed $\zeta \in \C$ satisfying $\zeta^m = -t$.

For $t \in \CP^1$, let $\calN_t \subset S_t$ denote the anticanonical divisor $\{z = 0\} \cup \{w = 0\} \subset S_t$.
Here $\calN_t$ has two components for $t \neq 0,\infty$, while $\calN_0$ (resp. $\calN_\infty$) is the image of the toric divisor under the map $\iota_0: \CP(1,mra-1,mr^2) \ra S_0$ (resp. $\iota_\infty: \CP(1,mra-1,ma^2) \ra S_\infty$).
For $t \neq 0,\infty$, the two components of $\calN_t$ intersect transversally at the nodal point $[1:0:0:0]$ and also meet at the singular 
point $[0:1:0:0]$.

\begin{lemma}\label{lem:C^+_well_placed_tri}   When $Q$ is a triangle,
 $C^+_t \subset S_t$ is $(r,a)$-well-placed with respect to $\calN_t$ for  $t \neq 0,\infty$, and is otherwise nonsingular.
 In particular, $C_t^+$ is $(r,a)$-unicuspidal.
\end{lemma}

\begin{proof}
Note that $C^+_t \subset U_x \cup U_z$ since  $C^+_t \cap U_y = \nil$ and $[0:0:0:1] \notin S_t$, and $C^+_t \cap U_z$ is smooth by construction.
In the chart $S_t \cap U_x \cong \C^2_{z,w}$ we have $C^+_t \cap U_x = \{(z,w) \in \C^2 \;|\; w^r = \zeta z^a\}$, which is $(r,a)$-well-placed with respect to $\calN_t \cap U_x = \{(z,w) \in \C^2 \;|\; z = 0\;\text{or}\; w= 0\}$.
\end{proof}

In the case of a general quadrilateral $Q$, let $\wt{C}_t^+ \subset \wt{S}_t$ denote the proper transform of $C_t^+ \subset S_t$ with respect to the birational map $\wt{S}_t \ra S_t$, and let $\wt{\calN}_t \subset \wt{S}_t$ be the total transform of $\calN_t$.
Since $C_t^+$ is disjoint from $[0:1:0:0]$, we have the following extension of Lemma~\ref{lem:C^+_well_placed_tri}.

\begin{lemma}\label{lem:well_placed_gen} For general $Q$ and 
 all $t\ne 0$ sufficiently close to $0$, the curve $\wt{C}_t^+ \subset \wt{S}_t$ is $(r,a)$-well-placed with respect to $\wt{\calN}_t$ and is otherwise nonsingular. 
\end{lemma}

\sss

We are now ready to complete the proof of Theorem~\ref{thm:uni_alg_curve}.

\begin{proof}[Proof of Theorem~\ref{thm:uni_alg_curve}]  
Let $\{\wt{S}_t\}_{t\approx 0}$
be the $\Q$-Gorenstein deformation of $V_Q$ constructed  in Proposition~\ref{prop:pencil:gen_case}. Using Lemma~\ref{lem:well_placed_gen}, there is a rational unicuspidal curve $\wt{C}^+_t \subset \wt{S}_t$ which is well-placed with respect to the anticanonical divisor $\wt{\calN}_t$.
Let $\ttil{S}_t$ be a further $\Q$-Gorenstein deformation of $\wt{S}_t$ which smooths the remaining singular points.
As in the proof of Theorem~\ref{thm:sesqui_main}, for $|t| > 0$ sufficiently small we can assume that $\wt{\calN}_t$ deforms to a rational nodal anticanonical divisor $\ttil{\calN}_t \subset \ttil{S}_t$ and $\wt{C}_t^+$ deforms to a rational curve $\ttil{C}_t \subset \ttil{S}_t$ which is $(r,a)$-well-placed with respect to $\ttil{\calN}_t$.\footnote{See \cite[Cor.3.5.5]{cusps_and_ellipsoids} for a discussion of the behavior of cuspidal constraints under deformations of the (almost) complex structure $J$.} 
By Proposition~\ref{prop:QG_def_diff_ATF},  there is a diffeomorphism $\Phi: \ttil{S}_t \ra \atf(Q_\nodal)$ such that $\Phi_*(\ttil{J})$ tames the symplectic form on $\atf(Q_\nodal)$, with $\ttil{J}$ the integrable almost complex structure on $\ttil{S}_t$.
Then the curve $\Phi(\ttil{C}^+_t) \subset \atf(Q_\nodal)$ is $\Phi_*(\ttil{J})$-holomorphic and satisfies the requirements of  the theorem. 
\end{proof}

\subsection{Classifying unicuspidal algebraic curves in the first Hirzebruch surface}\label{subsec:unicusp_appl}

Our goal here is to prove Theorem~\ref{thmlet:F_1_classif} on unicuspidal algebraic curves in the first Hirzebruch surface $F_1$.
As a warmup, we start by discussing the analogue for $\CP^2$, giving a new proof based on quantitative symplectic geometry that the list in Theorem~\ref{thm:bob_et_al}(d) is complete:
\begin{lemma}\label{lem:d_is_complete}
The curves constructed in Theorem~\ref{thm:orev_orig} give the complete list of data $(d,p,q)$ for unicuspidal rational algebraic plane curves with one Puiseux pair.
\end{lemma}
\begin{proof}
Any $(p,q)$-unicuspidal rational algebraic curve in $\CP^2$ of degree $d$ is in particular a symplectic curve, and hence according to \cite[Thm. G]{cusps_and_ellipsoids} the homology class $A = d\ell$ is $(p,q)$-perfect.
Then by \cite[Cor. 3.1.3]{McDuff-Schlenk_embedding}, $p/q$ must be a ratio of odd index Fibonacci numbers.
\end{proof}

The corresponding classification of perfect exceptional classes in $F_1$ was recently worked out in \cite{magillmcd2021,magill2022staircase} and is much more complicated than for $\CP^2$. Here we give a broad self-contained overview of the classification, refering to loc. cit. for the details.

Let $\perf(F_1)$ denote the set of all quadruples $(p,q,d,m) \in \Z_{\geq 0}^4$ with $p > q$ coprime such that $A = d\ell - me$ is a $(p,q)$-perfect exceptional class in $H_2(F_1)$ (see e.g. \cite[Def. 4.4.2]{cusps_and_ellipsoids}). Equivalently, by \cite[Thm. G]{cusps_and_ellipsoids} this is the set of quadruples $(p,q,d,m)$ such that there exists an index zero $(p,q)$-unicuspidal rational symplectic curve $C$ in $F_1$ with $[C] = d\ell - me \in H_2(F_1)$.\footnote{If such a curve exists then it exists for any symplectic form on $F_1$. However, it gives an interesting obstruction to embedding ellipsoids into $F_1$ with symplectic form in class $\PD(\ell- be)$ only if  $m/d\approx b$.}
Put 
\begin{align*}
\perfbar(F_1) := \{p/q \;|\;(p,q,d,m) \in \perf(F_1)\text{ for some } d,m\}.
\end{align*}
The classes in $\perf(F_1)$ divide naturally into two sets: those in $\perf^+(F_1)$ with $m/d> 1/3$ and those in $\perf^-(F_1)$ with $m/d < 1/3$; none have $m/d = 1/3$.

One can show that the forgetful map $$
\perf(F_1) \ra \perfbar(F_1) = \perfbar^+(F_1)\sqcup \perfbar^-(F_1)
$$
 sending $(p,q,d,m)$ to $p/q$ is injective. Indeed, by definition, for $(p,q,d,m) \in \perf(F_1)$ we have $d^2 - m^2 = pq-1$ and $3d-m = p+q$, and hence
\begin{align}\label{eq:tpq}
d_{p,q} = \tfrac{1}{8}(3p+3q + \eps t_{p,q}),\quad m_{p,q} = \tfrac{1}{8}(p+q + 3\eps t_{p,q})
\end{align} 
where $t_{p,q} := \sqrt{p^2 - 6pq + q^2 + 8}$ and there is a unique choice $\eps \in \{1,-1\}$ such that $d_{p,q}$ and $m_{p,q}$ are integers (see \cite[\S2.2]{magillmcd2021}).

We define the ``shift'' map 
\begin{align}\label{eq:F_1_shift_map}
S: (1,\infty) \ra (5,6), \;\;\;\;\; 
x \mapsto \frac{6x-1}{x},
\end{align}
and one can check that the intervals $S^k([6,\infty))$ are disjoint for $k \in \Z_{\geq 0}$, with union 
$\bigcup_{k \in \Z_{\geq 0}}  S^k([6,\infty)) =(3+2\sqrt{2},\infty)$. 
Note also that $S$ fixes the accumulation point $3+2\sqrt 2$ of the monotone staircase, and acts on 
 the $x$-coordinates $\tfrac{g_{j+3}}{g_j}$ of its steps (outer corner points)
  by $\tfrac{g_{j+3}}{g_j}\mapsto \tfrac{g_{j+6}}{g_{j+3}}$, where $g_1,g_2,g_3,\dots$ is the sequence defined by the recursion $g_{j+6} = 6g_{j+3}-g_j$ with initial values $1,1,1,1,2,4$ as in Table~\ref{table:CG_et_al}.

We also define the ``reflection'' map
\begin{align}\label{eq:F_1_ref_map}
R: (6,\infty) \ra (6,\infty),\;\;\;\;\; 
x  \mapsto \frac{6x-35}{x-6},
\end{align}
which is an involution fixing $7$ and interchanging $(6,7)$ with $(7,\infty)$.
It turns out that both $S$ and $R$ are symmetries of the set $\perfbar(F_1)$. Further, both symmetries take the classes with $m/d> 1/3$ to those with $m/d<1/3$, and vice versa.  Thus they interchange the sets 
$ \perfbar^+(F_1)$ and $ \perfbar^-(F_1)$.  

The following is a rough summary\footnote
{
The proof in \cite{magill2022staircase} that this list is complete uses ideas that are rather different from those presented here.
Indeed an essential input is the staircase accumulation point theorem in \cite{cristofaro2020infinite}. It would be interesting to know if there is another approach
to this classification result.}
 of the set $\perfbar(F_1)$ (see \cite[Fig.2.2]{magill2022staircase} for an illustration).

\begin{thm}[\cite{magill2022staircase,magillmcd2021}]\label{thm:F_1_perfbar_summary}
We have:
\begin{itemize}
  \item $\perfbar(F_1) \cap [1,3+2\sqrt{2}) = \{\tfrac{g_{j+3}}{g_j}\;|\; j \in \Z_{\geq 1}\}$;
  \item 
  for all $k \in \Z_{\geq 0}$ we have $\perfbar^+(F_1) \cap [2k+6, 2k+7] = \{2k+6\}$,  
and there is a homeomorphism $(2k+7,2k+8) \cong (-\tfrac{1}{2},\tfrac{3}{2})$ such that the image of $\perfbar^+(F_1) \cap (2k+7,2k+8)$ is the set of rational numbers in $(0,1)$ whose ternary expansion is finite and ends in $1$;
  \item $\perfbar(F_1) \cap (6,7) = \perfbar^-(F_1) \cap (6,\infty) = \{R(p/q)\;|\; p/q \in \perfbar(F_1)\cap (7,\infty)\}$;
  \item $\perfbar(F_1) \cap (3+2\sqrt{2},6) = \{S^i(p/q)\;|\; i \in \Z_{\geq 1}, p/q \in [6,\infty) \cap \perfbar(F_1)\}$.
\end{itemize}
\end{thm}

\begin{proof}[Proof of Theorem~\ref{thmlet:F_1_classif}]

We need to show that for every $p/q \in \perfbar(F_1)$ there is an index zero $(p,q)$-unicuspidal rational algebraic curve in $F_1$. 
Recall that $\perfbar(F_1) \cap (1,3+2\sqrt{2})$ corresponds precisely to the $x$-coordinates of the outer corners of the infinite staircase in $c_{\calH_1}(x)$ (i.e. the monotone symplectic form), which is covered by Theorem~\ref{thmlet:outer_corner_curves}.

When $p/q \in \perfbar^+(F_1) \cap [6,\infty)$, it follows from Magill~\cite{magill2022unobstructed} that 
there is a $T$-quadrilateral $Q \subset \R^2$ with a Delzant vertex $\vv$ that is adjacent to a vertex $\ww$ of type $\tfrac{1}{q^2}(1,pq-1)$. More precisely, we can take the edge directions at $\vv$ to be $(1,0),(0,1)$ and the edge directions at $\ww$ to be $(0,-1), (q^2,1-pq)$. 
  Moreover, because $Q$ is constructed by mutation from a moment image of $F_1$ with suitable symplectic form, it follows that  $\atf(Q_\nodal)$ is diffeomorphic to $F_1$. 
Thus, Theorem~\ref{thm:uni_alg_curve} shows that  for every $p/q\in
 \perfbar^+(F_1)\cap [6,\infty)$ there is a well-placed $(p,q)$-unicuspidal rational algebraic\footnote{As mentioned in footnote \ref{footnote:Chow}, this holomorphic curve is algebraic by  Chow's theorem.} curve in $F_1$.
By Theorem~\ref{thm:F_1_perfbar_summary}, we may therefore apply the reflection symmetry $R$ in Proposition~\ref{prop:two_twists}
 to conclude that there is such a curve
for every $p/q\in  \perfbar(F_1)\cap [6,\infty)$. Finally, Theorem~\ref{thm:F_1_perfbar_summary} implies that we may find such a curve for every $p/q\in  \perfbar(F_1)\cap (3+2\sqrt{2},\infty)$ by repeated applications of the Orevkov twist (or shift) $S$ to the curves already constructed.
\end{proof}

\begin{rmk} In the first version of this paper we proved that the elements $p/q \in \perfbar^+(F_1) \cap [6,7)$ have suitable representatives by extending the constructions in Magill~\cite{magill2022unobstructed} to cover these cases.  We omit these arguments here because they are long and take us rather far afield.  Note that  Magill~\cite{magill_in_prep} shows that a suitable $T$-polygon can be constructed for every element in
$\perfbar(F_1)\cap (3+2\sqrt{2},\infty)$.
\end{rmk}

\begin{rmk}[Relation to $\CP^1\times \CP^1$]\label{rmk:rel_to_CP1xC^1}
The analogue of Theorem~\ref{thm:F_1_perfbar_summary} for $\CP^1\times \CP^1$ has not yet been fully worked out, but the works \cite{usher2011deformed,farley2022four} suggest that a picture similar to that of $F_1$ should hold, in which case it is likely that our proof of Theorem~\ref{thmlet:F_1_classif} also carries over mutatis mutandis to the case of $\CP^1 \times \CP^1$.
One can check that the sequence $\tfrac{a_1}{a_0},\tfrac{a_2}{a_1},\tfrac{a_3}{a_2},\dots$ defined by $a_0 = 1, a_1 = 5, a_{k+1} = 6a_k - a_{k-1}$ lies in both $\perfbar(F_1)$ and $\perfbar(\CP^1 \times \CP^1)$ (in each case giving the $x$-coordinates of the outer corners for one strand of the monotone infinite staircase), and these are likely the only points in common.
An intriguing relation between points in $\perfbar(F_1) \cap (3+2\sqrt{2},\infty)$ and $\perfbar(\CP^1 \times \CP^1) \cap (3+2\sqrt{2},\infty)$ is suggested in \cite[Conj. 1.2.1]{farley2022four}. 
\end{rmk}

\begin{rmk}[Unicuspidal curves in odd Hirzebruch surfaces] \label{rmk:HhnonFano}
For $k \in \Z_{\geq 0}$, let $F_{2k+1}$ denote the $(2k+1)$st Hirzebruch surface \cite{hirzebruch1951klasse}, which we view as a complex structure on $\bl^1\CP^2$ having an irreducible rational holomorphic curve in class $(k+1)e - k\ell \in H_2(\bl^1\CP^2)$ with self-intersection number $-2k-1$.
Note that if $A = d\ell - me \in H_2(\bl^1\CP^2)$ is represented by a holomorphic curve in $F_{2k+1}$ then by positivity of intersection we must have $\tfrac{m}{d} \geq \tfrac{k}{k+1}$.
By \cite[Lem. B.5(i)]{magill2022staircase}, any $p/q \in \perfbar(F_1) \cap (2j+6,2j+8)$ satisfies
$\tfrac{j+1}{j+2} < \tfrac{m_{p,q}}{d_{p,q}} < \tfrac{j+2}{j+3}$. 
Thus $p/q \in \perfbar(F_{2k+1})$ only if $p/q > 2k+4$.
Therefore we conjecture: {\em for all $k \in \Z_{\geq 1}$, 
there exists an index zero algebraic $(p,q)$-unicuspidal rational curve in $F_{2k+1}$ if and only if $p/q \in \perfbar^+(F_1) \cap (2k+4,\infty)$.} 
 \end{rmk}

\section{Sesquicuspidal curves and stable embeddings beyond the accumulation point}\label{sec:stab}

In this section, we prove first Theorem~\ref{thmlet:deg_3_seed} on algebraic rational plane curves corresponding to the ghost stairs in \S\ref{subsec:ghost_stairs}. 
We then discuss stabilized ellipsoid embeddings and obstructions beyond the staircase accumulation points in \S\ref{subsec:folding_curve}.

\subsection{Degree three seed curves and the ghost stairs}\label{subsec:ghost_stairs}

Our proof of Theorem~\ref{thmlet:deg_3_seed} is based on the generalized Orevkov twist from \S\ref{sec:twist}, which boils it down to finding a single degree three seed curve.

\begin{proof}[Proof of Theorem~\ref{thmlet:deg_3_seed}]

The identity $\Fib_{4k+6} = 7 \Fib_{4k+2}  - \Fib_{4k-2}$ shows that the Fibonacci subsequence $\Fib_2,\Fib_6,\Fib_{10},\dots$ obeys the same recursive formula which is achieved by the Orevkov twist $\Phi_{\CP^2}$ (c.f. Proposition~\ref{prop:gen_orev_numerics}).
Therefore it suffices to construct a suitable degree $3$ seed curve, which is the content of the following proposition.
\end{proof}

\begin{prop}\label{prop:deg_3_seeds}
 There exists a degree three rational algebraic curve in $\CP^2$ which is $(8,1)$-well-placed with respect to a nodal cubic $\calN$.
\end{prop}

\begin{proof}[Proof of Proposition~\ref{prop:deg_3_seeds}]
As in \S\ref{subsec:orev_orig}, let $\calN$ be a fixed nodal cubic in $\CP^2$ with local branches $\calB_-,\calB_+$ near the double point $\db$.
If $J$ is an almost complex structure on $\CP^2$ which is integrable near $\db$ and $\Ddiv$ is any local $J$-holomorphic divisor through $\db$, we denote by $\calM^{J}_{\CP^2,3\ell}\lll \T_{\Ddiv}^{(8)}\db\rrr$
the moduli space of $J$-holomorphic degree $3$ rational curves which pass through $\db$ with contact order $8$ to $\Ddiv$. 
The count $\# \calM^{J}_{\CP^2,3\ell}\lll \T_{\Ddiv}^{(8)}\db\rrr$ for generic $J$ was computed in \cite{McDuffSiegel_counting} to be $4$.
Note that any curve in $\calM^{J_\std}_{\CP^2,3\ell}\lll \T_{\calB_-}^{(8)}\db\rrr$ excluding $\calN$ itself is by definition $(8,1)$-well-placed, although $J_\std$ is not generic.

Let $\{J_t\}_{t \in [0,1]}$ be a generic family of compatible almost complex structures on $\CP^2$ which are integrable near $\db$ and fix $\calB_-$ (but not $\calN$), with $J_0 = J_\std$.
For $t \in (0,1]$ we have $\calM^{J_t}_{\CP^2,3\ell}\lll \T_{\calB_-}^{(8)}\db\rrr = \{C_t^1,C_t^2,C_t^3,C_t^4\}$, where $C_t^k$ is a family of curves which varies smoothly in $t \in (0,1]$ for $k=1,2,3,4$.
Then each $C_t^k$ converges to some limiting configuration 
$C_0^k \in \ovll{\calM}^{J_\std}_{\CP^2,3\ell}\lll \T_{\calB_-}^{(8)}\db \rrr$ as $t \ra 0$.
Here $\ovll{\calM}^{J_\std}_{\CP^2,3\ell}\lll \T_{\calB_-}^{(8)}\db \rrr$ 
denotes the subset of the standard stable map compactification $\ovl{\calM}_{\CP^2,3\ell}^{J_\std}\lll \db \rrr$ consisting of those configurations such that if the marked point lies on a ghost component then the nearby nonconstant components together ``remember'' the constraint $\lll \T_{\calB_-}^{(8)}\db\rrr$ (see \cite[Def. 2.2.1]{mcduff2021symplectic}).
In fact, since a line cannot satisfy $\lll \T^{(3)}_{\calB_-}\db\rrr$ and a conic cannot satisfy $\lll \T^{(6)}_{\calB_-}\db \rrr$ (due to the presence of the other branch $\calB_+$), we can easily rule out configurations with multiple components, i.e. we have
$\ovll{\calM}_{\CP^2,3\ell}^{J_\std}\lll \T^{(8)}_{\calB_-}\db \rrr = \calM_{\CP^2,3\ell}^{J_\std}\lll \T_{\calB_-}^{(8)}\db \rrr$
and hence $C_0^1,C_0^2,C_0^3,C_0^4 \in \calM_{\CP^2,3\ell}^{J_\std}\lll \T^{(8)}_{\calB_-}\db \rrr$.

It remains to show that at least one of $C_0^1,C_0^2,C_0^3,C_0^4$ is distinct from $\calN$.
To see this, suppose by contradiction that $C_0^k = C_0^{k'} = \calN$ for distinct $k,k' \in \{1,2,3,4\}$.
Then $C_t^k$ and $C_t^{k'}$ both approximate $\calN$ and in particular the transversely intersecting branches $\calB_-,\calB_+ \subset \calN$ for $t$ small, and since they also both satisfy the constraint $\lll \T_{\calB_-}^{(8)}\db\rrr$, their intersection multiplicity satisfies
\begin{align*}
C_t^k \cdot C_t^{k'} \geq 8 + 1 + 1 > 9,
\end{align*}
which is a contradiction.
\end{proof}

\begin{rmk} J. Koll\'ar supplied us with the following proof that
the curve $C$ in Proposition~\ref{prop:deg_3_seeds} (and hence all of its twists) is sesquicuspidal, i.e. it has a single ordinary double point away from its distinguished cusp.  Suppose to the contrary that $C$ is cuspidal. Then $\db$ must be a smooth point on $C$, since otherwise $C\cdot \calN \ge 10$, and, if $H$ denotes the hyperplane class, we have $9[\db] \sim 3H$  in ${\rm Pic}(C)$.  
Because $C$ is cuspidal rather than nodal, ${\rm Pic}(C)$ is torsion-free.  Hence $H\sim 3[\db]$, which implies that $\db$ is a flex point of $C$. If $L$ is the flex tangent, then $\calN$ must be in the linear system  $[C, 3L]$. But in suitable coordinates this consists of the curves $\la(y^2 z - x^3) + \mu z^3 = 0$, and it is easy to check that all of its elements apart from $C$ are smoothly cut out; in particular, they have no double points.  
\end{rmk}

\begin{rmk}\label{rmk:higher_deg_seeds}
We sometimes refer to a $(3d-1,1)$-well-placed rational plane curve as in Proposition~\ref{prop:deg_3_seeds} informally as a ``degree $d$ seed curve''. 
The above argument easily extends to show that the number of degree three seed curves is precisely $3$.
We take up the problem of constructing higher degree seed curves in the forthcoming work \cite{sesqui}. 
In general a sesquicuspidal degree $d$ seed curve and all of its Orevkov twists have $\tfrac{1}{2}(d-1)(d-2)$ double points by the adjunction formula.
Note that the argument given above does not easily generalize to higher degrees, due to the possibility of more complicated configurations in $\ovll{\calM}_{\CP^2,d\ell}^{J_\std}\lll \T^{(3d-1)}_{\calB_-}\db \rrr$ involving one or more copies of $\calN$ and its covers.
\end{rmk}

\subsection{On the stable folding curve}\label{subsec:folding_curve}

In \cite{hind2015some}, it is shown\footnote{Strictly speaking, the passage from $E(1,a,T,\dots,T)$ to $E(1,a) \times \C^N$ for $T$ arbitrarily large relies on results in \cite{Pelayo-Ngoc_hofer_question}.} that for any $a \in \R_{> 1}$ and $N \in \Z_{\geq 1}$ there is a folding-type symplectic embedding 
\begin{align}
\intE(1,a) \times \C^N \hooksymp B^4(\tfrac{3a}{a+1}) \times \C^N
\end{align}
 which strengthens the construction in \cite{Guth_polydisks}.
As discussed e.g. in \cite[\S1.2]{SDEP}, this embedding is conjectured to be sharp for all $a > \tau^4$, and proving this can be reduced to an existence problem for $(p,q)$-sesquicuspidal symplectic curves in $\CP^2$.
 In this subsection we discuss generalizations of this to rigid del Pezzo surfaces and to the convex toric domains considered in \cite{cristofaro2020infinite}.

In the following, given a symplectic manifold $M$, we denote by $c \cdot M$ the same smooth manifold but with symplectic form scaled by $c \in \R_{>0}$.
Recall that $E(1,a)$ denotes a closed ellipsoid, and we denote its interior by $\intE(1,a)$.

\begin{prop}\label{prop:folding1} For any $a \in \R_{>1}$, there are symplectic embeddings:
\begin{align}
\intE(1,a) \times \C \hooksymp \tfrac{a}{a+1} \cdot M \times \C
\end{align}
for $M = \CP^2(3) \#^{\times j} \ovl{\CP}^2(1)$ with $j = 0,1,2,3$ and $M = \CP^1(2) \times \CP^1(2)$.
\end{prop}

The paper \cite{cristofaro2020infinite} also establishes infinite staircases for the ellipsoid embedding functions of twelve convex toric domains $X_1,\dots,X_{12}$, whose moment polygons are pictured in Figure~\ref{fig:rational_staircase_targets}. 
More precisely, for each $j=1,\dots,12$ we put $X_j := \mu_{\C^2}^{-1}(\Om_j) \subset \C^2$, where $\mu_{\C^2}: \C^2 \ra \R_{\geq 0}^2, (z_1,z_2) \mapsto (\pi |z_1|^2,\pi |z_2|^2)$ is the moment map for the standard $\mathbb{T}^2$-action on $\C^2$.
Note that, in contrast to closed toric symplectic manifolds, these are compact domains with piecewise smooth boundary in $\C^2$.
As we recalled in Remark~\ref{rmk:dp_versus_ctd}, for $j=1,\dots,12$ the (unstabilized) ellipsoid embedding function $c_{X_j}$ coincides with $c_M$, where $M$ is the rigid del Pezzo surface having the same negative weight expansion as $X_j$.

Similar to Proposition~\ref{prop:folding1}, we have:
\begin{prop}\label{prop:folding2} For any $a \in \R_{>1}$, there are symplectic embeddings
\begin{align}
\intE(1,a) \times \C \hooksymp \tfrac{a}{a+1} \cdot X_k \times \C
\end{align}
 for $k = 1,2,3,4,5,6,8$.
\end{prop}
\begin{rmk}
In contrast to the four-dimensional (unstabilized) situation, it is not known whether the stabilized ellipsoid embedding functions of target spaces with the same negative weight expansions necessarily coincide (e.g. $c_{X_3 \times \C}$ versus $c_{X_4 \times \C}$ versus $C_{\CP^1(2) \times \CP^1(2) \times \C})$.
\end{rmk}
\begin{rmk} The polygons $\Om_1,\dots,\Om_{12}$ in Figure~\ref{fig:rational_staircase_targets} correspond to 
twelve of the famous sixteen reflexive polygons.
The remaining four reflexive polygons have the same negative weight expansion as $\CP^2(3) \#^{\times 5} \ovl{\CP}^2(1)$,
whose ellipsoid embedding function does not contain an infinite staircase (see \cite[Rmk. 5.21]{cristofaro2020infinite}).

It is interesting to note that $M = \CP^2(3) \#^{\times 5} \ovl{\CP}^2(1)$ admits an almost toric fibration $\pi: M \ra Q_\nodal$ where $Q$ is a dual Fano $T$-quadrilateral with a Delzant vertex (see \cite[Fig. 18($C_1$)]{vianna2017infinitely}).
In particular, most of the results in \S\ref{sec:singI} and \S\ref{sec:singII} still apply in this case.
Similarly, $M = \CP^2(3) \#^{\times 6} \ovl{\CP}^2(1)$ admits an almost toric fibration $\pi: M \ra Q_\nodal$ where $Q$ is a dual Fano $T$-pentagon with a Delzant vertex (see \cite[Fig. 19($A_1$)]{vianna2017infinitely})
\end{rmk}

\begin{figure}[h]
\caption{The convex toric domains considered in \cite{cristofaro2020infinite}, along with their negative weight expansions.}
\label{fig:rational_staircase_targets}
\centering
\includegraphics[width=1\textwidth]{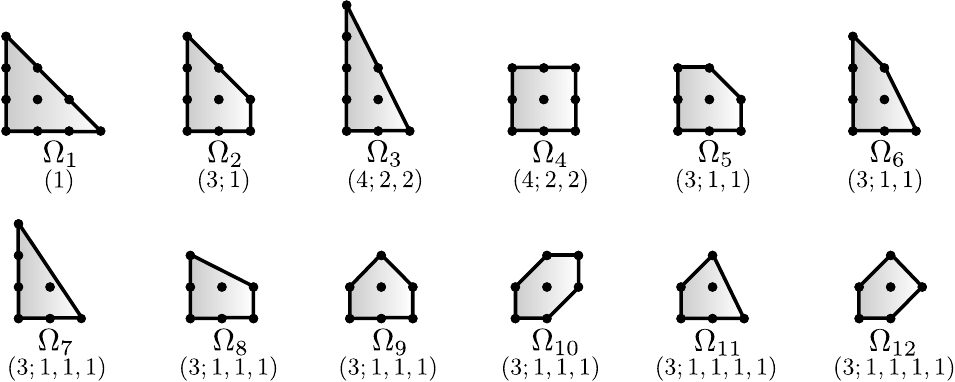}
\end{figure}

\begin{proof}[Proofs of Proposition~\ref{prop:folding1} and Proposition~\ref{prop:folding2}]
  We apply \cite[Prop. 3.1]{cristofaro2022higher}. Specializing to the case $\mu = \tfrac{a}{a+1}$, we get $\la = 1 - \tfrac{\mu}{a} = \tfrac{a}{a+1}$, so $\mu = \la$.
Note that $2\mu\tfrac{2\la-1}{\la+\mu-1} = 2\mu = \la + \mu$.
Then $f$ is the linear function satisfying $f(0) = 2\la = \tfrac{2a}{a+1}$ and $f(2a/(a+1)) = \la = \tfrac{a}{a+1}$. This means that there is a symplectic embedding of $(1-\delta) \cdot E(1,a)\times \C$ into $\tfrac{a}{a+1} \cdot X_{\Omega_H} \times \C$ for all $\delta > 0$, where $\Om_H \subset \R_{\geq 0}^2$ is the quadrilateral having vertices $(0,0),(0,2), (2,1), (2,0)$, and $X_{\Omega_H} \subset \C^2$ is the corresponding four-dimensional convex toric domain. Using \cite[Thm. 4.4]{Pelayo-Ngoc_hofer_question}, we can upgrade this to an embedding $\intE(1,a) \times \C \hooksymp \tfrac{a}{a+1} \cdot X_{\Om_H} \times \C$.

Inspecting Figure~\ref{fig:rational_staircase_targets}, we see that $\Omega_H$ (or its reflection about the diagonal) is a subset of $\Om_i$ for $i = 1,2,3,4,5,6,8$.
Similarly, the moment polygons corresponding to $M = \CP^2(3), \CP^2(3) \# \ovl{\CP}^2(1),\CP^1(2) \times \CP^1(2),\CP^2(3) \#^{\times 2} \ovl{\CP}^2(1)$ are $\Om_1,\Om_2,\Om_4,\Om_5$ respectively, each of which directly contains $\Om_H$ as a subset.

As for $\CP^2(3) \#^{\times 3} \ovl{\CP}^2(1)$, the moment polygon is (up to an integral affine transformation) given by $\Om_{10}$. This does not contain $\Om_H$ as a subset, but there is an almost toric fibration for $\CP^2(3) \#^{\times 3} \ovl{\CP}^2(1)$ whose base polygon is precisely $\Om_H$ -- see \cite[Fig. 16]{vianna2017infinitely}.
\end{proof}
\begin{rmk}
  Note that $\Omega_H$ has area $3$, as do the polygons $\Om_7,\Om_9,\Om_{10}$, while the polygons $\Om_{11},\Om_{12}$ have area $5/2$. In particular, by volume considerations there cannot be any four-dimensional embedding $\intX_{\Om_H} \hooksymp X_k$ for $k = 11,12$.
\end{rmk}

\sss

It is natural to ask what happens in the remaining cases not covered by Propositions \ref{prop:folding1} and \ref{prop:folding2}:
\begin{question}
 Is there a stabilized symplectic embedding $$
 \intE(1,a) \times \C \hooksymp \tfrac{a}{a+1} \cdot M \times \C\;\;\mbox{ where } \;\;
M= \CP^2(3) \#^{\times 4}\ovl{\CP}^2(1)?
$$ What about $\intE(1,a) \times \C \hooksymp \tfrac{a}{a+1} \cdot X_{k}
 $ for $k = 7,9,10,11,12$?
\end{question}
\NI Also, extending the aforementioned conjecture for stabilized embeddings of ellipsoids into the four-ball, we posit:
\begin{conjecture}\label{conj:folding_sharp}
The symplectic embeddings in Proposition~\ref{prop:folding1} and Proposition~\ref{prop:folding2} are all optimal for all $a > a_\acc$, where $a_\acc$ is the accumulation point of the corresponding staircase (c.f. Table~\ref{table:CG_et_al}).
\end{conjecture}
\NI As evidence, we observe that the obstruction coming from any index zero sesquicuspidal curves is consistent with this conjecture:

\begin{prop}\label{prop:ses_obs_on_folding_curve}
Let $(M,\omega_M)$ be a closed symplectic manifold with $[\omega_M] = c_1 \in H^2(M;\R)$, and let $C$ be an index zero $(p,q)$-sesquicuspidal rational symplectic curve in $M$. 
Then the corresponding obstruction for a symplectic embedding $E(1,p/q) \times \C^N \hooksymp \la \cdot M \times \C^N$ coming from Theorem~\ref{thm:stab_obs_from_curve} is $\la \geq \tfrac{(p/q)}{(p/q)+1}$.
\end{prop}
\begin{proof}
Since $C$ has index zero and  $c_1([C]) = p+q$,  the obstruction is 
\begin{align*}
\la \geq \tfrac{p}{[\omega_M] \cdot [C]} = \tfrac{p}{c_1([C])} = \tfrac{p}{p+q} = \tfrac{(p/q)}{(p/q)+1}.  
\end{align*}
\end{proof}

\begin{rmk}
According to \cite[Thm. G]{cusps_and_ellipsoids}, for any closed symplectic four-manifold $M$, the perfect exceptional homology classes in $H_2(M)$ are in bijective correspondence with index zero unicuspidal symplectic curves in $M$. For example, \cite{magillmcd2021,magill2022staircase} describes all perfect exceptional classes for the first Hirzebruch surface $\bl^1\CP^2$ (c.f. \S\ref{subsec:unicusp_appl}), and by Proposition~\ref{prop:ses_obs_on_folding_curve} these all give obstructions consistent with Conjecture~\ref{conj:folding_sharp}.
\end{rmk}

\renewbibmacro{in:}{}
\printbibliography
\end{document}